\documentclass[12pt]{amsart}
\usepackage[margin=1.2in]{geometry}
\usepackage[final]{microtype}
\usepackage{parskip} 
\usepackage{amsmath,amssymb,amsthm,amsfonts,array,mathtools}
\usepackage{mathrsfs,dsfont,stmaryrd,slashed,cancel,relsize}
\usepackage{mathbbol}

\usepackage[T1]{fontenc}
\usepackage{lmodern}

\usepackage{libertine}

\usepackage{bbm}
\DeclareSymbolFontAlphabet{\mathbb}{AMSb}
\DeclareSymbolFontAlphabet{\mathbbl}{bbold}
\usepackage{graphicx,psfrag}
\usepackage{tikz,tikz-cd,pgfplots}
\pgfplotsset{compat=1.18}
\usetikzlibrary{arrows,calc,decorations.markings,fadings,patterns,positioning,shapes}
\tikzset{>=latex}
\tikzstyle{mypoint}=[inner sep=0pt,outer sep=0pt,minimum size=5pt,fill,circle]
\usepackage{xcolor}
\usepackage{listings}
\definecolor{codegreen}{rgb}{0,0.6,0}
\definecolor{codegray}{rgb}{0.5,0.5,0.5}
\definecolor{codepurple}{rgb}{0.58,0,0.82}
\definecolor{backcolour}{rgb}{0.96,0.96,0.96}

\lstdefinestyle{mystyle}{
  backgroundcolor=\color{backcolour},
  basicstyle=\ttfamily\footnotesize,
  commentstyle=\color{codegreen},
  keywordstyle=\color{blue},
  numberstyle=\tiny\color{codegray},
  stringstyle=\color{codepurple},
  breaklines=true,
  numbers=left,
  numbersep=5pt,
  showspaces=false,
  showstringspaces=false,
  showtabs=false,
  tabsize=2,
  captionpos=b
}
\usepackage{fancyhdr}
\usepackage[framemethod=tikz]{mdframed}
\usepackage{pagecolor}

\usepackage{thmtools}

\declaretheorem[name=Definition,within=section]{definition}
\declaretheorem[name=Theorem,sibling=definition]{theorem}
\declaretheorem[name=Corollary,sibling=definition]{corollary}
\declaretheorem[name=Proposition,sibling=definition]{proposition}

\declaretheorem[name=Lemma,sibling=definition]{lemma}
\declaretheorem[name=Example,sibling=definition]{example}

\newtheorem{innercustomthm}{Theorem}
\newenvironment{customthm}[1]
  {\renewcommand\theinnercustomthm{#1}\innercustomthm}
  {\endinnercustomthm}

\theoremstyle{definition}
\declaretheorem[name=Remark,sibling=definition]{remark}

\usepackage{etoolbox}
\newcommand{\define}[4]{\expandafter#1\csname#3#4\endcsname{#2{#4}}}

\forcsvlist{\define{\DeclareMathOperator}{}{}}{im,coker,rad,nil,Ann,Ass,codim,Spec,mSpec,diam,ord,Supp,supp,disc,Ob,vol,rank,Sym,Alt,Cl,tr,spl,Proj,Pic,CaCl,Div,THH,TC,cone,Spf,Syn,colim,Spa,an,Nyg,HT,Fil}

\forcsvlist{\define{\newcommand}{\mathsf}{}}%
  {Set,Grp,Ab,CRing,Mod,Vect,Cat,Top,PreSh,Sh,Sch,Nat,%
   Fun,Diff,Alg,Rep,Coh,QCoh,Perf,Loc,Ind}

\forcsvlist{\define{\newcommand}{\mathrm}{}}
{R,D,M,I,E,B,X,Y,V,W,U}

\forcsvlist{\define{\newcommand}{\mathrm}{}}{Hom,Mor,id,GL,SL,PSL,PGL,SO,SU,Mat,Ext,Tor,Res,Cor,Inf,End,Irr,Aut,Gal,lcm,sign,triv,diag,Map,op,ev,act,alg,sep,unr,nr,ab,T,MU,MO,KU,un,dR,cris,qsyn,pris,qrsp,perf,gr,conj}

\forcsvlist{\define{\newcommand}{\mathbf}{}}{N,Z,Q,C,F,A}

\newcommand{\rA}{\mathrm{A}}
\newcommand{\rC}{\mathrm{C}}
\newcommand{\rS}{\mathrm{S}}

\newcommand{\rG}{\mathrm{G}}

\newcommand{\q}{\mathfrak{q}}

\newcommand{\m}{\mathfrak{m}}

\newcommand{\G}{\mathbf{G}}

\newcommand{\et}{\mathrm{\acute{e}t}}

\newcommand{\proet}{\mathrm{pro\acute{e}t}}
\newcommand{\heart}{\ensuremath\heartsuit}

\newcommand{\Prism}{{\mathbbl{\Delta}}}
\newcommand{\laur}[1]{(\mkern-2mu(#1)\mkern-2mu)}

\newenvironment{display}{\begin{center}\begin{tikzcd}}%
  {\end{tikzcd}\end{center}}

\usepackage[
  hypertexnames=false,
  colorlinks=true,
  linkcolor={red!50!black},
  citecolor={blue!50!black},
  urlcolor={blue!80!black}
]{hyperref}
\renewcommand{\O}{\ensuremath{\mathcal{O}}}
\renewcommand{\H}{\mathrm{H}}
\renewcommand{\L}{\mathrm{L}}
\renewcommand{\P}{\mathbf{P}}

\title{Syntomification and crystalline local systems}

\author{Dylan Pentland}
\email{dpentland@math.harvard.edu}

\begin{document}

\begin{abstract}
Let $p$ be a prime, and let $\X$ be a smooth $p$-adic formal scheme over $\Spf \O_K$ where $K/\Q_p$ is a finite extension. We show that reflexive sheaves on the stack $\X^{\Syn}$ are equivalent to $\Z_p$-lattices in crystalline local systems on the rigid generic fiber $\X_\eta$, and then use this to study the essential image of the \'{e}tale realization functor on the isogeny category of perfect complexes on $\X^{\Syn}$. We also show when $\X/\Spf \O_K$ is smooth and proper that $\Perf(\X^{\Syn})[1/p]$ is equivalent to a category of admissible filtered $F$-isocrystals in perfect complexes.
\end{abstract}

\maketitle

\tableofcontents

\setcounter{tocdepth}{1}

\section{Introduction}

Let $K/\Q_p$ be a finite extension with residue field $k$, and let $\Rep_{\Q_p}^\cris(\rG_K)$ denote the category of crystalline Galois representations in $\Q_p$-vector spaces of $\rG_K:=\Gal(\bar{K}/K)$. In \cite{KisinCris} Kisin gave a fully faithful functor
\[\D_{\mathfrak{S}}: \Rep_{\Z_p}^\cris(\rG_K) \to \Mod^{\varphi}_{\mathfrak{S}}\]
to the category $\Mod^{\varphi}_{\mathfrak{S}}$ of Breuil-Kisin modules over $\mathfrak{S}=\W(k)\llbracket u_0 \rrbracket$ and characterized the essential image.

A prismatic variant of this result was proven by Bhatt-Scholze. Let $\Vect^{\varphi}((\Spf \O_K)_{\Prism},\O_{\Prism})$ denote the category of prismatic $F$-crystals as defined in \cite{prismFcris}.

\begin{theorem}[Bhatt-Scholze, \cite{prismFcris}]
Let $K/\Q_p$ be a finite extension. Then there is a symmetric monoidal equivalence 
\[\T_\et: \Vect^{\varphi}((\Spf \O_K)_{\Prism},\O_{\Prism}) \to \Rep_{\Z_p}^\cris(\rG_K).\]
Moreover, the diagram
\begin{display}
\Vect^{\varphi}((\Spf \O_K)_{\Prism},\O_{\Prism}) \ar{rd}{\mathrm{ev}} \ar{rr}{\T_\et} \ar[swap]{rr}{\sim} & & \Rep_{\Z_p}^\cris(\rG_K) \ar{ld}{\D_{\mathfrak{S}}} \\
& \Mod^{\varphi}_{\mathfrak{S}} &
\end{display}
commutes, where $\mathrm{ev}$ is the evaluation functor on the Breuil-Kisin prism $(\mathfrak{S},\E(u_0))$.
\end{theorem}

Viewing this as a type of Riemann-Hilbert equivalence, it is natural to ask whether or not it admits a derived variant. In \cite{Fgauge}, Bhatt-Lurie define a stack $\X^{\Syn}$ associated to a $p$-adic formal scheme $\X$. Similar to previously constructed stacks such as $\X^{\Prism}$ in \cite{APC2}, a central idea is that $\R\Gamma(\X^{\Syn},\O\{i\})$ produces the syntomic cohomology $\R\Gamma_{\Syn}(\X,\Z_p(i))$.

Let $\X_\eta$ denote the rigid generic fiber of $\X$, and define $\D^{(b)}_{\mathrm{lisse}}(\X_\eta,\Z_p)$ to be the full subcategory of $\D_{\proet}(\X_\eta,\Z_p)$ consisting of locally bounded derived $p$-complete objects whose mod $p$ reduction has cohomology sheaves that are locally constant with finitely generated stalks. Bhatt-Lurie construct a $t$-exact \'{e}tale realization functor
\[\T_\et: \Perf(\X^{\Syn}) \to \D^{(b)}_{\mathrm{lisse}}(\X_\eta,\Z_p)\]
utilizing the equivalence in \cite[Corollary 3.7]{prismFcris}. 

In the case of a point, they extend the result of Bhatt-Scholze to the syntomic stack.

\begin{theorem}[Bhatt-Lurie, \cite{Fgauge}]
Let $\mathcal{E}\in \Perf(\Z_p^{\Syn})$. Then for all $i$, $\H^i(\T_\et(\mathcal{E}))[1/p]$ is crystalline.

Moreover, for a general finite extension $K/\Q_p$ there is a subcategory of reflexive sheaves $\mathsf{Refl}(\O_K^{\Syn})\subset \Coh(\O_K^{\Syn})$ on which $\T_\et$ induces an equivalence $\mathsf{Refl}(\O_K^{\Syn}) \simeq \Rep_{\Z_p}^\cris(\rG_K)$.
\label{thm:bhatt-lurie}
\end{theorem}

Restriction to the substack $\O_K^{\Prism}$ induces an equivalence 
\[\mathsf{Refl}(\O_K^{\Syn}) \simeq \Vect^{\varphi}((\Spf \O_K)_{\Prism},\O_{\Prism}).\]
The functor $\T_\et$ factors through this restriction, and so the equivalence is compatible with the result of Bhatt-Scholze (and hence also with \cite{KisinCris}). In \cite{Fgauge}, it is speculated that this result generalizes past a point in Remark 6.6.1. This is the first main result. In what follows, as in the notation section we assume $\X$ is smooth and quasicompact.

\begin{customthm}{A}[Theorem \ref{thm:generalsmoothTetcrys}]
For $\mathcal{E}\in \Perf(\X^{\Syn})$ we have 
\[\H^i(\T_\et(\mathcal{E}))[1/p]\in \Loc_{\Q_p}^\cris(\X_\eta)\]
for all $i$. Moreover, there is a subcategory of reflexive sheaves $\mathsf{Refl}(\X^{\Syn})\subset \Coh(\X^{\Syn})$ (see Definition \ref{def:reflexive}) on which $\T_\et$ induces an equivalence $\mathsf{Refl}(\X^{\Syn}) \simeq \Loc_{\Z_p}^\cris(\X_\eta)$.
\end{customthm}

The second assertion also holds without a quasicompactness assumption (which is only introduced to control isogeny categories). In \cite{Fgaugelift}, Guo-Li already partially answered this question. By combining this with the main result of Guo-Reinecke (\cite{Zpcrystalline}) they obtain a fully faithful functor
\[\Pi_\X: \Loc_{\Z_p}^\cris(\X_\eta) \to \Coh(\X^{\Syn})=\Perf(\X^{\Syn})^\heart\]
such that $\T_\et(\Pi_\X(\mathcal{L}))\simeq \mathcal{L}$. What remains is to characterize the essential image in a way analogous to the case of a point. We show the essential image of this functor is $\mathsf{Refl}(\X^{\Syn})$, the subcategory of reflexive objects in $\Coh(\X^{\Syn})$. This alternate description easily implies the first part of the theorem, as well as the following corollary deduced by understanding $\Coh(\X^{\Syn})[1/p]$. In what follows, the same remark about how we interpret isogeny categories applies when $\X$ is not quasicompact.

\begin{corollary}[Corollary \ref{cor:Perf_cris}]
The $t$-exact functor
\begin{display}
\Perf(\X^{\Syn})[1/p] \ar{r}{\T_\et[1/p]} & \D^{(b)}_{\mathrm{lisse}}(\X_\eta,\Z_p)[1/p]
\end{display}
induces an equivalence $\Coh(\X^{\Syn})[1/p]\simeq \Loc_{\Q_p}^\cris(\X_\eta)$ on the heart. The essential image of $\T_\et[1/p]$ contains the essential image of
\[\D^b(\Loc_{\Q_p}^\cris(\X_\eta)) \to \D^{(b)}_{\mathrm{lisse}}(\X_\eta,\Z_p)[1/p],\]
and is contained in the full subcategory of $\D^{(b)}_{\mathrm{lisse}}(\X_\eta,\Z_p)[1/p]$ where every cohomology sheaf is crystalline.
\end{corollary}

In fact, both inclusions are in general strict inclusions as explained in Remark \ref{rem:ess_im}. This corollary allows us to deduce a variant of the $\mathrm{C}_{\cris}$ conjecture later in Corollary \ref{cor:der_C_cris}. In future work, the author aims to understand the essential image of $\T_\et$ on coherent $F$-gauges without inverting $p$. 

It is desirable to have a more explicit description of the category $\Perf(\X^{\Syn})[1/p]$. In the case of $\X=\Spf \Z_p$, it is shown in \cite{Hauck} that
\[\Perf(\Z_p^{\Syn})[1/p] \simeq \D^b(\Rep_{\Q_p}^{\cris}(\rG_{\Q_p})).\]
One may also interpret this, using Colmez-Fontaine's equivalence $\Rep_{\Q_p}^\cris(\rG_{\Q_p}) \simeq \mathrm{MF}_{\Q_p}^{\varphi, \mathrm{wa}}$ in \cite{ColmezFontaine} between crystalline Galois representations and weakly admissible filtered $F$-isocrystals, as an equivalence $\Perf(\Z_p^{\Syn})[1/p] \simeq \D^b(\mathrm{MF}_{\Q_p}^{\varphi, \mathrm{wa}})$. This statement is what generalizes; it is the second main result.

\begin{customthm}{B}[Theorem \ref{thm:rational_equiv}]
Assume that $\X/\Spf\O_K$ is smooth proper. Then there is an equivalence of categories
\[\Perf(\X^{\Syn})[1/p]\simeq \Perf^{\mathrm{adm}}_{\mathrm{fIsoc}^{\varphi}}(\X).\]
Here $\Perf_{\mathrm{fIsoc}^{\varphi}}(\X)$ is as in Definition \ref{def:admfil_Fisoc}, and the superscript $\mathrm{adm}$ denotes the full subcategory of objects with admissible cohomologies (as in Definition \ref{def:fIsoc}).
\end{customthm}

While we initially define $\Perf_{\mathrm{fIsoc}^{\varphi}}(\X)$ using formal stacks, we show there is an explicit description of this category using $\D$-modules. One might be tempted to generalize Hauck's result to general $\X$ by trying to relate $\Perf(\X^{\Syn})[1/p]$ to $\D^b(\Loc_{\Q_p}^\cris(\X_\eta))$; we show that in Example \ref{ex:derived_counter} this already fails for $\P^1_{\Z_p}$ due to a failure of full faithfulness and we expect that it rarely works outside of the case of $\Spf \O_K$, where we confirm it indeed generalizes this way (Proposition \ref{prop:O_K_equiv}). As a positive result in this direction, in the category $\D^b(\Loc_{\Q_p}^\cris(\X_\eta))$ the $\R\Hom$ would be computed by crystalline extension groups. The following result shows this still holds in a limited sense.

\begin{proposition}[Proposition \ref{prop:ext1_cohom}]
Assume $\X$ is smooth and quasicompact over $\Spf \O_K$ and let $\mathcal{E}\in \Coh(\X^{\Syn})$. Then the \'{e}tale realization induces an isomorphism
\[\H^1_{\Syn}(\X,\mathcal{E})[1/p] \simeq \Ext^1_{\Loc_{\Q_p}^\cris(\X_\eta)}(\Q_p,\T_\et(\mathcal{E})[1/p]).\]
In particular, if $\X=\Spf \O_K$ then $\H^1_{\Syn}(\X,\mathcal{E})[1/p] \simeq \H^1_f(\rG_K,\T_\et(\mathcal{E})[1/p])$ where $\H^1_f$ denotes the Bloch-Kato Selmer group.
\end{proposition}

\textbf{Overview of the proofs}. To make sense of our results, we first need to set up some general theory about $t$-structures on $\Perf(\X^{\Syn})$ and coherent sheaves, which is done in \S 2.2. For Theorem A, we first show the second part of the theorem that $\mathsf{Refl}(\X^{\Syn})\simeq \Loc_{\Z_p}^\cris(\X_\eta)$. The first step in this argument is to analyze the kernel of the \'{e}tale realization in \S 3.1. We first show that coherent prismatic $F$-crystals in the kernel of the \'{e}tale realization are $p$-power torsion. The main difficulty in deducing the claimed kernel of the \'{e}tale realization in Theorem \ref{thm:general_Tet_kernel} is showing the same claim for $F$-gauges, namely that $F$-gauges in the kernel of $\T_\et$ are also $p$-power torsion. This is done by adapting the method of \cite{Fgauge} by constructing a graded connection on the Nygaard associated graded. We then show that Guo-Reinecke's category $\Vect^{\varphi,\mathrm{an}}(\X_{\Prism},\O_{\Prism})$ can be viewed as reflexive objects in coherent prismatic $F$-crystals (using \S 3.1 as input). Using Guo-Li's functor $\Pi_\X$ from \cite{Fgaugelift} to lift these to sheaves on $\X^{\Syn}$, we then use \S 3.1 again to show the output is reflexive (one must also verify coherence first, as their definition differs and does not use the stack). We can then deduce the first part of Theorem A from the second part.

For Theorem B, we use Theorem A to deduce Corollary \ref{cor:Perf_cris} which is used to describe the essential image in Theorem B (by using the statement about the heart). Full faithfulness is the main difficulty, which is shown using a variant of the Beilinson fiber square over $\O_K$ with coefficients (see \cite{BeilinsonFib}). The key input for this is a cohomology calculation on $\O_K^{\Syn}$ using the methodology of \cite{KtheoryZpn}. Similar squares appear in \cite{Hauck} and \cite{devalapurkar}, but restrict either to the case that $\X/\Spf \Z_p$ is smooth or work relative to the $q$-de Rham prism instead of the Breuil-Kisin prism, and so cannot be used to deduce Theorem B for smooth proper $\X/\Spf \O_K$.

\textbf{Notation}. Throughout, we fix a finite extension $K/\Q_p$ with residue field $k$ and ring of integers $\O_K$. We write $\X$ for a smooth and quasicompact $p$-adic formal scheme over $\Spf \O_K$, and $\X_\eta/\Spa K$ for its rigid generic fiber. We will adopt the conventions of \cite{Fgauge} for the various stacks $\X^{\Syn}, \X^{\Nyg}$, $\X^{\Prism}$, etc. We write $\X_{\Prism}$ for the absolute prismatic site of $\X$, and $\O_{\Prism}$ for its structure sheaf. We use $(-)^{\wedge}_{\I}$ to denote derived $\I$-adic completion (and similarly for $p$). We also use the slightly nonstandard term ``flat-local surjection'' in the sense of Definition \ref{def:flat_epi} to mean a map which is locally surjective in the fpqc topology; this notion suffices for the descent results we need. All categories are regarded as $(\infty,1)$-categories, although these will not play an essential role in the paper. We use cohomological shift conventions everywhere.

For a smooth $p$-adic formal scheme $\X/\Spf \O_K$, we set $\Loc_{\Z_p}^\cris(\X_\eta)$ to be the category of pro-\'{e}tale crystalline $\widehat{\Z}_p$-local systems as in \cite[Definition 2.31]{Zpcrystalline}. We write $\Loc_{\Q_p}^\cris(\X_\eta)$ for the isogeny category $\Loc_{\Z_p}^\cris(\X_\eta)[1/p]$ (so in particular they always admit a lattice), and in general use $(-)[1/p]$ when applied to a $\Z_p$-linear category to denote its isogeny category.

\textbf{Acknowledgements}. The author would like to thank Mark Kisin for his advice and encouragement, and Kush Singhal, Bhargav Bhatt, Maximilian Hauck, Sasha Petrov, Oakley Edens, and Daishi Kiyohara for useful conversations. We hope the intellectual debt to \cite{Fgauge} is clear to the reader. The author would like to thank Peter Scholze for pointing out an error in an earlier version of Corollary \ref{cor:Perf_cris}.

This work was supported by the NSF GRFP under Grant No. DGE 2140743.
\section{Preliminaries}

\subsection{Stacks}

In \cite{prismatization} Drinfeld introduced the formal stack\footnote{By a \textit{formal stack} $\X$ we will always mean an accessible sheaf $\R \mapsto \X(\R)$ for the flat (fpqc) topology on $p$-nilpotent rings valued in groupoids.} $\Z_p^{\Prism}$, and later in \cite{APC2} this was extended to the relative prismatization $(\X/\rA)^{\Prism}$ and the absolute prismatization $\X^{\Prism}$ for $p$-adic formal schemes $\X$. As we will only ever need $p$-adic formal schemes which are at worst quasisyntomic, we will make definitions in this generality.

We first recall the definition of the absolute prismatization, using the definition of a Cartier-Witt divisor (\cite[Definition 3.1.4]{APC}). The functor assigning a $p$-nilpotent ring $\R$ to the groupoid of Cartier-Witt divisors gives a stack $\Z_p^{\Prism}$, and the construction assigning a Cartier-Witt divisor $\alpha: \I \to \W(\R)$ to the derived quotient $\overline{\W(\R)}=\W(\R)/\I$ gives a stack $\G_a^{\Prism} \to \Z_p^{\Prism}$. Carrying out transmutation, one arrives at the following definition.

\begin{definition}[\cite{APC2}, Definition 3.1]
Let $\X$ be a quasisyntomic $p$-adic formal scheme. The absolute prismatization $\X^{\Prism}$ is the formal stack sending a $p$-nilpotent ring $\R$ to the groupoid of pairs $(\alpha: \I \to \W(\R), \eta: \Spec \overline{\W(\R)} \to \X)$ where $\alpha$ is a Cartier-Witt divisor.
\end{definition}

In this generality, as shown in \cite[Corollary 8.17]{APC2} we know that $\R\Gamma(\X^{\Prism},\O)$ computes the prismatic cohomology of $\X$. We note also that \cite[Proposition 8.15]{APC2} shows that $\D_{\mathrm{qc}}(\X^{\Prism})$ is the category of prismatic crystals in quasicoherent sheaves.

In \cite{Fgauge}, the additional stacks $\X^{\Nyg}$ and $\X^{\Syn}$ are defined, which have similar properties. Namely, following Drinfeld in \cite{prismatization} the notion of a \textit{filtered Cartier-Witt divisor} is introduced in \cite[Definition 5.3.1]{Fgauge}. Analogously, $\Z_p^{\Nyg}$ sends a $p$-nilpotent ring $\R$ to the groupoid of filtered Cartier-Witt divisors $d: \M \to \W$ over $\R$. The construction sending a filtered Cartier-Witt divisor $d: \M \to \W$ to
\[(\W/\M)(\R):= \R\Gamma(\Spec \R, \W/\M)\]
produces an animated $\W$-algebra stack $\G_a^{\Nyg} \to \Z_p^{\Nyg}$. Transmutation provides the following general definition. 

\begin{definition}[\cite{Fgauge}, Definition 5.3.10]
Let $\X$ be a quasisyntomic $p$-adic formal scheme. The absolute Nygaard-filtered prismatization $\X^{\Nyg}$ is the formal stack sending a $p$-nilpotent ring $\R$ to the groupoid of pairs $(d: \M \to \W, \eta: \Spec (\W/\M)(\R)\to \X)$ where $d$ is a filtered Cartier-Witt divisor.
\end{definition}

The stack $\X^{\Nyg}$ comes with two maps
\[j_{\dR},j_{\HT}: \X^{\Prism} \to \X^{\Nyg}\]
defined via transmutation from \cite[Constructions 5.3.5 and 5.3.2]{Fgauge}.

Finally, we will primarily be dealing with the stack $\X^{\Syn}$. In \cite[Construction 7.4.1]{APC}, the definition of syntomic cohomology $\R\Gamma_{\Syn}(\X,\Z_p(i))$ is given as the fiber
\[\varphi\{i\}-\mathrm{can}: \mathrm{Fil}^i_{\Nyg} \Prism_\X\{i\} \to \Prism_\X\{i\}.\]
Here, $\varphi\{i\}$ is the $i$th divided Frobenius. This construction can be realized in the stacky perspective by defining $\X^{\Syn}$ as the pushout in formal stacks
\begin{display}
\X^{\Prism} \sqcup \X^{\Prism} \ar{d} \ar{r}{j_{\dR} \sqcup j_{\HT}} & \X^{\Nyg} \ar{d}{j_{\Nyg}} \\
\X^{\Prism} \ar{r}{j_{\Prism}} & \X^{\Syn}
\end{display}
as in \cite[Definition 6.1.1]{Fgauge}. This pushout is in particular a coequalizer. We will use several times that the map $\X^{\Nyg} \to \X^{\Syn}$ is \'{e}tale.

The map $j_{\dR}$ corresponds to $\mathrm{can}$, while $j_{\HT}$ corresponds to $\varphi$. One may deduce that
\[\D_{\mathrm{qc}}(\X^{\Syn}) \simeq \begin{tikzcd} \mathrm{eq}(\D_{\mathrm{qc}}(\X^{\Nyg}) \ar[shift left]{r}{j^*_{\dR}} \ar[shift right,swap]{r}{j^*_{\HT}} & \D_{\mathrm{qc}}(\X^{\Prism}))\end{tikzcd},\]
and in particular on cohomology we get a natural exact triangle
\begin{display}
\R\Gamma(\X^{\Syn},\mathcal{E}) \ar{r} & \R\Gamma(\X^{\Nyg},j^*_{\Nyg} \mathcal{E}) \ar{r}{j_{\HT}^*-j_{\dR}^*} & \R\Gamma(\X^{\Prism},j^*_{\Prism} \mathcal{E})
\end{display}
where $j_{\Nyg}$ and $j_{\Prism}$ are as in the pushout square defining $\X^{\Syn}$. Using \cite[Example 6.1.8]{Fgauge}, putting $\mathcal{E}=\O\{i\}$ we recover the definition in \cite{APC}.

As pointed out in Remark 6.3.4 in \cite{Fgauge}, the category $\Perf(\X^{\Syn})$ admits a restriction functor yielding perfect prismatic $F$-crystals. We use $\Perf^{\varphi}(\X_{\Prism},\O_{\Prism})$ to refer to prismatic $F$-crystals in perfect complexes as defined in \cite{Fgaugelift} \S 3.1. We will often use the following observation from \cite{Fgauge}.

\begin{lemma}[\cite{Fgauge}, Remark 6.3.4]
Restriction to $\X^{\Prism}$ gives a functor
\[\Perf(\X^{\Syn}) \to \Perf^{\varphi}(\X_{\Prism},\O_{\Prism}).\]
\label{lem:restriction_Fcris}
\end{lemma}

This will be useful for us since the construction of the \'{e}tale realization factors through this restriction functor, so many arguments may be reduced to arguments about prismatic $F$-crystals.

\subsection{Coherent sheaves}

We will also need a theory of coherent sheaves for $\X^{\Syn}$ when $\X/\Spf \O_K$ is smooth. One possible definition is given in \cite{Fgaugelift}, but it is likely not the same as the natural category to associate using the stacky approach to $F$-gauges. In this subsection we will use an approach which is closer to how \cite{Fgaugelift} treats coherent prismatic $F$-crystals but adapted to $\X^{\Nyg}$ and $\X^{\Syn}$.

We will first show when $\X/\Spf \O_K$ is smooth that $\Perf(\X^{\Syn})$ comes equipped with a natural $t$-structure. To do so, we will need to use flat-local surjections. However, we note that we do not mean a flat covering of stacks. Rather, we use the following slightly weaker notion which suffices for descent.

\begin{definition}
Let $\X \to \Y$ be a map of $p$-adic formal stacks in the flat topology. Then we say $\X \to \Y$ is a \textit{flat-local surjection} if it is surjective locally in the flat topology.
\label{def:flat_epi}
\end{definition}

This differs from the usual notion of a \textit{flat covering of $p$-adic formal stacks}, which would ask that the map is representable and the pullback along $\Spec \R \to \Y$ for a $p$-nilpotent ring $\R$ induces a flat covering of $\Spec \R$. We will distinguish this second stronger notion by calling the map a flat cover of formal stacks.

This notion is sufficient for most results. Indeed, it suffices if we want to deduce descent for $\D_{\mathrm{qc}}(-)$, $\Perf(-)$, and $\Vect(-)$. This property is sufficient to deduce that $\X \to \Y$ is an effective epimorphism in the topos of flat sheaves, which means $\Y \simeq \colim \X^{\times_\Y(-)}$ where $\X^{\times_\Y (-)}$ are the terms of the \v{C}ech nerve. See \cite{Dirac} Appendix B for a proof of this result. Since $\D_{\mathrm{qc}}(-)$, $\Perf(-)$ and $\Vect(-)$ are sheaves for the flat topology, the desired claim follows.

\begin{lemma}[\cite{APC2}, Lemma 6.3]
Let $\X \to \Y$ be a quasisyntomic cover of quasisyntomic $p$-adic formal schemes. Then $\X^{\Prism} \to \Y^{\Prism}$ is surjective locally in the flat topology.
\end{lemma}

\begin{remark}
We will later show something slightly stronger in Proposition \ref{prop:prism_preserve_flat_covers}, but this will not be needed for any arguments.
\end{remark}

To define a $t$-structure on $\Perf(\X^{\Syn})$, we will pick Rees stacks as flat-local surjections and show that the $t$-structure is independent of the choice we made.

\begin{definition}
Let $(\rA,\I)$ be a prism. The completed Rees stack $\mathrm{Rees}_{\I^\bullet}(\rA)$ is defined as the stack
\[\Spf (\bigoplus_{i\in \Z} (\mathrm{Fil}^i_{\I^\bullet} \rA) t^{-i})_{(p,\I)}^{\wedge}/\G_m\]
where $t$ is given degree 1. We endow the graded ring $(\bigoplus_{i\in \Z} (\mathrm{Fil}^i_{\I^\bullet} \rA) t^{-i})^{\wedge}_{(p,\I)}$ with the $(p,\I)$-adic topology.
\end{definition}

When $\I=(d)$, we have
\[\mathrm{Rees}_{\I^\bullet}(\rA) \simeq (\Spf \rA\langle u,t \rangle/(ut-d))/\G_m\]
where $\deg(t)=1, \deg(u)=-1$.

Analogously to how \cite{Fgaugelift} defines coherent prismatic $F$-crystals, we use Breuil-Kisin prisms to construct coherent $F$-gauges.

\begin{proposition}[\cite{Fgaugelift} Proposition 3.7, Corollary 3.6]
Let $\X/\Spf \O_K$ be a smooth affine formal scheme. Let $\tilde{\R}$ be a choice of smooth lift of the special fiber $\X_s/\Spec k$ to $\W(k)$ and $\E(u_0)$ an Eisenstein polynomial for a uniformizer $\pi\in \O_K$. Then there is a prism
\[(\rA,\I) = (\tilde{\R}\llbracket u_0\rrbracket,(\E(u_0)))\]
such that $\Spf \rA/\I \simeq \X$, and $(\rA,\I)$ covers the final object $\ast \in \mathsf{Sh}(\X_{\Prism})$. Moreover the Frobenius on $\rA$ is finite, quasisyntomic, and faithfully flat.
\label{prop:BK_prism}
\end{proposition}

Given such a prism $(\rA,\I)$, there is a natural cover we can produce. In what follows, $\rA^{(1)}$ is the Frobenius twist $\rA \otimes_{\rA,\varphi} \rA$ with the relative Frobenius sending $a \otimes b \mapsto \varphi(a)b$ (i.e. the right factor is used for the $\rA$-module structure). We also use the symbol $(1)$ to denote the Frobenius twist of prismatic cohomology when applicable, as well as $\varphi^*$ for clarity in some situations.

\begin{proposition}
Assume that $\X$ is a smooth affine formal scheme over $\Spf \O_K$, and $(\rA,\I)$ is a Breuil-Kisin prism such that $\X\simeq \Spf(\rA/\I)$. Then
\[\rho_{\rA}: \mathrm{Rees}_{\I^\bullet}(\rA) \to \X^{\Nyg}\]
is a flat-local surjection. There is also a flat-local surjection
\[\rho_{\rA}': \mathrm{Rees}_{\Fil^\bullet_{\Nyg}}(\rA^{(1)}) \to \X^{\Nyg}\]
where the Nygaard filtration is defined as $\Fil^i_{\Nyg}=\{x\in \rA^{(1)}: \varphi(x) \in \I^i \rA\}$.
\label{prop:BK_flatcover}
\end{proposition}

\begin{proof}
This is \cite[Remark 5.5.19]{Fgauge} for the first item. For the second item there is a map
\[\varphi: \mathrm{Rees}_{\I^\bullet}(\rA) \to \mathrm{Rees}_{\Fil^\bullet_{\Nyg}}(\rA^{(1)})\]
induced by $\varphi: \rA^{(1)} \to \rA$ (which is faithfully flat in this case). In fact the Rees algebra in the target is simply the base change $\rA \otimes_{\rA,\varphi} \rA\langle u,t \rangle/(ut-d)$ if $\I=(d)$. The map $\rho_{\rA}$ will then factor as $\rho_{\rA}' \circ \varphi$ -- in the quasiregular semiperfectoid case this is easy to see (by virtue of the computation of $\R^{\Nyg}$ for $\R$ quasiregular semiperfectoid in \cite{Fgauge} as $\mathrm{Rees}_{\Fil^\bullet_{\Nyg}} \Prism_\R$, where the Nygaard filtration on $\Prism_\R$ has $\Fil^i = \{x\in \Prism_\R: \varphi(x)\in \I^i \Prism_\R\}$ where $\I$ is the prismatic ideal of $\Prism_\R$), and we can deduce the general factorization from this by quasisyntomic descent and computing \v{C}ech nerves.
\end{proof}

In general, one may use the construction of \cite[Remark 5.5.19]{Fgauge} to give a map $\rho_{\rA}: \mathrm{Rees}_{\I^\bullet}(\rA) \to \X^{\Nyg}$ when $\X$ is over $\Spf(\rA/\I)$ (with $\I$ principal). We may perform the same factorization to also produce $\rho_{\rA}': \mathrm{Rees}_{\Fil^\bullet_{\Nyg}}(\rA^{(1)}) \to \X^{\Nyg}$, and alternatively one may also write down an explicit Cartier-Witt divisor as in \cite[Construction 1.1.6]{LahotiManam}.

We may produce relative Nygaard stacks for $\X/\Spf (\rA/\I)$ for a prism $(\rA,\I)$ with $\I$ principal, which we define via the Cartesian square
\begin{display}
(\X/\rA)^{\Nyg} \ar{d}{\pi} \ar{r} & \X^{\Nyg} \ar{d} \\
\mathrm{Rees}_{\Fil^\bullet_{\Nyg}} \rA^{(1)} \ar{r}{\rho_{\rA}'} & (\rA/\I)^{\Nyg}
\end{display}
for a prism $(\rA,\I)$. This is generalized to work relative to $\delta$-pairs in \cite[Remark 1.1.10]{LahotiManam}. The cohomology of this stack computes Nygaard filtered relative prismatic cohomology via applying $\R \pi_*$.

When $\X=\Spf \rA/\I=\Spf \bar{\rA}$, the syntomic stack of $\X$ relative to $\rA$ appears as a pushout
\begin{display}
\Spf \rA^{(1)} \sqcup \Spf \rA^{(1)} \ar{d} \ar{r}{j_{\dR} \sqcup j_{\HT}} & \mathrm{Rees}_{\Fil^\bullet_{\Nyg}}(\rA^{(1)}) \ar{d}{j_{\Nyg}} \\
\Spf \rA^{(1)} \ar{r}{j_{\Prism}} & (\bar{\rA}/\rA)^{\Syn}
\end{display}
The map $j_{\HT}$ is given by $\Spf \rA^{(1)} \to \Spf \rA \simeq \Spf \left(\bigoplus_{i\in \Z} \I^i t^{-i}\right)/\G_m \to \mathrm{Rees}_{\Fil^\bullet_{\Nyg}}(\rA^{(1)})$, where the first map is induced by extension of scalars, the second map is induced by the filtered Frobenius. The map $j_{\dR}$ is the inclusion of the $t\neq 0$ locus. This pins down the general definition of the relative syntomic stack as the pushout of formal stacks
\begin{display}
\varphi^* (\X/\rA)^{\Prism} \sqcup \varphi^*(\X/\rA)^{\Prism} \ar{d} \ar{r}{j_{\dR} \sqcup j_{\HT}} & (\X/\rA)^{\Nyg} \ar{d}{j_{\Nyg}} \\
\varphi^* (\X/\rA)^{\Prism} \ar{r}{j_{\Prism}} & (\X/\rA)^{\Syn}
\end{display}
where the morphisms $j_{\dR}, j_{\HT}$ of relative stacks are induced by the morphisms of absolute stacks and the morphisms from the previous square. Here, we use $\varphi^*$ to mean that as a stack over $\Spf \rA$ we perform base change along $\Spf \rA^{(1)} \to \Spf \rA$. This construction is a stacky variant of the relative syntomic cohomology in \cite{PrismaticDelta}.

\begin{lemma}
Suppose that $\X=\Spf \R$ and $\Prism_{\X/\rA}$ is discrete. Then we have a canonical identification
\[(\X/\rA)^{\Nyg} \simeq \mathrm{Rees}_{\Fil^\bullet_{\Nyg}} \Prism^{(1)}_{\X/\rA}.\]
\label{lem:relqrsp_nygstack}
\end{lemma}

\begin{proof}
This is essentially a simpler version of an argument identical to \cite[Theorem 5.5.10]{Fgauge}, so we will omit some details in what follows. Since the prismatic cohomology is discrete one may produce a map of stacks
\[\eta: (\X/\rA)^{\Nyg} \to \mathrm{Rees}_{\Fil^\bullet_{\Nyg}} \Prism^{(1)}_{\X/\rA},\]
exactly as in step 3 of the proof. We have a structure map $\pi: (\X/\rA)^{\Nyg} \to  \mathrm{Rees}_{\Fil^\bullet_{\Nyg}} \rA^{(1)}$. Using the description of cohomology on this stack (again by a similar argument as in Step 2 of \cite[Theorem 5.5.10]{Fgauge}) we see for $\X/\Spf(\rA/\I)$ smooth that $\R\pi_* \O_{(\X/\rA)^{\Nyg}}$ is identified with Nygaard filtered prismatic cohomology under the Rees dictionary as a sheaf on $\mathrm{Rees}_{\Fil^\bullet_{\Nyg}} \rA^{(1)}$. After Kan extension to the general case of a bounded $p$-adic formal scheme, we get a map $\mathrm{Rees}( \Fil^\bullet_{\Nyg} \Prism_{\X/\rA}^{(1)}) \to \R\pi_* \O_{(\X/\rA)^{\Nyg}}$ (as commutative algebras in $\D_{\mathrm{qc}}(\mathrm{Rees}_{\Fil^\bullet_{\Nyg}} \rA^{(1)})$). Specializing to our $\X$ the prismatic cohomology is discrete, and then by adjunction we get $\eta$.

Then stratifying the target by the $\{t \neq 0\}$, $\{t=0, u\neq 0\}$, and $\{t=u=0\}$ loci we may use \cite[Theorem 7.17]{APC2} to show this map is an isomorphism on each stratum, hence an isomorphism via the same method as \cite[Theorem 5.5.10]{Fgauge}. Note that the final identification for the Hodge stack is not automatic from this result for the prismatization, but still follows from a similar argument.
\end{proof}

These relative stacks allow us to give a description of $\O_K^{\Nyg}$ analogous to the description via quasisyntomic descent from quasiregular semiperfectoid rings, but importantly where the first Rees stack in the equivalence is Noetherian regular (see the footnote of \cite[Remark 5.5.19]{Fgauge}; in general all the Rees charts we use are regular by the same argument). We will use this later in \S 4 to work with coefficients in the relative setting easily.

\begin{corollary}
Let $(\rA,\I)=(\W(k)\llbracket u_0\rrbracket,(\E(u_0)))$ be a Breuil-Kisin prism for $\O_K$. Then there is an equivalence of stacks
\[\O_K^{\Nyg} \simeq \colim_{[n]\in \Delta^\op} (\O_K/\W(k)\llbracket u_0,\ldots,u_{n}\rrbracket)^{\Nyg}.\]

In particular, we get an equivalence of categories
\[\Perf(\O_K^{\Nyg}) \simeq \lim_{[n]\in \Delta} \Perf(\mathrm{Rees}_{\Fil^\bullet_{\Nyg}}(\rA^{(n)}))\]
where 
\[\rA^{(n)} = \Prism^{(1)}_{\O_K/\W(k)\llbracket u_0,\ldots,u_{n}\rrbracket} = \varphi^* \W(k)\llbracket u_0,\ldots,u_{n}\rrbracket \left\{\frac{u_1-u_0}{\E(u_0)}, \ldots, \frac{u_{n}-u_0}{\E(u_0)}\right\}^{\wedge}\]
with the Nygaard filtration coming from prismatic cohomology. The maps in the limit are induced by $u_i \mapsto u_j$ (with the obvious face and degeneracy maps), and the prism structure on $\W(k)\llbracket u_0, \ldots, u_{n}\rrbracket$ is given by the ideal $(\E(u_0))$.
\label{cor:Nyg_relqrsp_desc}
\end{corollary}

\begin{proof}
Observe that as defined in the statement of the corollary we have
\[\mathrm{Rees}_{\Fil^\bullet_{\Nyg}}(\rA^{(n)}) \simeq (\O_K/\W(k)\llbracket u_0,\ldots,u_{n}\rrbracket)^{\Nyg},\]
using Lemma \ref{lem:relqrsp_nygstack} and \cite[Theorem 9.6]{PrismaticDelta} (due to Bhatt-Scholze), so all claims follow from showing that as stacks we have
\[\O_K^{\Nyg} \simeq \colim_{[n]\in \Delta^\op} (\O_K/\W(k)\llbracket u_0,\ldots,u_{n}\rrbracket)^{\Nyg}.\]
By Lemma \ref{lem:relqrsp_nygstack} and \cite{PrismaticDelta} we see that the right hand side can be constructed by using Nygaard filtered prismatic cohomology extended to $\delta$-pairs; it consists of the Rees stacks corresponding to the \v{C}ech nerve for Nygaard filtered prismatic cohomology in Theorem 1.2(6). In what follows, we also use Theorem 1.2(3) to identify Nygaard filtered relative prismatic cohomology over $\W(k)$ with absolute Nygaard filtered prismatic cohomology.

For the perfectoid cover $\mathrm{S}=\rA_{\perf}/\I$ we have a commutative diagram
\begin{display}
\colim_{[n]\in \Delta^\op} (\mathrm{S}^{(n)})^{\Nyg}  \ar{r} &  \colim_{[n]\in \Delta^\op} \colim_{[m]\in \Delta^\op} (\Spf \mathrm{S}^{(n)}/\W(k)\llbracket u_0,\ldots,u_{m}\rrbracket)^{\Nyg} \\ 
\O_K^{\Nyg} \ar{u}{\sim} \ar{r} & \colim_{[m]\in \Delta^\op} (\O_K/\W(k)\llbracket u_0,\ldots,u_{m}\rrbracket)^{\Nyg} \ar{u}{\sim}
\end{display}
where $\Spf \mathrm{S}^{(n)}$ is the $n$th term in the completed \v{C}ech nerve of $\Spf \mathrm{S} \to \Spf \O_K$. The vertical maps are equivalences since the Nygaard stack as well as $(-/\rA)^{\Nyg}$ send quasisyntomic covers to flat-local surjections (see for example Corollary \ref{cor:Nyg_preserve_flat_covers} for a stronger version) and are compatible with Tor-independent limits. By the first property we may then deduce $\mathrm{S}^{\Nyg} \to \O_K^{\Nyg}$ is a flat-local surjection and by the second property we may identify terms of the \v{C}ech nerve with $(\mathrm{S}^{(n)})^{\Nyg}$ giving the left vertical equivalence. The right equivalence is similar.

It then suffices to verify the top horizontal map is an equivalence. For a quasiregular semiperfectoid $\W(k)\llbracket u_0\rrbracket$-algebra $\mathrm{S}^{(n)}$, we have by Theorem 1.2(6) and (3) in \cite{PrismaticDelta} that
\[\Fil^\bullet_{\Nyg} \Prism_{\mathrm{S}^{(n)}} \simeq \Fil^\bullet_{\Nyg} \Prism^{(1)}_{\mathrm{S}^{(n)}/\W(k)} \simeq \lim_{[m]\in \Delta} \Fil^\bullet_{\Nyg} \Prism^{(1)}_{\mathrm{S}^{(n)}/\W(k)\llbracket u_0,\ldots,u_{m}\rrbracket}\]
as $p$-complete filtered rings. Fortunately, it turns out that $\Prism^{(1)}_{\mathrm{S}^{(n)}/\W(k)\llbracket u_0,\ldots,u_{m}\rrbracket}$ is actually still discrete, so the claim follows from Lemma \ref{lem:relqrsp_nygstack} and the Rees equivalence. This follows by a cotangent complex computation. Via the transitivity triangle we get
\[\L_{\O_K/\O_K\llbracket v_1,\ldots, v_m\rrbracket} \widehat{\otimes}_{\O_K} \mathrm{S}^{(n)} \to \L_{\mathrm{S}^{(n)}/\O_K\llbracket v_1,\ldots, v_m\rrbracket} \to \L_{\mathrm{S}^{(n)}/\O_K}\] 
and thus the claim follows by checking the $p$-complete Tor amplitude of the leftmost and rightmost terms (see \cite[Remark 4.21]{BMS2} for the right term). Here $\O_K\llbracket v_1,\ldots, v_m\rrbracket := \W(k)\llbracket u_0,\ldots,u_{m}\rrbracket/\E(u_0)$ with coordinates $v_i:=u_i-u_0$. Thus applying the Hodge-Tate comparison the desired discreteness claim follows.
\end{proof}

\begin{remark}
One may check the ring maps $\rA \to \rA^{(n)}$ are flat (rather than just $(p,\I)$-completely flat), since the source is Noetherian and both source and target are derived $(p,\I)$-complete (\cite[Lemma 5.15]{Bhatt}). Note also we crucially use that $\W(k)$ is a perfect $\delta$-ring, so this construction does not generalize.
\label{rem:flat}
\end{remark}

We now move to defining a $t$-structure on $\Perf(\X^{\Syn})$. Analogously to how $\Coh(\O_K^{\Syn})$ is defined in \cite{Fgauge}, we can make the following construction.

\textbf{Construction}. Choose a covering of $\X$ by smooth affines $\U_i$, and for each affine choose a Breuil-Kisin prism $(\rA_i,\I_i)$ for $\U_i$.

Then there is a flat-local surjection
\[\bigsqcup_i \mathrm{Rees}_{\I_i^\bullet} \rA_i \to \X^{\Nyg} \to \X^{\Syn}.\]
This covering is a disjoint union of quotients of Noetherian formal schemes by $\G_m$, and thus we can define $\Coh(\coprod_i \mathrm{Rees}_{\I_i^\bullet} \rA_i)$ in the usual way. We induce a $t$-structure on $\Perf(\X^{\Syn})$ from this covering by defining the $\le 0$ and $\ge 0$ parts to be perfect complexes which pull back to the $\le 0$ and $\ge 0$ parts of the natural $t$-structure on $\Perf(\coprod_i \mathrm{Rees}_{\I_i^\bullet} \rA_i)$.

It is possible to check manually that this $t$-structure does not depend on the choice of a Breuil-Kisin prism using an argument similar to \cite{Fgaugelift}, making it canonical. Let $(\rA,\I)$ and $(\rA',\I')$ be two different choices of Breuil-Kisin prisms for $\X$, and then take $(\B,\mathrm{J})$ to be the product in the absolute prismatic site of $\X$. Then using $\mathrm{Rees}_{\Fil_{\Nyg}^\bullet} \B^{(1)}$ we obtain a flat-local surjection which refines both coverings using Breuil-Kisin prisms. We can then argue both Breuil-Kisin prisms give the same $t$-structure entirely analogously to \cite[Proposition 3.11]{Fgaugelift}.

When working with formal stacks, we can also do this by instead giving a construction of a $t$-structure which makes no choices and then checking it agrees with the previous construction. There is a canonical $t$-structure on $\D_{\mathrm{qc}}(\X^{\Syn})$ for any $\X$, by defining $\D^{\le 0}_{\mathrm{qc}}(\X^{\Syn})$ to be sheaves $\mathcal{E}$ such that for every $p$-nilpotent test object $f:\Spec \R \to \X^{\Syn}$ we have $f^* \mathcal{E}\in \D^{\le 0}(\R)$. We define $\D^{\ge 0}$ via orthogonality: that is, we say $\mathcal{E}\in \D_{\mathrm{qc}}^{\ge 0}(\X^{\Syn})$ when $\Map(\mathcal{E}',\mathcal{E})\simeq *$ for all $\mathcal{E}'\in \D_{\mathrm{qc}}^{\le -1}(\X^{\Syn})$ (defined analogously). This definition works for $\D_{\mathrm{qc}}(\Y)$ on any formal stack $\Y$, and when $\Y$ is $\mathrm{Rees}_{\Fil^{\bullet}_{\Nyg}}(\rA^{(1)})$ or $\mathrm{Rees}_{\I^\bullet} \rA$ for a Noetherian regular prism $\rA$ we see restricting to $\Perf$ gives the standard $t$-structure on $\Perf$. For the assertion about $\Perf$, it is crucial that we are in a Noetherian regular situation, otherwise the $t$-structure on quasicoherent sheaves might fail to restrict to $\Perf$.

\begin{remark}
Pullback to the cover $\mathrm{Rees}_{\I^\bullet} \rA$ for a Breuil-Kisin prism such that $\X\simeq \Spf \rA/\I$ is $t$-exact for the canonical $t$-structure on $\D_{\mathrm{qc}}(\X^{\Syn})$ by Footnote 70 in \cite[Remark 5.5.19]{Fgauge} (combined with $\X^{\Nyg} \to \X^{\Syn}$ being an \'{e}tale cover).
\label{rem:t-exact}
\end{remark}

\begin{proposition}
Let $\X/\Spf \O_K$ be smooth. The $t$-structure on $\Perf(\X^{\Syn})$ does not depend on the choice of Breuil-Kisin prisms or affine covering.
\label{prop:t_struct_XSyn}
\end{proposition}

\begin{proof}
A priori trying to restrict the $t$-structure on $\D_{\mathrm{qc}}(\X^{\Syn})$ to $\Perf(\X^{\Syn})$ might not define a $t$-structure since truncations or cotruncations might fail to be perfect. However perfectness can be tested after pullback along flat-local surjections by descent, so we can test this after pullback to the cover $\sqcup_i \mathrm{Rees}_{\I_i^\bullet} \rA_i$ for a covering family of Breuil-Kisin prisms after noting this pullback is $t$-exact by Remark \ref{rem:t-exact}. But as $\rA_i$ and the associated Rees stacks are Noetherian regular, truncations/cotruncations of perfect complexes are again perfect so we conclude that we get a canonical $t$-structure on $\Perf(\X^{\Syn})$. Thus this gives a well-defined $t$-structure. Moreover, the argument also shows it agrees with our previous construction.
\end{proof}

We can then call this the canonical $t$-structure on $\Perf(\X^{\Syn})$.

\begin{definition}
For the canonical $t$-structure on $\Perf(\X^{\Syn})$, define $\Coh(\X^{\Syn}):= \Perf(\X^{\Syn})^{\heart}$.
\end{definition}

We will also frequently use that this $t$-structure is compatible with \'{e}tale realization.

\begin{lemma}
The restriction functor
\[(-)|_{\X^{\Prism}}: \Perf(\X^{\Syn}) \to \Perf^{\varphi}(\X_{\Prism},\O_{\Prism})\]
is $t$-exact for the canonical $t$-structure on perfect prismatic $F$-crystals in \cite[Lemma 3.9]{Fgaugelift}. Thus $\T_\et$ is $t$-exact.
\label{lem:Tet_texact}
\end{lemma}

\begin{proof}
This is by construction, as the definitions for both $t$-structures use Breuil-Kisin prisms for a covering of $\X$ (so we just use the same covering).

The second claim follows from \cite[Lemma 3.16]{Fgaugelift} and that a composite of $t$-exact functors is $t$-exact.
\end{proof}

We will also need some basic claims about duals of coherent sheaves and how $\T_\et$ interacts with these.

\begin{definition}
Let $\X$ be a smooth $p$-adic formal scheme over $\Spf \O_K$, and $\mathcal{E}\in \Coh(\X^{\Syn})$. We define
\[\mathcal{E}^\vee := \underline{\Hom}(\mathcal{E},\O_{\X^{\Syn}})\]
as the $\O$-linear dual, using the internal Hom in coherent sheaves. We set 
\[\mathcal{E}^{\vee \vee}:= \underline{\Hom}(\underline{\Hom}(\mathcal{E},\O_{\X^{\Syn}}),\O_{\X^{\Syn}}).\]
\label{defn:doubledual}
\end{definition}

In a closed monoidal category, in this setup we always have a natural map $\mathcal{E} \to \mathcal{E}^{\vee \vee}$.

We remark that $(-)^{\vee}$ sends coherent $F$-gauges to coherent $F$-gauges, as $\O_K^{\Syn}$ and in general $\X^{\Syn}$ for smooth quasicompact $\X/\Spf \O_K$ admit a flat-local surjection from a Noetherian regular affine formal scheme (and pullback along this cover preserves duals).

\begin{lemma}
Let $\X$ be a smooth $p$-adic formal scheme over $\Spf \O_K$. Let $\mathcal{E}\in \Coh(\X^{\Syn})$. Then $\T_\et$ is symmetric monoidal, and preserves the internal Hom so that $\T_\et(\mathcal{E}^\vee) = \T_\et(\mathcal{E})^\vee$.
\label{lem:dual}
\end{lemma}

\begin{proof}
We only add hypotheses on $\X$ so that it has a good theory of coherent sheaves.

Let us factor $\T_\et$ as
\[\Coh(\X^{\Syn}) \to \Coh(\X_{\Prism},\O_{\Prism}[1/\I_{\Prism}]_p^{\wedge})^{\varphi=1} \simeq \D^{(b)}_{\mathrm{lisse}}(\X_\eta,\Z_p)^\heart.\]
The first map is obtained by pullback to $\X^{\Prism}$ and then inverting $\I_{\Prism}$. Pullback to a substack is symmetric monoidal, as is tensoring with $\O_{\Prism}[1/\I_{\Prism}]_p^{\wedge}$. The final equivalence is symmetric monoidal as well, see Section 2.2 in \cite{Tannakian} (this is the underived version).

To check the claim about duals, it suffices to check this construction preserves internal Homs. The same argument reduces us to checking this for 
\[\Coh(\X_{\Prism},\O_{\Prism}[1/\I_{\Prism}]_p^{\wedge})^{\varphi=1} \simeq \D^{(b)}_{\mathrm{lisse}}(\X_\eta,\Z_p)^\heart,\]
where it is straightforward (note that for representations we take the $\Z_p$-linear dual and then give this the natural Galois action).
\end{proof}

It will be useful to know that when $\X \to \Y$ is a quasisyntomic cover, $\X^{\Syn} \to \Y^{\Syn}$ is a flat-local surjection, as this is important for descent results. We explain how to deduce this claim for expository reasons here. One can show using \cite{drinfeldalg} and \cite{Fgaugelift} something slightly stronger, which is that it is a flat-local surjection as well as a flat morphism of stacks in a non-representable sense. By this, we mean that the morphism $\X \to \Y$ of $p$-adic formal stacks has the property that there exist stacks $\U, \V$ and a commutative diagram
\begin{display}
\U \ar{d} \ar{r} & \X \ar{d} \\
\V \ar{r} & \Y
\end{display}
so $\U \to \V$, $\V \to \Y$ and $\U \to \X \times_\Y \V$ are all fpqc coverings in the usual representable sense. In what follows, we are able to take $\U$ and $\V$ to be reasonable stacks; for example for the prismatization these can be taken to be disjoint unions of affine formal schemes.

For example, $\Spf \Z_p \sqcup \Spec \F_p \to \Spf \Z_p$ is a flat-local surjection, but does not have this property. This property also suffices to check $t$-exactness of the pullback on $\D_{\mathrm{qc}}(-)$ for $p$-adic formal stacks. The author would like to thank Sasha Petrov and Kush Singhal for help in arguing this flatness property; any mistakes are due to the author.

\begin{proposition}[\cite{APC2}, Lemma 6.3]
Suppose $f:\X\to \Y$ is a quasisyntomic cover of quasisyntomic $p$-adic formal schemes. Then the induced map $f: \X^\Prism \to \Y^\Prism$ is a flat-local surjection of formal stacks. It is moreover flat.
\label{prop:prism_preserve_flat_covers}
\end{proposition}

\begin{proof}
Suppose first that $\X = \Spf \rS$ and $\Y = \Spf \R$ where $\rS$ and $\R$ are quasiregular semiperfectoid rings such that $\R\to \rS$ is a quasisyntomic cover. We want to show that the corresponding map of initial prisms $\Prism_\R\to \Prism_\rS$ is $(p,\I)$-completely faithfully flat, where $\I$ is the prismatic ideal of $\Prism_\R$ (and $\Prism_{\rS}$ by rigidity of prisms). We want to show that $\bar{\Prism}_\R\to \bar{\Prism}_{\rS}$ is $p$-completely faithfully flat. By Lemma \ref{lem:flat_on_associated_graded}, it suffices to check that the map 
\[\gr_\ast^{\conj} \bar{\Prism}_\R \to \gr_\ast^{\conj} \bar{\Prism}_{\rS}\]
is $p$-completely faithfully flat, where $\Fil^\conj_\ast \bar{\Prism}_\R$ is the conjugate filtration.

By the Hodge-Tate comparison, we know that 
\[\gr_i^{\conj} \bar{\Prism}_\R/p = \left({\bigwedge}_\R^i \L_{\R/\tilde \R} \otimes_\R^{\L} \R/p\right)\{-i\}[-i]\]
where $\tilde{\R}$ is any perfectoid ring with a map $\tilde \R\to \R$, and similarly for $\bar \Prism_{\rS}$. The following argument which also appears in the proof of \cite[Proposition 3.30]{Fgaugelift} allows us to deduce the desired faithful flatness.

By \cite[Remark 4.21]{BMS2} (see also \cite[Lemma 4.25]{BMS2}), $\L_{\rS/\R}$ actually has $p$-complete Tor amplitude in $[-1,-1]$ (as we already know it is in $[-1,0]$ since the map is quasisyntomic), and so we have an exact sequence of $\rS/p$-modules 
\[0\to (\L_{\R/\tilde \R}[-1]\otimes_\R^\L \R/p) \otimes_{\R/p} \rS/p \to \L_{\rS/\tilde \R}[-1]\otimes_{\rS}^\L \rS/p \to \L_{\rS/\R}[-1]\otimes_{\rS}^\L \rS/p \to 0.\]
Here all three modules are flat.
    
Using Propositions 25.2.4.2 and 25.2.3.4 in \cite{SAG}, we have
\[\left({\bigwedge}^i \L_{\R/ \tilde \R} \otimes_\R^{\L} \R/p \right)[-i] = \Gamma^i_{\R/p}(\L_{\R/\tilde \R}[-1] \otimes_\R^\L \R/p)\]
and similarly for $\rS$. Thus the claim we want to check is that $\Gamma^\ast_{\rS/p}(\L_{\rS/\tilde \R}[-1]\otimes_\rS^\L \rS/p)$ is flat over $\Gamma^\ast_{\R/p}(\L_{\R/\tilde \R}[-1]\otimes_\R^\L \R/p)$. Indeed faithful flatness of $\Gamma^\ast_{\rS/p}(\L_{\rS/\tilde \R}[-1]\otimes_{\rS}^\L \rS/p)$ over $\Gamma^\ast_{\R/p}(\L_{\R/\tilde \R}[-1]\otimes_\R^\L \R/p)$, by virtue of $\R/p \to \rS/p$ being faithfully flat, follows from faithful flatness of $\Gamma^\ast_{\rS/p}((\L_{\R/\tilde \R}[-1]\otimes_\R^\L \R/p) \otimes_{\R/p} \rS/p) \to \Gamma^\ast_{\rS/p}(\L_{\rS/\tilde \R}[-1]\otimes_\rS^\L \rS/p)$. This last claim about faithful flatness of divided power algebras follows from the exact sequence of flat modules. We explain one general argument for this claim about divided power algebras. For an exact sequence of modules
\[0 \to \rA \to \B \to \rC \to 0\]
over a ring $k$, if $\rC$ is a flat $k$-module then $\Gamma_k^*(\rA) \to \Gamma_k^*(\B)$ is faithfully flat. Indeed, flatness of the third term $\rC$ implies the exact sequence is universally exact by Stacks Project \href{https://stacks.math.columbia.edu/tag/058M}{058M}, and must then be a filtered colimit of split exact sequences by (6) of Stacks Project \href{https://stacks.math.columbia.edu/tag/058K}{058K}. This filtered colimit is constructed for any choice of a filtered colimit presentation of the third object so we may refine this so each third term in the exact sequence is finite free by Lazard's theorem. As $\Gamma_k^*$ commutes with filtered colimits we can now assume the exact sequence is split and $\rC$ is finite free, in which case the claim is easy to verify. It follows that $\Prism_\R \to \Prism_\rS$ induces an fpqc covering of formal schemes.

The reduction method in \cite[Corollary 6.12.3]{drinfeldalg} then shows how to deduce the flat-local surjection claim from the quasiregular semiperfectoid case. Indeed, we may reduce to the case $\X$ and $\Y$ are affine: all claims are local on the source and target, and to reduce to the affine case we also use Remark 3.9 in \cite{APC2} specialized to Zariski covers. The claim that $\Y_\infty^\Prism \to \Y^\Prism$ is a flat-local surjection is shown in the case of $\Y_\infty \to \Y$ for the specific quasiregular semiperfectoid cover $\Y_\infty$ in \cite[Corollary 6.12.3]{drinfeldalg} for arbitrary $\Y$ (we use $(-)_\infty$ to denote this covering construction applied to a general affine $p$-adic formal scheme).

Using this case and compatibility of prismatization with Tor-independent limits, as the property is flat-local on the target we are reduced to checking $\X^\Prism \times_{\Y^\Prism} (\Y_\infty)^\Prism \simeq (\X \times_\Y \Y_\infty)^\Prism \to \Y_\infty^\Prism$ is a flat-local surjection. Now considering $(\X \times_\Y \Y_\infty)_\infty \to \X \times_\Y \Y_\infty \to \Y_\infty$, as the composite map is now a quasisyntomic cover of quasiregular semiperfectoid rings, the composite is a flat-local surjection on prismatizations. Using the 2 out of 3 property for flat-local surjections, since $(\X \times_\Y \Y_\infty)_\infty^\Prism \to (\X \times_\Y \Y_\infty)^\Prism$ is a flat-local surjection the claim follows. 

To deduce the flatness claim the previous reduction to the affine case also applies, and we may use the same diagram
\begin{display}
(\X \times_\Y \Y_\infty)_\infty^\Prism \ar{r} \ar{d} & \X^{\Prism} \ar{d} \\
\Y_\infty^\Prism \ar{r} & \Y^{\Prism}
\end{display}
where it remains to verify $\Y_\infty^\Prism \to \Y^{\Prism}$ is an fpqc covering of formal stacks (and also the same check for $(\X \times_\Y \Y_\infty)_\infty^\Prism \to (\X \times_\Y \Y_\infty)^{\Prism}$). In the quasiregular semiperfectoid case $(\X \times_\Y \Y_\infty)_\infty^\Prism \to \Y_\infty^\Prism$ we have already checked the stronger claim.

To address both of these remaining claims at once, suppose that $\X$ is quasiregular semiperfectoid, but $\Y$ is general and $\X\to \Y$ is a quasisyntomic cover. We already know $\X^{\Prism} \to \Y^{\Prism}$ is a flat-local surjection, so we will aim to show the map is an fpqc covering of stacks. Since the prismatization commutes with Tor-independent limits, we have a pull-back diagram
\[
\begin{tikzcd}
(\X\times_\Y \X)^\Prism \ar{r}\ar{d} & \X^\Prism \ar{d}\\
\X^\Prism \ar{r} & \Y^\Prism
\end{tikzcd}
\]
Since $\X\times_\Y \X$ is also quasiregular semiperfectoid by \cite[Lemma 4.30]{BMS2}, we see that the left vertical map and the upper horizontal map are fpqc covers. Now, as $\X^\Prism \to \Y^\Prism$ is a surjection of sheaves in the flat topology, for any test object $\mathrm{S} \to \Y^{\Prism}$ there exists an fpqc cover $\mathrm{S}'\to \mathrm{S}$ with a map $\mathrm{S}' \to \Y^\Prism$ factoring through $\X^\Prism$, i.e. we have a diagram 
\[
\begin{tikzcd}
\Box & {(\X\times_\Y \X)^\Prism} & {\X^\Prism} \\
\mathrm{S}' & {\X^\Prism} & {\Y^\Prism}
\arrow[from=1-1, to=1-2]
\arrow["{\text{fpqc}}", from=1-1, to=2-1]
\arrow["{\text{fpqc}}", from=1-2, to=1-3]
\arrow["{\text{fpqc}}", from=1-2, to=2-2]
\arrow[from=1-3, to=2-3]
\arrow[from=2-1, to=2-2]
\arrow[bend right, from=2-1, to=2-3]
\arrow[from=2-2, to=2-3]
\end{tikzcd}
\]
where, \textit{a priori}, only the indicated arrows are fpqc covers of formal stacks. By virtue of the left square being Cartesian $\Box \to \mathrm{S}'$ is an fpqc cover of formal stacks, as indicated in the diagram. As both the left and right squares are Cartesian, we deduce the outer square is as well and so $\Box \simeq \X^{\Prism} \times_{\Y^{\Prism}} \mathrm{S}' \to \mathrm{S}'$ is fpqc. This suffices to check $\X^{\Prism} \times_{\Y^{\Prism}} \mathrm{S} \to \mathrm{S}$ is an fpqc cover as we may check this fpqc locally on the target, and since the test object was arbitrary it follows $\X^{\Prism} \to \Y^{\Prism}$ is an fpqc cover.
\end{proof}

We needed the following lemma in the proof of the proposition.

\begin{lemma}
Suppose $\rA$ is a filtered ring and $\M$ a filtered $\rA$-module, both of which are equipped with $\N$-indexed exhaustive (honest) increasing filtrations. If the associated graded module $\gr^\ast \M$ is faithfully flat over $\gr^\ast \rA$, then $\M$ is faithfully flat over $\rA$.
\label{lem:flat_on_associated_graded}
\end{lemma}

\begin{proof}
This follows from \cite[Proposition 3.12]{bjork} Chapter 2 (plus a short argument for faithfulness). See \cite[Lemma 3.29]{Fgaugelift} for another argument.
\end{proof}

\begin{corollary}[\cite{drinfeldalg}, Corollary 6.12.5]
Suppose $f:\X\to \Y$ is a quasisyntomic cover of quasisyntomic $p$-adic formal schemes. Then the induced map $f: \X^{\Nyg} \to \Y^{\Nyg}$ is a flat-local surjection of formal stacks, and similarly for $\X^{\Syn}$. The map is flat in both cases.
\label{cor:Nyg_preserve_flat_covers}
\end{corollary}

\begin{proof}
As before, for the flat-local surjection claim the argument in \cite[Corollary 6.12.3]{drinfeldalg} reduces the claim to checking that the induced map from a quasisyntomic cover $\Spf \rS \to \Spf \R$ of quasiregular semiperfectoid rings induces an fpqc covering of formal stacks on the associated Nygaard stacks. This claim for quasiregular semiperfectoids appears already in \cite[Lemma 6.12.4]{drinfeldalg}.

We now explain one possible argument for quasiregular semiperfectoids, noting $\R^{\Nyg}$ is the Rees stack of the Nygaard filtration on $\Prism_\R$ for $\R$ quasiregular semiperfectoid (see \cite[Theorem 6.11.7]{drinfeldalg}). We have already checked that the map of Nygaard stacks $\rS^{\Nyg} \to \R^{\Nyg}$ is faithfully flat on the $t\neq 0$ locus, and one can check also faithful flatness on the $t=0$ locus. We have essentially already done the nontrivial input to check this, since $\mathrm{gr}_{\Nyg} \Prism_\R = \bigoplus_{i\in \N} \Fil_i^{\conj} \overline{\Prism}_\R$. Now define a multiplicative filtration on the associated graded, given by $F_m = \bigoplus_{i\in \N} \Fil_{\min(m,i)}^{\conj} \overline{\Prism}_\R$. That is, on the $i$th summand we place a truncated conjugate filtration. The $m$th associated graded is then a direct sum $\bigoplus_{i\ge m} \gr_m^{\conj} \overline{\Prism}_\R$, and in particular after appropriately indexing the total associated graded is the polynomial extension $\mathrm{gr}_{\conj}^* \overline{\Prism}_\R[u]$ as a graded ring (the degree in the variable $u$ records $i-m$). The previous flatness check used for the conjugate filtration's associated graded then suffices, using Lemma \ref{lem:flat_on_associated_graded}. Since the Nygaard filtration is an honest filtration on quasiregular semiperfectoids these two checks suffice. More precisely, the derived $t=0$ fiber is classical due to the filtration being honest. The flatness claim is then also handled the same way as Proposition \ref{prop:prism_preserve_flat_covers} once we have the stronger claim in the quasiregular semiperfectoid case; we note \cite[Proposition 6.12.1]{drinfeldalg} replaces the use of Remark 3.9 in \cite{APC2}.

The claim for $\X^{\Syn} \to \Y^{\Syn}$ follows since $\Y^{\Nyg}$ is an \'{e}tale cover of $\Y^{\Syn}$ (also using that $\X^{\Nyg} \simeq \X^{\Syn} \times_{\Y^{\Syn}} \Y^{\Nyg}$).
\end{proof}

This also allows us to give a proof of the following property of $\X^{\Syn}$ in the $p$-adic setting, asserted for $\X/\Spec k$ in Warning 4.1.3 in \cite{Fgauge} (the warning is that separatedness can fail).

\begin{lemma}
Let $\X/\Spf \Z_p$ be a quasisyntomic $p$-adic formal scheme with affine diagonal. Then $\X^{\Syn}$ has affine diagonal.
\label{lem:aff_diag}
\end{lemma}

\begin{proof}
Suppose first that $\X$ is affine, equipped with $\Spf \rA \to \X^{\Prism}$ a flat-local surjection where $\rA = \Prism_\R$ for a quasiregular semiperfectoid ring $\R$ (which exists if $\Spf \R \to \X$ is a quasisyntomic cover by \cite[Lemma 6.3]{APC2}). Then consider the diagram
\begin{display}
\Spf \rA \times_{\X^{\Prism}} \Spf \rA \ar{d} \ar{r} & \Spf \rA \times_{\Spf \Z_p} \Spf \rA \ar{d} \\
\X^{\Prism} \ar{r} & \X^{\Prism} \times_{\Spf \Z_p} \X^{\Prism}.
\end{display}
This diagram is cartesian by writing $\X^{\Prism}=\X^{\Prism} \times_{\X^{\Prism}} \X^{\Prism}$ and using that limits commute with limits.

For the top map in this diagram, we can simplify to get
\[\Spf \rA \times_{\X^{\Prism}} \Spf \rA \simeq (\Spf \R \times_\X \Spf \R)^{\Prism} \simeq \Spf \Prism_{\Spf \R \times_\X \Spf \R}.\]
For the first equivalence we use that prismatization commutes with finite Tor-independent limits (Remark 3.5 in \cite{APC2}). Thus the top morphism is affine. It follows $\X^{\Prism}$ has affine diagonal as this property is flat local on the target (here one can use the fpqc covering property for a quasiregular semiperfectoid source provided by the proof of Proposition \ref{prop:prism_preserve_flat_covers}).

For general $\X$, we have a quasisyntomic covering $\tilde{\X}=\sqcup_i \Spf \R_i$ of $\X$ by quasiregular semiperfectoid rings by choosing an affine covering $\sqcup_i \Spf \rA_i$ of $\X$ and picking quasiregular semiperfectoid covers of each affine. Then there is a flat-local surjection (really an fpqc cover by Proposition \ref{prop:prism_preserve_flat_covers}) of $\X^{\Prism} \times_{\Spf \Z_p} \X^{\Prism}$ by $\widetilde{\X}^{\Prism} \times_{\Spf \Z_p} \widetilde{\X}^{\Prism}$ where $\widetilde{\X}^{\Prism} = \sqcup_i \Spf \Prism_{\R_i}$. We can then obtain the same diagram, but with $\widetilde{\X}^{\Prism}$ in place of $\Spf \rA$:
\begin{display}
\widetilde{\X}^{\Prism} \times_{\X^{\Prism}} \widetilde{\X}^{\Prism} \ar{d} \ar{r} & \widetilde{\X}^{\Prism} \times_{\Spf \Z_p} \widetilde{\X}^{\Prism} \ar{d} \\
\X^{\Prism} \ar{r} & \X^{\Prism} \times_{\Spf \Z_p} \X^{\Prism}.
\end{display}
We claim the preimages $\Spf \R_i \times_{\X} \Spf \R_j$ of the affines $\Spf \R_i \times_{\Spf \Z_p} \Spf \R_j$ covering $\tilde{\X} \times_{\Spf \Z_p} \tilde{\X}$ are affine and quasiregular semiperfectoid. If $\Spf \R_i$ and $\Spf \R_j$ cover $\Spf \rA_i, \Spf \rA_j \subset \X$ then since $\X$ has affine diagonal we see $\Spf \rA_i \times_\X \Spf \rA_j = \Spf \rA_{ij}$ for some ring $\rA_{ij}$. We then see $\Spf \R_i \times_{\X} \Spf \R_j = \Spf(\R'_i \widehat{\otimes}_{\rA_{ij}}\R'_j)$ where $\R'_i = \R_i \widehat{\otimes}_{\rA_i} \rA_{ij}$ and $\R'_j$ is defined analogously. This shows the result is affine; the rings $\R'_i, \R'_j$ are actually again quasiregular semiperfectoid (for example \cite[Lemma 1.6]{Tannakian} applies). It then follows $\Spf \R_i \times_{\X} \Spf \R_j$ is quasiregular semiperfectoid, as $\R'_i \widehat{\otimes}_{\rA_{ij}}\R'_j$ is easily checked to be quasiregular semiperfectoid.

Using the analogous affine covering $\Spf \Prism_{\R_i} \times_{\Spf \Z_p} \Spf \Prism_{\R_j}$ of $\widetilde{\X}^{\Prism} \times_{\Spf \Z_p} \widetilde{\X}^{\Prism}$ the preimages $\Spf \Prism_{\R_i} \times_{\X^{\Prism}} \Spf \Prism_{\R_j} \simeq \Spf \Prism_{\Spf \R_i \times_{\X} \Spf \R_j}$ are then all prismatizations of quasiregular semiperfectoids and hence affine formal schemes (using again the compatibility of prismatization with finite Tor-independent limits). Thus the map of formal schemes
\[\widetilde{\X}^{\Prism} \times_{\X^{\Prism}} \widetilde{\X}^{\Prism} \to \widetilde{\X}^{\Prism} \times_{\Spf \Z_p} \widetilde{\X}^{\Prism}\]
is affine, and the same argument shows $\X^{\Prism}$ has affine diagonal.

A similar argument shows $\X^{\Nyg}$ has affine diagonal, where we must also make the analogous argument for the Nygaard stack that the cover by the Rees stack of a quasiregular semiperfectoid is a flat-local surjection (using Corollary \ref{cor:Nyg_preserve_flat_covers}). Note that $(-)^{\Nyg}$ also commutes with finite Tor-independent limits.

The claim for $\X^{\Syn}$ can be checked by a similar argument, or by noting that it is obtained from a stack $\X^{\Nyg}$ with affine diagonal by gluing along two open immersions $j_{\dR},j_{\HT}$ with disjoint images (see \cite[Remark 5.3.6]{Fgauge}) that are affine monomorphisms.
\end{proof}

\section{Reflexive sheaves on \texorpdfstring{$\X^{\Syn}$}{XSyn}}

We can define reflexive $F$-gauges for general smooth $\X/\Spf \O_K$ in a similar way to the case of $\X=\Spf \Z_p$ (done in \cite{Fgauge}).

\begin{definition}
\label{def:reflexive}
Let $\X/\Spf \O_K$ be smooth. An $F$-gauge in $\Coh(\X^{\Syn})$ is \textit{reflexive} if the natural map $\mathcal{E} \to \mathcal{E}^{\vee \vee}$ is an isomorphism. We denote this full subcategory by $\mathsf{Refl}(\X^{\Syn})$.
\end{definition}

This definition is slightly different, so we verify it is compatible with the one in \cite{Fgauge}. That is, after picking a Breuil-Kisin prism $(\W(k)\llbracket u_0 \rrbracket,\E(u_0))$ we ask that the pullback along
\[p: (\Spf \W(k)\llbracket u_0\rrbracket\langle u,t\rangle/(ut-\E(u_0)))/\G_m \to \O_K^{\Syn}\]
is a reflexive sheaf after forgetting the grading.

\begin{lemma}
Let $\mathcal{E}\in \Coh(\O_K^{\Syn})$. Then $\mathcal{E} \to \mathcal{E}^{\vee \vee}$ is an equivalence if and only if $\mathcal{E}$ is reflexive in the sense of \cite{Fgauge}.
\label{lem:reflexivedef_equiv}
\end{lemma}

\begin{proof}
We again consider the flat-local surjection
\[\rho: (\Spf \W(k)\llbracket u_0\rrbracket\langle u,t\rangle/(ut-\E(u_0)))/\G_m \to \O_K^{\Nyg} \to \O_K^{\Syn}.\]

Suppose first that $\mathcal{E} \to \mathcal{E}^{\vee \vee}$ is an equivalence. Then applying $\rho^*$, from $t$-exactness (and symmetric monoidality) of the pullback we have $\rho^*(\mathcal{E}^{\vee})=(\rho^* \mathcal{E})^{\vee}$. Thus, this pulls back to $\rho^* \mathcal{E} \simeq (\rho^* \mathcal{E})^{\vee \vee}$, and $\mathcal{E}$ is reflexive in the sense of \cite{Fgauge}. Conversely, suppose we are only given the data of $\rho^* \mathcal{E} \simeq (\rho^* \mathcal{E})^{\vee \vee}$. We can check equivalences after pullback by a flat-local surjection by descent, so it follows the natural map $\mathcal{E}\to \mathcal{E}^{\vee \vee}$ inducing this must be an equivalence.
\end{proof}

Our goal in this section will be to prove that \'{e}tale realization induces an equivalence of categories $\mathsf{Refl}(\X^{\Syn}) \simeq \Loc_{\Z_p}^\cris(\X_\eta)$ when $\X/\Spf \O_K$ is smooth. We regard this as an improvement of the following result of Guo-Li, which already produces $F$-gauges lifting analytic prismatic $F$-crystals; we simply provide a characterization of the essential image. The advantage of this characterization is that it makes it clear that there is always a natural map
\[\mathcal{E} \to \mathcal{E}^{\vee \vee}\]
to the reflexive hull of an $F$-gauge, where the target is reflexive.

\begin{theorem}[Guo-Li, Theorem 3.32 in \cite{Fgaugelift}]
Let $\X/\Spf \O_K$ be a smooth $p$-adic formal scheme.

There is a fully faithful functor
\[\Pi_\X: \Coh^{\varphi,\I-\mathrm{tf}}(\X_\Prism) \to \Perf(\X^{\Syn})\]
uniquely characterized by the condition that for $S \to \X$ a $p$-completely flat quasiregular semiperfectoid cover it assigns the $F$-gauge with filtration
\[\mathrm{Fil}^\bullet(\Pi_\X(\mathcal{E})(\Prism_S)) = \varphi_{\mathcal{E}}^{-1}(\I^\bullet \mathcal{E}(\Prism_S))\]
but with the same underlying prismatic crystal and Frobenius.

Moreover, $\Pi_\X$ is the right adjoint of the forgetful functor $(-)|_{\X^{\Prism}}$ (on $\I$-torsionfree coherent objects)\footnote{Their notion of a coherent $F$-gauge is a priori slightly different, but in Lemma \ref{lem:PiX_essim} we at least check $\Coh(\X^{\Syn})$ contains their category.}; as the counit is an equivalence by the description, $\Pi_\X$ is fully faithful.
\label{thm:Fgaugelift}
\end{theorem}

As we will be interested in $F$-gauges up to isogeny later (and we will need it to argue $\mathsf{Refl}(\X^{\Syn}) \simeq \Loc_{\Z_p}^\cris(\X_\eta)$), it will be useful to first determine the kernel of
\[\T_\et: \Coh(\X^{\Syn}) \to \D^{(b)}_{\mathrm{lisse}}(\X_\eta,\Z_p)^{\heart}.\]
We provide a similar characterization of the kernel as in the case of a point shown in \cite{Fgauge}, namely that it consists of $(p,v_{1,\X})$-power torsion $F$-gauges.

\subsection{The kernel of the \'{e}tale realization}

We will want to understand the kernel of the \'{e}tale realization in order to show $\Pi_\X$ induces an equivalence
\[\mathsf{Refl}^{\varphi}(\X_{\Prism},\O_{\Prism})\simeq \mathsf{Refl}(\X^{\Syn}).\]

In order to state the result, we will need to understand the section $v_{1,\X}$ of $\O\{p-1\}/p$. This can be defined very generally.

\begin{definition}[\cite{Fgauge}, Construction 6.2.1]
Let $\X$ be a quasisyntomic formal scheme. To define a class
\[v_{1,\X}\in \H^0_{\Syn}(\X,\O\{p-1\}/p)\]
we can characterize it locally on quasiregular semiperfectoid rings and then check that it descends to a well-defined class.

For a $p$-torsionfree quasiregular semiperfectoid ring $\R$ with associated prism $(\Prism_\R,\I)$, we identify $\Prism_\R \{p-1\}/p \simeq \I^{-1}/p$. We can identify $\H^0(\R^{\Nyg},\O\{p-1\}/p)$ with $\mathrm{Fil}^{p-1}_{\Nyg} \Prism_\R \{p-1\}/p \subset \Prism_\R \{p-1\}/p\simeq \I^{-1}/p$. To produce the desired section it then suffices to produce a nonzero map
\[\I/p \to \mathrm{Fil}^{p-1}_{\Nyg} \Prism_\R/p\]
which exists because there is a natural inclusion $\I/p \subset \mathrm{Fil}^p_{\Nyg} \Prism_\R/p \subset \mathrm{Fil}^{p-1}_{\Nyg} \Prism_\R/p$. This map defines $v_{1,\R}$ after checking the section descends to $\R^{\Syn}$.
\label{def:v_1}
\end{definition}

In this subsection we will prove the following theorem by adapting the strategy for $\X=\Spf \Z_p$.

\begin{theorem}
Let $\X/\Spf \O_K$ be smooth and quasicompact. Then the kernel of $\T_\et$ on $\Coh(\X^{\Syn})$ consists precisely of coherent $F$-gauges which are killed by $(p,v_{1,\X})^n$ for $n \gg 0$.
\label{thm:general_Tet_kernel}
\end{theorem}

Precisely, this means that the $F$-gauge is $p$-power torsion and its mod $p$ reduction is further $v_1$-power torsion.

We will first show some results about coherent prismatic $F$-crystals to help deduce the $p$-power torsion part of this claim, which is the main content (the $v_{1,\X}$-power torsion is easy to deduce afterward). Recall that when $\X/\Spf \O_K$ is smooth there is a standard $t$-structure on $\Perf^{\varphi}(\X_{\Prism},\O_{\Prism})$ via \cite{Fgaugelift} \S 3, so we may make the following definition.

\begin{definition}[\cite{Fgaugelift}]
A coherent prismatic $F$-crystal is an $F$-crystal in perfect complexes in the heart of the standard $t$-structure on $\Perf^{\varphi}(\X_{\Prism},\O_{\Prism})$. We denote the category of coherent prismatic $F$-crystals by $\Coh^{\varphi}(\X_{\Prism},\O_{\Prism})$.
\end{definition}

This of course can also be understood as $\Coh^{\varphi}(\X^{\Prism})$, using the stack instead. Recall that restriction of a coherent $F$-gauge to $\X^{\Prism}$ naturally acquires a coherent $F$-crystal structure (as implied by Lemma \ref{lem:Tet_texact}). Despite the name, coherent prismatic $F$-crystals turn out to have strict restrictions on possible torsion, and in fact are vector bundles after inverting $p$.

\begin{proposition}
If $\X$ is smooth and quasicompact over $\Spf \O_K$ and $\mathcal{E}\in \Coh^{\varphi}(\X_{\Prism},\O_\Prism)$, the underlying prismatic crystal satisfies $\mathcal{E}[1/p]\in \Vect(\X^\Prism) [1/p]$.
\label{prop:Fcris_pinv_VB}
\end{proposition}

\begin{proof}
We may test this locally for $\X=\Spf \R$ small affine in the sense of \cite{Zpcrystalline2} (here we use quasicompactness to ensure we only have finite powers of $p$ in denominators), and therefore can choose a Breuil-Kisin prism $(\rA,\I)$ so that $\X=\Spf \R$. This then gives a flat-local surjection $\Spf \rA \to \X^{\Prism}$, so we may test the property $\mathcal{E}[1/p]\in \Vect(\X^\Prism)[1/p]$ after evaluation on the Breuil-Kisin prism $(\rA,\I)$. By definition the $\rA$-module $\M := \mathcal{E}(\rA,\I)$ is a finite height $\varphi$-module over $\rA$, so then \cite[Proposition 4.13]{Zpcrystalline2} applies and we deduce $\M[1/p]$ is a vector bundle (we remark the argument does not require a torsionfree hypothesis). 
\end{proof}

\subsubsection{The locally free locus in the cyclotomic case} In the cyclotomic case it is possible to show a stronger result about the locus where a coherent prismatic crystal (without the Frobenius isogeny condition) is a vector bundle, which we now explain. If the reader only wishes to understand the proof of Theorem \ref{thm:general_Tet_kernel}, they should skip to Proposition \ref{prop:p_tors_Tet}.

We will need some preliminaries about the $q$-connection that evaluation on a cyclotomic prism produces. If $\Spf \R$ is $p$-completely \'{e}tale over $\Spf \W(k)\langle \T_1^\pm, \ldots, \T_\ell^\pm \rangle$ for $\ell\ge 0$, then as in \cite{liustacky} \S 4.1 we can produce a variant of the $q$-de Rham prism $\rA=(\R\llbracket q-1 \rrbracket, [p]_{q^{p^n}}) = (\R\llbracket q-1 \rrbracket, \varphi^{n}([p]_q))$ as well as a $q$-connection on the evaluation on this prism. Letting $\X_\W=\Spf \R$, this prism induces a flat-local surjection $\Spf \rA \to (\X_\W[\zeta_{p^{n+1}}])^{\Prism} \to (\X_\W[\zeta_{p^{n}}])^{\Prism}$ (the second map using \cite[Lemma 6.3]{APC2}; we use $\X_\W[\zeta_{p^n}]$ as shorthand for $\Spf \R \widehat{\otimes}_{\W(k)} \W(k)[\zeta_{p^{n}}]$). This allows us to study sheaves on $(\X_\W[\zeta_{p^{n}}])^{\Prism}$ via pullback to this covering. Note that we use $\zeta_{p^n}$ here rather than $\zeta_{p^{n+1}}$ for $n\ge 0$ to also treat the unramified case, as otherwise this would be omitted. We fix $n\ge 0$ for this subsection.

For any coherent prismatic crystal $\mathcal{E}$ after evaluation on $(\rA,\varphi^{n}([p]_q))$ we obtain a module $\M$, and using \S 4.1 in \cite{liustacky} in the unramified case we may produce maps $\nabla_{\M,i}: \M \to \M$ satisfying a twisted Leibniz rule
\[\nabla_{\M,i}(am) = \gamma_i(a) \nabla_{\M,i}(m) + \nabla_{\rA,i}(a) m\]
where $\gamma_i$ sends $\T_i \mapsto q^{p^{n+1}} \T_i$ and fixes all other $\T_j$ (i.e. is the unique lift to $\rA$ of this map on $\W(k)\langle \T_1^\pm, \ldots, \T_\ell^\pm \rangle \llbracket q-1 \rrbracket$). Here, $\nabla_{\rA,i}: \rA \to \rA$ is given by $f \mapsto [p]_{q^{p^n}}\frac{\gamma_i(f) - f}{q^{p^{n+1}}\T_i - \T_i}$. These assemble to a full $q$-connection $\nabla_\M = \sum_i \nabla_{\M,i} \mathrm{d} \T_i: \M \to \M \otimes \mathrm{q}\Omega^1$, where $\mathrm{q}\Omega^1 = \rA \widehat{\otimes}_{\W(k)\langle \T_1^\pm, \ldots, \T^\pm_\ell \rangle \llbracket q-1 \rrbracket}  \Omega^1_{\W(k)\langle \T_1^\pm, \ldots, \T^\pm_\ell \rangle \llbracket q-1 \rrbracket/\W(k)\llbracket q-1 \rrbracket}$.

We can further produce a $q$-derivation $\partial: \M \to \M$ satisfying
\[\partial(am) = \gamma_\rA(a) \partial(m) + \partial_{\rA}(a) m\]
where $\gamma_{\rA}$ sends $q \mapsto q^{p^{n+1}+1}$ and $\partial_{\rA} = \frac{\gamma_{\rA}-\id}{q^{p^n}-1}$. Note that this differs from the convention in \cite{liustacky} by a unit.

For handling the case of $\X_\W$, we will need the following lemma. We use $\Phi_{p^k}(q)$ to denote the $p^k$th cyclotomic polynomial in $q$. When $k\ge 1$ this agrees with $[p]_{q^{p^{k-1}}}$. Via a gcd calculation one may check that in $\Spec \rA[1/p]$ the loci $\V(\Phi_{p^j}(q))$ for $j \ge 1$ are disjoint subsets of $\Spec \rA[1/p]$ for an algebra $\rA$ over $\W(k)\llbracket q-1 \rrbracket$. Note that these loci are indeed nonempty after inverting $p$, as while $\Phi_{p^j}(q)$ has an inverse in $\W(k)[1/p]\llbracket q-1 \rrbracket$ it does not in $\W(k)\llbracket q-1 \rrbracket[1/p]$.

\begin{lemma}
Let $\rA=\W(k)\llbracket q-1 \rrbracket$ and equip it with $\partial_{\rA}$ (as defined above).

Now suppose $\I\le \rA$ is a nonzero ideal which is stable under $\partial_{\rA}$. Then $\V(\I[1/p])\subset \Spec \rA[1/p]$ is contained in finitely many of the disjoint loci $\V(\Phi_{p^j}(q))\subset \Spec \rA[1/p]$ for $j > n$.
\label{lem:diff_ideal_classification}
\end{lemma}

\begin{proof}
We are given that $\I$ is a nonzero ideal which is stable under $\partial=\partial_{\W(k)\llbracket q-1 \rrbracket}$, which implies it is stable under $\gamma=\gamma_{\W(k)\llbracket q-1 \rrbracket}$. Since $\gamma$ is an automorphism fixing $p$ and must then have $\gamma(\I[1/p])=\I[1/p]$ (using that $\rA[1/p]$ is Noetherian), we deduce the finitely many minimal primes $\q_i$ of $\I[1/p]$ have finite $\gamma$-orbits. We will now classify such primes.

After inverting $p$ we find $\W(k)\llbracket q-1 \rrbracket[1/p]$ is a PID and we may then assume the minimal primes $\q_i \le \W(k)\llbracket q-1 \rrbracket[1/p]$ are generated by $f \in \W(k)[q-1]$ for some irreducible polynomial $f \equiv (q-1)^d \pmod{p}$. As $\gamma$ sends $q \mapsto q^{p^{n+1}+1}$, by factoring $f$ over $\C_p$ we note its roots in the variable $q$ lie in $\m+1 \subset \O_{\C_p}$ due to $f(q-1) \equiv (q-1)^d \pmod{p}$. Since $\gamma^j(q)=q^{(p^{n+1}+1)^j}$, to have a finite orbit we require $q\in \mu_{p^\infty}(\O_{\C_p})$ (any other roots of unity fail to lie in $\m+1$). Thus we allow only $\Phi_{p^j}(q)$ for $j\ge 0$ as possibilities for $f$ when we invert $p$. It follows that $\I[1/p]$ contains a product of finitely many powers of cyclotomic polynomials $\Phi_{p^j}(q)$.

Finally, we must show we can eliminate the remaining cyclotomic factors $\Phi_{p^j}(q)$ for $0\le j \le n$ from this product lying in $\I[1/p]$. Indeed, since we know $\W(k)\llbracket q-1 \rrbracket[1/p]$ is a PID we may assume $\I[1/p]$ is generated by $\prod_{s\in \rS} \Phi_{p^{s}}(q)$ for some finite multiset $\rS$ (i.e. allowing repeated factors). If one of these cyclotomic polynomials is $\Phi_{p^j}(q)$ for $0\le j \le n$, we may assume $\I[1/p]$ is generated by $\Phi_{p^j}(q)^N x$ where $x$ is a finite product of other cyclotomic factors. Applying $\partial$ we get $\gamma(\Phi_{p^j}(q)^N) \partial(x) + \partial(\Phi_{p^j}(q)^N) x\in \I[1/p]$. Using $j\le n$, one may compute that the first summand is divisible by $\Phi_{p^j}(q)^N$. The second summand is not: we know from the earlier gcd computation that $x$ is coprime to $\Phi_{p^j}(q)^N$ in $\W(k)\llbracket q-1 \rrbracket[1/p]$. A computation shows $\partial(\Phi_{p^j}(q)^N)$ is not divisible by $\Phi_{p^j}(q)^N$. This then gives a contradiction.
\end{proof}

The following technical lemma can then be shown by reducing to the previous one.

\begin{lemma}
Let $\X_\W\to \Spf \W(k)\langle \T_1^\pm, \ldots, \T_\ell^\pm \rangle$ be \'{e}tale, with $\rA$ the associated $q$-de Rham prism. Then if $\mathcal{E}\in \Coh(\X_\W[\zeta_{p^{n}}]^{\Prism})$, after evaluating on the prism $(\rA,[p]_{q^{p^n}})$ the module $\mathcal{E}(\rA)[1/p]$ is a vector bundle away from finitely many of the disjoint loci $\V(\Phi_{p^j}(q)) \subset \Spec \rA[1/p]$ for $j > n$.
\label{lem:unram_Fcris_pinv_VB}
\end{lemma}

\begin{proof}
We first may further reduce to the case that $\X_\W=\Spf \R$ is connected; by modifying the toric chart, potentially enlarging $k$ by a finite extension, we can further assume the special fiber is geometrically integral over $k$. It suffices to show we obtain a vector bundle on the desired locus after evaluating on the cyclotomic prism $(\rA,[p]_{q^{p^n}})=(\R\llbracket q-1\rrbracket, [p]_{q^{p^n}})$ and inverting $p$. Evaluating $\mathcal{E}$ on this prism yields a module $\M$ equipped with $\nabla_{\M,i}$ and $\partial$. If $\I=\mathrm{Fitt}(\M)$ is the first nonzero Fitting ideal it then suffices to show that $\I$ is the unit ideal after inverting $p$ and finitely many of $\Phi_{p^j}(q)$ for $j>n$ by using Stacks Project \href{https://stacks.math.columbia.edu/tag/07ZD}{07ZD}. Fitting ideals are stable under base change, so using this property we may claim 
\[\partial_{\rA}(\I)\subseteq \I, \nabla_{\rA,i}(\I) \subseteq \I\] 
where $\partial_{\rA}$ and $\nabla_{\rA,i}$ are the $q$-derivation and connection on the base prism $\rA$. We now show this for $\partial_{\rA}$; the proof for $\nabla_{\rA,i}$ is similar. Let $\rA_{\varepsilon} = \rA[\varepsilon]/(\varepsilon^2 - (q^{p^n}-1) \varepsilon)$, which is equipped with a map \[\sigma = \id + \varepsilon \partial_{\rA}: \rA \to \rA_{\varepsilon}.\] The existence of the $q$-derivation on $\M$ yields an isomorphism $\sigma^* \M \simeq \M_{\varepsilon}$ where on the right the module is the base change $\M \otimes_{\rA} \rA_{\varepsilon}$ along the canonical inclusion $\rA \to \rA_{\varepsilon}$. Indeed, one sends 
\[m \otimes 1 \mapsto m \otimes 1 + \partial_\M (m) \otimes \varepsilon\] 
and the axioms of a $q$-derivation make this an $\rA_{\varepsilon}$-module isomorphism. Base change of Fitting ideals implies $\mathrm{Fitt}(\M_{\varepsilon}) = \I \rA_{\varepsilon}$ and $\mathrm{Fitt}(\sigma^* \M)= \sigma(\I) \rA_\varepsilon$. Equality of these two ideals forces $\partial_{\rA}(\I) \subseteq \I$ by looking at the $\varepsilon$ component.

One may also similarly deduce stability of the ideal under $\nabla_{\rA,i}$. We will show that any ideal 
\[\I \le \R\llbracket q-1\rrbracket\]
stable under the operators $\nabla_{\rA,i}$ in fact contains a nonzero element of $\W(k)\llbracket q-1 \rrbracket$.

We may reduce to the weaker condition of considering ideals stable under $\gamma_{\rA,i}$ (as this is implied by $\nabla_{\rA,i}$-stability). Observe that $\R\llbracket q-1 \rrbracket$ can equivalently be viewed as $(p,q-1)$-completely \'{e}tale over $\W(k)\llbracket q-1 \rrbracket \langle \T_1^\pm,\ldots, \T_\ell^\pm \rangle_{(p,q-1)}$ (where the completion is $(p,q-1)$-adic). Now one can consider the analytic locus with respect to the maximal ideal of $\W(k)\llbracket q-1 \rrbracket$, defined as the adic space $\rA_a := \Spa(\rA,\rA) \setminus \V(p,q-1)$. For the $q$-de Rham prism $\Spf \W(k)\llbracket q-1 \rrbracket$ with the $(p,q-1)$-adic topology this construction decomposes as an adic space $\Y=(\Spf \W(k)\llbracket q-1 \rrbracket)_a=\mathbb{D}^\circ_{\W(k)[1/p]} \cup y_\partial$, which is quasicompact; the point $y_\partial$ has residue field $k\laur{q-1}$. The rational neighborhoods $\U_N = \{|p|\le |q-1|^N\}$ form a basis of opens around $y_\partial$. The following argument is somewhat technical, but the idea is simple. On fibers over $y\in\Y$, since $\gamma_{\rA,i}$ fix $\W(k)\llbracket q-1\rrbracket$ the $\gamma_{\rA,i}$-stability of $\I$ gives stability on all fibers $\I_y$. This forces the zero locus of $\I_y$ to be trivial by a rigid analytic accumulation point argument on many such fibers, constraining the zero locus of $\I$ enough to deduce the claim. To make this argument more straightforward we have already reduced to the case where the special fiber of $\X_\W$ is geometrically integral, forcing these fibers to all be connected and reduced.

One first $p$-saturates the ideal $\I$ (which preserves stability under $\gamma_{\rA,i}$) to ensure $\I/p$ is nonzero, deducing from $q-1$-torsionfreeness that on $\rA_a$ the ideal $\I_{y_\partial}$ is nonzero in the fiber $\rA_{k\laur{q-1}}$ over $y_\partial$; we will show $\I$ is then forced to become the unit ideal in this fiber $(\rA_a)_{y_\partial}$, meaning that $\V(\I)\subset \rA_a$ does not intersect some $\pi^{-1}(\U_N)$. Indeed, if $\pi:\rA_a\to \Y$ denotes the projection, then $C=\V(\I)\cap \pi^{-1}(\U_1)$ is a quasicompact spectral space. The subsets $C\cap \pi^{-1}(\U_N)$ are decreasing constructibly closed subsets of $C$, since the $\U_N$ are quasicompact rational opens. If all were nonempty, compactness for the constructible topology would force their intersection to be nonempty, but this intersection is $C\cap \pi^{-1}(y_\partial)$, which we will show to be empty. Thus $\V(\I)\cap \pi^{-1}(\U_N)=\emptyset$ for $N\gg 0$.
 
Now we study the fiber over $y_\partial$. We note this fiber is connected since we reduced at the start to the case where the special fiber of $\Spf \R$ is geometrically integral. If the ideal $\I_{y_\partial}$ is proper, we may pick a classical point $x$ in $\V(\I_{y_\partial})$. Base changing to a completed residue field of $x$, the \'{e}tale map is then a local isomorphism at the lift of the classical point $x$: we may then lift a small polydisk in the base $k\laur{q-1} \langle \T_i^\pm \rangle_{(q-1)}$ to a polydisk around $x$. Large $p$-powers of $\gamma_i$ preserve the disk, hence also lift. By studying the action of $\gamma_i$, since $|q-1|<1$ applying sufficiently high powers of $\prod_i \gamma_i^{p^{m_i}}$ we force the polydisk to contain infinitely many zeroes of local sections of $\I_{y_\partial}$ with nonzero stalk at $x$. These zeroes appear as a product of sets of infinitely many zeroes in each coordinate direction accumulating at $x$, forcing the section to be identically zero by the rigid analytic identity theorem. This applies at every point in $\V(\I_{y_\partial})$, and thus would contradict $\I_{y_\partial}$ being nonzero. Using connectedness we deduce $\I_{y_\partial}$ must be the unit ideal.

Now we study the other fibers for points in $\mathbb{D}^\circ_{\W(k)[1/p]}$. Take some nonzero $f\in \I$ and write $f=\sum_{n\ge 0} f_n (q-1)^n$. Let $n_0$ be the smallest integer such that $f_{n_0}\neq 0$, and choose a closed point $x$ away from the zero locus of $f_{n_0}$ in $\Spa(\R[1/p],\R)$ induced by a map $\R \to \O_L$. This induces an evaluation map $\R\llbracket q-1 \rrbracket \to \O_L\llbracket q-1 \rrbracket$; taking the norm $\mathrm{Nm}_{\O_L\llbracket q-1 \rrbracket/\W(k)\llbracket q-1 \rrbracket}$ of this evaluation map applied to $f$ produces $c\in \W(k)\llbracket q-1 \rrbracket$ with the property that $c(y)\neq 0$ implies $\I_y \neq 0$ for the fiber over $y\in \Y$. On points $y$ away from $\U_N$ we may again analyze fibers, using the same method to deduce that $\I_y$ is the unit ideal unless $q(y)$ is a root of unity. Away from $\U_N$, there are only finitely many such points as we have $|(q-1)(y)|<|p|^{1/N}$, constraining the support of $\rA/\I$ on the analytic locus to the zero locus of $g=c(q-1) \Pi_N(q)$ where $\Pi_N$ is some product of cyclotomic factors. So the module $g^m(\rA/\I)$ is sent to zero on $(\rA/\I)_a$, hence killed by a power of $(p,q-1)$. This shows the claim.

Once we know $\I \cap \W(k)\llbracket q-1 \rrbracket$ is nonzero, we observe the intersection remains $\partial_{\rA}$-stable. So Lemma \ref{lem:diff_ideal_classification} applies, and the desired claim follows.
\end{proof}

For an ordinary connection (as opposed to a $q$-connection) we have the following result, using \cite[Corollary 2.5.2.2]{andre2001}. 

\begin{corollary} 
Let $\X/\Spf \W(k)$ be smooth and quasicompact. If $\M\in \Coh(\X)$ admits a connection, then $\M[1/p]$ is a vector bundle. 
\label{cor:pinv_VB} 
\end{corollary}

\begin{proof}
As $\X$ is smooth and quasicompact, there is a finite covering by formal smooth affines and it suffices to prove the claim for these. In this case, taking the adic generic fiber induces a coherent sheaf with connection on the adic generic fiber, which is then a vector bundle by applying \cite[Corollary 2.5.2.2]{andre2001}.
\end{proof}

This torsion control also implies Proposition \ref{prop:Fcris_pinv_VB} in the cyclotomic case. We may test the claim of Proposition \ref{prop:Fcris_pinv_VB} Zariski locally in $\X$. We may then assume that $\X=\X_\W[\zeta_{p^n}]$ as in Lemma \ref{lem:unram_Fcris_pinv_VB}. Let $\I$ denote the first nonzero Fitting ideal of the module $\M$ we get after evaluation on the corresponding cyclotomic prism $(\rA,[p]_{q^{p^n}})$. Our goal is to show $\I[1/p]=\rA[1/p]$ using the prismatic $F$-crystal structure, so then applying Stacks Project \href{https://stacks.math.columbia.edu/tag/07ZD}{07ZD} the desired claim follows (using that $(\X/\rA)^{\Prism} \to \X^{\Prism}$ is a flat-local surjection). We may use $\I[1/[p]_{q^{p^n}}]\simeq \varphi(\I)[1/[p]_{q^{p^n}}]$, which follows from stability of Fitting ideals under base change. Recall that via a gcd calculation one may check in $\Spec \rA[1/p]$ that the loci $\V(\Phi_{p^j}(q))$ for $j \ge 1$ are disjoint subsets of $\Spec \rA[1/p]$. Now $\I$ is not stable under the faithfully flat map $\varphi$, but $\I[1/\Phi_{p^{n+1}}(q)]=\varphi(\I)[1/\Phi_{p^{n+1}}(q)]$ implies that $\V(\varphi(\I))$ agrees with $\V(\I)$ restricted to $\Spec \rA[1/p,1/\Phi_{p^{n+1}}(q)]$. We know 
\[\V(\I[1/p])\subset \bigcup_{j\in \rS} \V(\Phi_{p^j}(q))\]
as loci in $\Spec \rA[1/p]$ for some finite subset $\rS \subset \Z_{> n}$ by Lemma \ref{lem:unram_Fcris_pinv_VB}; we choose the finite set $\rS \subset \Z_{> n}$ to be minimal. Since $\varphi$ is faithfully flat and the preimage of $\V(\Phi_{p^j}(q))$ is $\V(\Phi_{p^{j+1}}(q))$ in $\Spec \rA[1/p]$, we learn $\rS \setminus \{n+1\} = (\rS+1) \setminus \{n+1\}$, forcing $\rS=\emptyset$. That is, $\V(\I[1/p])$ is empty in $\Spec \rA[1/p]$ as desired.

\subsubsection{Proving Theorem \ref{thm:general_Tet_kernel}} 

Next, we apply Proposition \ref{prop:Fcris_pinv_VB} to understand the kernel of the \'{e}tale realization to work towards Theorem \ref{thm:general_Tet_kernel}.

\begin{proposition}
Assume $\X$ is smooth and quasicompact over $\Spf \O_K$ and let $\mathcal{E}$ be a coherent prismatic $F$-crystal. If $\T_\et(\mathcal{E})$ is $p$-power torsion, then so is $\mathcal{E}$. 
\label{prop:p_tors_Tet}
\end{proposition}

\begin{proof}
We may work locally, so we can assume that $\X$ is smooth affine so that $\X^{\Prism}$ has a Breuil-Kisin prism $(\rA,\I)$ where $\rho: \Spf \rA\to \X^{\Prism}$ is a flat-local surjection. Using Proposition \ref{prop:Fcris_pinv_VB} and letting $\rA_{\perf}$ denote the perfection of $\rA$, the pullback of $\mathcal{E}$ to $\Spf \rA_{\perf}$ must then be a vector bundle over $\rA_{\perf}[1/p] = \W(\R_{\perf}^\flat)[1/p]$. For a perfectoid ring, there is an equivalence
\[\T_\et[1/p]: \Vect(\W(\R_{\perf}^\flat)[1/\I]_p^{\wedge}[1/p])^{\varphi=1} \simeq \Loc_{\Z_p}(\R_{\perf,\eta})[1/p].\]
Since $\T_\et(\mathcal{E})[1/p]$ is assumed to be zero, it follows $\mathcal{E}(\W(\R_{\perf}^\flat),\I)[1/\I]^{\wedge}_p[1/p]$ is zero. Hence we conclude $\mathcal{E}(\W(\R_{\perf}^\flat),\I)[1/\I]^{\wedge}_p$ is $p$-power torsion. It follows by $p$-completely flat descent that the $F$-crystal $\mathcal{E}(\rA,\I)[1/\I]_p^{\wedge}$ is $p$-power torsion by replacing $\mathcal{E}$ with its $p$-torsionfree quotient $\mathcal{E}^{p-\mathrm{tf}}$ and using descent to conclude this is zero (the Frobenius on $\rA$ is faithfully flat, so $\Spf \rA_{\perf}[1/\I]_p^{\wedge} \to \Spf \rA[1/\I]_p^{\wedge}$ is a flat cover). Finally, knowing that $\mathcal{E}(\rA,\I)[1/p]$ is a vector bundle by Proposition \ref{prop:Fcris_pinv_VB} this forces $\mathcal{E}(\rA,\I)[1/p]=0$ (as it would be a projective module which becomes zero after an injective base change), after which the flat-local surjection $\rho$ proves the claim.
\end{proof}

We can now understand the kernel of $\T_\et$ using the argument in Remark 6.3.5 of \cite{Fgauge}.

\begin{proof}[Proof of Theorem \ref{thm:general_Tet_kernel}]
It is easy to check that $(p,v_1)$-power torsion objects are killed by $\T_\et$. Let $\mathcal{E}\in \Coh(\X^{\Syn})$ be killed by $\T_\et$. First, we note that $\mathcal{E}/p\mathcal{E}$ is $v_{1,\X}$-power torsion. As $\X$ is smooth over $\O_K$, picking a framing reduces us to the case where $\X=\Spf \R$ is \'{e}tale over $\Spf \O_{K} \langle \X_1^\pm, \ldots, \X_d^\pm \rangle$ and thus admits a quasisyntomic cover by a perfectoid $\Spf \tilde{\R}$. We may check the $v_{1,\X}$-power torsion claim locally, and on a perfectoid this is obvious since \'{e}tale realization simply inverts $v_{1,\tilde{\R}}$ for $p$-torsion $F$-gauges. Using quasicompactness and coherence we get a uniform exponent.

Next, we must show that $\mathcal{E}$ is $p$-power torsion. We know that $\mathcal{E}|_{\X^{\Prism}}$ is $p$-power torsion using Proposition \ref{prop:p_tors_Tet}, as the \'{e}tale realization is zero and compatible with restriction to $\X^{\Prism}$, and then we can deduce $\mathcal{E}$ is $p$-power torsion analogously to Remark 6.3.5 of \cite{Fgauge}, except we use a cover of $\X^{\Syn}$.

Again by quasicompactness, we may test the $p$-power torsion claim Zariski locally for $\X=\Spf \R$ using a Breuil-Kisin prism $(\rA,(d))$. Consider the associated cover
\[\rho_{\rA}: \mathrm{Rees}_{\I^\bullet}(\rA) \to \X^{\Syn}\]
and identify the Rees stack with $(\Spf \rA\langle u,t \rangle/(ut-d))/\G_m$ where $\deg(t)=1, \deg(u)=-1$. We simplify the argument by replacing $\mathcal{E}$ by its $p$-torsionfree quotient, and we now want to show it pulls back to $0$ on this cover (in which case we are done by flat descent). We know the restriction of $\mathcal{E}$ to the $t\neq 0$ locus of $\X^{\Nyg}$ (i.e. after pullback along $j^*_{\dR}$) is zero as this yields $\mathcal{E}|_{\X^{\Prism}}$, so we deduce that $\mathcal{E}$ is $t$-power torsion. Here one can use the grading on $\M:=\rho_{\rA}^* \mathcal{E}$ to deduce this, as $\M[1/t]^{\wedge}_p=0$ implies $\M$ is $t$-power torsion for a graded module.

Now filtering by powers of $t$, we can assume $t\M=0$ or that $\M$ is a sheaf on 
\[\Spf(\rA\langle u,t\rangle/(ut-d))/\G_m |_{t=0} \simeq [\Spf \R\langle u \rangle/\G_m].\]
The sheaf $\M$ vanishes after pullback to $[\Spf \R\langle u^\pm \rangle/\G_m]$ as the pullback along $j_{\HT}^*$ also vanishes by virtue of $\mathcal{E}$ being an $F$-gauge: both $j^*_{\dR}, j^*_{\HT}$ identify with $j^*_\Prism$ which we know vanishes. We obtain a finitely generated graded module $\M=\rho_{\rA}^* \mathcal{E}|_{t=0}$ over $\R \langle u \rangle$ such that $\M[1/u]_p^{\wedge}=0$. We will show a graded connection on 
\[\M^{(1)} = (\rho'_{\rA})^* \mathcal{E}|_{t=0} \in \Coh((\mathrm{Rees}_{\Fil^\bullet_{\Nyg}} \rA^{(1)})_{t=0})\] 
arises naturally from $\mathcal{E}$ on $\Coh(\X^{\Nyg})$. Then we will use this to deduce $\M[1/p]=0$, which since we replaced $\mathcal{E}$ by its $p$-torsionfree quotient forces $\rho_{\rA}^* \mathcal{E}=0$ as desired. Recalling Proposition \ref{prop:BK_flatcover}, showing $\M^{(1)}[1/p]=0$ implies $\M[1/p]=0$ via pullback. We also know $\M^{(1)}[1/u]_p^{\wedge}=0$ for the same reason as $\M[1/u]_p^{\wedge}=0$.

We now turn to producing the graded connection. Write $(\rA,(d))=(\R_0\llbracket u_0 \rrbracket, (\E(u_0)))$ where $\E(u_0)$ is an Eisenstein polynomial for a uniformizer $\pi$ of $\O_K$. We may adapt the Nygaard-filtered variant of $\nabla_\M$ from Definition \ref{def:nabla} to the module $\M^{(1)}$. We may use the same construction, just replacing $\W(k)$ with $\R_0$: we have $\Gamma_0 = \R_0\llbracket u_0 \rrbracket$ and $\Gamma_1 = \R_0\llbracket u_0,u_1 \rrbracket \left \{\frac{u_1-u_0}{\E(u_0)}\right\}^{\wedge}$ where we take the $(p,\E(u_0))$-adic completion and also $\mathrm{F}$-complete. The map $\Theta$ can be constructed completely analogously, just now using a $\Gamma_0$-linear basis. The connection is similarly Nygaard filtered, and again has a twisted variant on $\M^{(1)}$ which is what we will use and denote $\nabla_{\M^{(1)}}$.\footnote{It is worth noting in the general case $\nabla_{\M^{(1)}}$ doesn't capture the full data of a prismatic crystal: using the analogous descent complex as in the discussion preceding Lemma \ref{lem:trunc_desc_complex}, this is specifying instead the data of a coherent sheaf on $(\X/\R_0)^{\Prism}$, the prismatic stack relative to the $\delta$-ring $\R_0$ as in \cite{PrismaticDelta}. However we only need $\nabla_{\M^{(1)}}$ in the argument.}

Passing to the Nygaard associated graded, we obtain an induced map $\overline{\nabla}_{\M^{(1)}}$ on 
\[\mathrm{Rees}_{\Fil^\bullet_{\Nyg}}(\rA^{(1)})_{t=0}\simeq \rA^{(1)}/(1 \otimes \E(u_0)) \langle u \rangle/\G_m.\] 
For $\M^{(1)}$ coming from the structure sheaf, we may express $\nabla = \Theta \circ (\eta_R - \eta_L)$ in the notation of \cite{KtheoryZpn}, and then on the twist $\R_0\llbracket u_0 \rrbracket^{(1)}$ the Nygaard filtration has $\Fil_{\Nyg}^i$ generated by $u^i:=1 \otimes \E(u_0)^i$; following these elements through the definition of $\nabla$, one expands $(\eta_R - \eta_L)(1 \otimes \E(u_0)^i)=1\otimes (\E(u_1)^i-\E(u_0)^i)$ and applies $\Theta$ to deduce $\overline{\nabla}$ satisfies $\mathrm{gr}^{i-1}_{\Nyg} \overline{\nabla}(u^i)=i \E'(\pi) u^{i-1}$. In particular one uses a similar argument as in \cite[Lemma 4.51(ii)]{KtheoryZpn} to obtain divided powers of $g_0:=u_1-u_0$ and rewrite powers $g_0^n$ for $n\ge p$ in the $g_u$ basis that $\Theta$ is defined in; the identity for simplifying $g_u^p$ in the $g_u$ basis actually becomes $g_u^p=-p g_{u+1}$ in the Nygaard associated graded. Using these divided powers we may Taylor expand $\E(u_1)^i$, and then applying $\Theta$ the only term contributing is $1 \otimes i \E'(u_0) \E(u_0)^{i-1} g_0$, which is sent by $\Theta$ to the desired output $i \E'(\pi) u^{i-1}$ on the associated graded (note $\rA^{(1)}/(1 \otimes \E(u_0)) \simeq \R$ and $1 \otimes \E'(u_0) \mapsto \E'(\pi)$).

We may twist $\M^{(1)}$ so that $\Fil^0_{\Nyg} \M^{(1)} = \M^{(1)}$ for convenience. For general $\M^{(1)}$, since on the Nygaard associated graded we have $g_u^p=-p g_{u+1}$ we deduce the Hopf algebroid $(\mathrm{gr}^\ast_{\Nyg} \Gamma_0^{(1)},\mathrm{gr}^\ast_{\Nyg} \Gamma_1^{(1)})$ (with $\Gamma_0, \Gamma_1$ as in \cite{KtheoryZpn} Section 4.6 for $\R=\O_K$; see \S 4.2 for more details) has $\mathrm{gr}^\ast_{\Nyg} \Gamma_1^{(1)}$ given as a free divided power algebra in $g_0$ over $\mathrm{gr}^\ast_{\Nyg} \Gamma_0^{(1)}$ with $g_0$ primitive, and moreover $\Theta$ is similar to \cite[Example 4.49]{KtheoryZpn} since now $\Theta(g_0^{[n]})=0$ unless $n=1$ (where we get $1$). This situation suffices to deduce a Leibniz rule for a general finitely presented comodule $\M^{(1)}$. Define $\overline{\nabla}_{\M^{(1)}}:=(1\otimes \Theta)\circ \rho$. Writing $\rho(m)=\sum_{n\ge 0} m_n \otimes g_0^{[n]}$, we have $m_0=m$ by counitality, and we write $\eta_R(a)=\sum_{i\ge 0} \overline{\nabla}^{[i]}(a)g_0^{[i]}$ for $a \in \mathrm{gr}^\ast_{\Nyg}\Gamma_0^{(1)}$. Here $\overline{\nabla}^{[i]}(a)$ are by definition the coefficients of $g_0^{[i]}$, agreeing with $\overline{\nabla}$ when $i=1$ and giving the identity when $i=0$. Hence
\[\rho(ma)=\rho(m)\eta_R(a)=\sum_{i,j\ge 0} m_j\overline{\nabla}^{[i]}(a) \otimes g_0^{[j]}g_0^{[i]}.\]
Using $g_0^{[i]}g_0^{[j]}=\binom{i+j}{i}g_0^{[i+j]}$, the only terms contributing to the $g_0^{[1]}$-coefficient are $(i,j)=(1,0)$ and $(0,1)$. Moreover, from the expansion of $\rho(m)$ and the definition $\overline{\nabla}_{\M^{(1)}}=(1\otimes \Theta)\circ \rho$ we have $m_1=\overline{\nabla}_{\M^{(1)}}(m)$. Thus this coefficient is $m\overline{\nabla}(a)+m_1a=m\overline{\nabla}(a)+\overline{\nabla}_{\M^{(1)}}(m)a$. Since $\Theta(g_0^{[n]})=0$ for $n\neq 1$ and $\Theta(g_0)=1$, this coefficient is exactly $\overline{\nabla}_{\M^{(1)}}(ma)$, giving the Leibniz rule, or equivalently $\overline{\nabla}_{\M^{(1)}}(am)=\overline{\nabla}(a)m+a\overline{\nabla}_{\M^{(1)}}(m)$.

Since $\M^{(1)}[1/u]_p^{\wedge}=0$ and since it is a \textit{graded} module over $\rA \otimes_{\rA,\varphi} \R\langle u\rangle$ we may conclude it must be $u$-torsion (again using the grading). Now consider $\M^{(1)}[1/p]$. Take the minimal $i>0$ such that $u^i m=0$ for all $m$, and then applying $\overline{\nabla}_{\M^{(1)}}$ to both sides contradicts minimality by showing $0=i \E'(\pi) u^{i-1} m$; the factor in front is now a unit since we inverted $p$ (the other factor coming from the Leibniz rule vanishes, being divisible by $u^i$). Thus after we invert $p$ the module is forced to be zero and thus $\M[1/p]=0$ as desired.
\end{proof}

\subsection{Reflexive F-gauges}

We can define reflexive prismatic $F$-crystals in a similar way to reflexive $F$-gauges.

\begin{definition}
A coherent prismatic $F$-crystal $\M$ on a smooth $p$-adic formal scheme over $\O_K$ is reflexive if the canonical map $\M \to \M^{\vee \vee}$ is an equivalence. We let $\mathsf{Refl}^{\varphi}(\X_{\Prism},\O_{\Prism})$ denote the category of reflexive prismatic $F$-crystals.

Here, $(-)^\vee$ means we take the $\O_{\Prism}$-linear dual and endow it with the structure of a coherent prismatic $F$-crystal.
\end{definition}

There is something to check in this definition, namely that the $\O_{\Prism}$-linear dual still has the structure of a coherent prismatic $F$-crystal for $\X/\Spf \O_K$ smooth. For a prism $(\rA,\I)$ and $\M\in \Perf^\heart(\rA)$ such that $\varphi^* \M[1/\I]\simeq \M[1/\I]$ we claim that
\[\varphi^*\Hom(\M,\rA)[1/\I] \simeq \Hom(\M,\rA)[1/\I]\]
if the Frobenius $\varphi$ on $\rA$ is flat.

Observe $\Hom(\M,\rA)[1/\I]=\Hom(\M[1/\I],\rA[1/\I])$. Then \[\varphi^* \Hom(\M[1/\I],\rA[1/\I]) \simeq \Hom(\varphi^* \M[1/\I],\rA[1/\I]) \simeq \Hom(\M[1/\I],\rA[1/\I])\]
where we used flatness and $\varphi^* \rA[1/\I]\simeq \rA[1/\I]$ in the first isomorphism and the assumption on $\M$ that $\varphi^* \M[1/\I]\simeq \M[1/\I]$ in the second. For $\X$ smooth over $\Spf \O_K$, we can describe $\Coh^{\varphi}(\X_{\Prism},\O_{\Prism})$ by descent using Breuil-Kisin prisms. The Frobenius on such a prism is faithfully flat by \cite[Corollary 3.6]{Fgaugelift}, and the same will hold for prisms in the \v{C}ech nerve. Thus, we have a good notion of duals.

For the following lemma, we regard a reflexive module as a module $\M$ in $\Perf^\heart(\rA)$ such that the canonical map $\M \to \M^{\vee \vee}$ is an isomorphism, and denote this category by $\mathsf{Refl}(\rA)$. For a prism $(\rA,\I)$, we write
\[\mathsf{Refl}^{\varphi}(\rA) = \{(\M,f) : \M\in \mathsf{Refl}(\rA), f: \varphi_{\rA}^* \M[1/\I] \simeq \M[1/\I]\}.\]

\begin{lemma}
The category $\mathsf{Refl}^{\varphi}(\X_{\Prism},\O_{\Prism})$ satisfies $p$-complete \'{e}tale descent in smooth $p$-adic formal schemes over $\O_K$.

If $\X$ is smooth affine, then given any Breuil-Kisin prism $(\rA,\I)$ for $\X$ we have
\[\mathsf{Refl}^{\varphi}(\X_{\Prism},\O_{\Prism}) \simeq \lim_{n\in \Delta} \mathsf{Refl}^{\varphi}(\rA^{(n)})\]
where $\rA^{(n)}$ is the $n$th term in the \v{C}ech nerve of $(\rA,\I)$.
\label{lem:refl_desc}
\end{lemma}

\begin{proof}
We argue these claims for $\Coh^{\varphi}(\X_{\Prism},\O_{\Prism})$, as the assertions for reflexive objects in this category follow formally from the descent equivalences being compatible with duals.

For \'{e}tale descent, let $f: \tilde{\X} \to \X$ be an \'{e}tale covering. Then $\tilde{\X}$ remains smooth, so we have a well-defined $t$-structure on $\Perf^{\varphi}(\tilde{\X}_{\Prism},\O_{\Prism})$. As $\Perf^{\varphi}(\X_{\Prism},\O_{\Prism})$ satisfies quasisyntomic descent (\cite[Proposition 2.14]{prismFcris}), it also satisfies descent for this covering. Now it suffices to check that the pullback map $f^*: \Perf^{\varphi}(\X_{\Prism},\O_{\Prism}) \to \Perf^{\varphi}(\tilde{\X}_{\Prism},\O_{\Prism})$ is $t$-exact, which can be seen from \cite[Proposition 3.11]{Fgaugelift}.

In the situation where $\X$ is smooth affine with a Breuil-Kisin prism $(\rA,\I)$, we have
\[\Perf^{\varphi}(\X_{\Prism},\O_{\Prism}) \simeq \lim_{n\in \Delta} \Perf^{\varphi}(\rA^{(n)}).\]
The claim for $\Coh^{\varphi}(\X_{\Prism},\O_{\Prism})$ follows because the maps $\rA \to \rA^{(n)}$ are flat (rather than just $(p,\I)$-completely flat) as the source is Noetherian and both source and target are derived $(p,\I)$-complete.
\end{proof}

We now show that, when $\X/\Spf \O_K$ is smooth, reflexive prismatic $F$-crystals coincide with analytic prismatic $F$-crystals.

\begin{proposition}
Let $\X$ be affine and \'{e}tale over $\Spf \O_K \langle \X_1^\pm,\ldots,\X_n^\pm \rangle$, and let $(\rA,\I)$ be a Breuil-Kisin prism associated to $\X$. If $\mathcal{E}\in \mathsf{Refl}^{\varphi}(\X_{\Prism},\O_{\Prism})$, then
\[\mathcal{E}(\rA,\I)|_{\Spec \rA \setminus \V(p,\I)}\]
is a vector bundle.
\label{prop:refl_VB}
\end{proposition}

\begin{proof}
We observe that $\T_\et(\mathcal{E})$ is a reflexive object in $\D^{(b)}_{\mathrm{lisse}}(\X_\eta,\Z_p)^\heart$ due to compatibility with duals and $t$-exactness of $\T_\et$ (see \cite[Lemma 3.16]{Fgaugelift}; we use the underived dual). This implies it is in the subcategory $\Loc_{\Z_p}(\X_\eta)$ as the condition $\mathbb{L}\simeq \mathbb{L}^{\vee \vee}$ implies $\mathbb{L}\in \D^{(b)}_{\mathrm{lisse}}(\X_\eta,\Z_p)^{\heart}$ has no $p$-torsion. 

By \cite[Corollary 3.8]{prismFcris}, there is an equivalence
\[\Vect(\X_{\Prism},\O_{\Prism}[1/\I_{\Prism}]_p^{\wedge})^{\varphi=1} \simeq \Loc_{\Z_p}(\X_\eta).\]
The \'{e}tale realization will factor through this equivalence, so we deduce $\mathcal{E}[1/\I_{\Prism}]_p^{\wedge}$ is a vector bundle (here we are using the more general equivalence in \cite[Corollary 3.7]{prismFcris} and the $t$-exactness of this equivalence, from the argument in \cite[Lemma 3.16]{Fgaugelift}). Thus if $\M=\mathcal{E}(\rA,\I)$ then $\M[1/\I]_p^{\wedge}$ is a vector bundle.

As Proposition \ref{prop:Fcris_pinv_VB} applies to any prismatic $F$-crystal, we deduce from it that $\M[1/p]$ is a vector bundle. By Beauville-Laszlo gluing\footnote{In the Noetherian situation we are working in, this amounts to faithfully flat descent via the map $\rA[1/\I] \to \rA[1/\I]^{\wedge}_p \times \rA[1/\I,1/p]$, as completion is flat.}, we have
\[\Vect(\rA[1/\I]) \simeq \Vect(\rA[1/\I,1/p]) \times_{\Vect(\rA[1/\I]_p^{\wedge}[1/p])} \Vect(\rA[1/\I]_p^{\wedge})\]
so we deduce that $\M[1/\I]$ is a vector bundle. It then follows that $\M|_{\Spec \rA \setminus \V(p,\I)}$ is a vector bundle.
\end{proof}

We can now use the \'{e}tale descent of $\mathsf{Refl}^{\varphi}(\X_{\Prism},\O_{\Prism})$ to finish.

\begin{theorem}
Let $\mathcal{E}$ be a coherent prismatic $F$-crystal on some smooth $\X/\Spf \O_K$. Then $\mathcal{E}$ is reflexive if and only if $\mathcal{E}$ is an analytic prismatic $F$-crystal.
\label{thm:prismFcris_reflexive}
\end{theorem}

\begin{proof}
First, we claim there is a fully faithful inclusion functor
\[\Vect^{\varphi,\mathrm{an}}(\X_\Prism, \O_\Prism) \to \mathsf{Refl}^{\varphi}(\X_{\Prism},\O_{\Prism})\]
where the target is the category of reflexive prismatic $F$-crystals. Full faithfulness is automatic, so we only need to show analytic prismatic $F$-crystals are reflexive. By Zariski descent of both categories (see Lemma \ref{lem:refl_desc}), we can assume $\X$ is \'{e}tale over $\Spf \O_K\langle \X_1^\pm, \ldots, \X_n^\pm \rangle$ by choosing a framing. We can choose a Breuil-Kisin prism $(\rA,\I)$ so $\Spf \rA/\I\simeq \X$. Noting that we can test reflexivity after evaluation on this prism by the second claim in Lemma \ref{lem:refl_desc}, we use that any $*$-extension of a vector bundle on $\Spec \rA \setminus \V(p,\I)$ is a reflexive module over $\rA$ as $\rA$ is Noetherian regular. This is shown in Stacks Project \href{https://stacks.math.columbia.edu/tag/0EBJ}{0EBJ} (one may reduce to this situation). 

Next, it suffices to show for every reflexive prismatic $F$-crystal $\mathcal{E}$ that
\[\M:=\mathcal{E}(\rA,\I)\in \Vect^{\varphi}(\Spec \rA \setminus \V(p,\I)).\]
Indeed, by the descent results of Lemma \ref{lem:refl_desc} as well as the analogous results for analytic prismatic $F$-crystals shown in \cite{Zpcrystalline}, this allows us to verify $\mathcal{E}$ is an analytic prismatic $F$-crystal as it will automatically lie in $\Vect^{\varphi}(\Spec \rA^{(n)} \setminus \V(p,\I))$ for each transversal prism in the \v{C}ech nerve of $(\rA,\I)$. As $\X$ is assumed \'{e}tale over $\O_K \langle \X_1^\pm, \ldots, \X_n^\pm \rangle$, we know by Proposition \ref{prop:refl_VB} that $\M\in \Vect^{\varphi}(\Spec \rA \setminus \V(p,\I))$, as $\varphi^* \M[1/\I] \simeq \M[1/\I]$ and $\M|_{\Spec \rA \setminus \V(p,\I)}$ is a vector bundle.
\end{proof}

We can now characterize reflexive sheaves on $\X^{\Syn}$ as being precisely those lifted from analytic prismatic $F$-crystals via $\Pi_\X$.

\begin{lemma}
The functor $\Pi_\X$
\[\Vect^{\varphi,\mathrm{an}}(\X_{\Prism},\O_{\Prism})\to \Perf(\X^{\Syn})\]
has essential image contained in $\mathsf{Refl}(\X^{\Syn})$.
\label{lem:PiX_essim}
\end{lemma}

\begin{proof}
To check coherence, we may reduce to the affine case (the functor in Theorem \ref{thm:Fgaugelift} is compatible with \'{e}tale pullbacks). Then checking coherence amounts to verifying that the pullback to the stack $\mathrm{Rees}_{\I^\bullet} \rA \to \X^{\Nyg} \to \X^{\Syn}$ is coherent for a Breuil-Kisin prism $(\rA,\I)$\footnote{The notion of coherence in \cite{Fgaugelift} imposes that the pullback of $\mathcal{E}$ is discrete with an honest filtration in $\Perf(\mathrm{S}^{\Syn})$ for every quasiregular semiperfectoid ring with a quasisyntomic map $\Spf \mathrm{S} \to \X$, and our notion of coherence only obviously gives discreteness for quasiregular semiperfectoid rings related to the perfection of a Breuil-Kisin prism.}. Let $(\rA_{\mathrm{perf}},\I_{\mathrm{perf}})$ denote the perfection of this prism, so $\rA_{\mathrm{perf}}=(\colim_{\varphi} \rA)^{\wedge}_{(p,\I)}$. Since the Frobenius on $\rA$ is faithfully flat by \cite[Corollary 3.6]{Fgaugelift}, $\rA \to \colim_{\varphi} \rA$ is faithfully flat and $\rA \to \rA_{\mathrm{perf}}$ is derived $(p,\I)$-completely faithfully flat. 

We may argue the map of Rees stacks in this case is actually faithfully flat for their natural algebraizations since in
\[\mathrm{Rees}_{\I^\bullet} \rA_{\perf} \to \mathrm{Rees}_{\I^\bullet} \rA \to \X^{\Syn},\]
the (algebraization of) the target of the first map is a Noetherian algebraic stack. Indeed we may test this on the Rees algebra covers, where the claim we need to argue is that a $(p,\I)$-completely faithfully flat map $\R \to \R'$ of derived $(p,\I)$-complete rings $\R$ and $\R'$ is further faithfully flat if $\R$ is Noetherian. The flatness claim follows from \cite[Lemma 5.15]{Bhatt}, since the source (as a ring map) is Noetherian and derived $(p,\I)$-complete. It is further faithfully flat since maximal ideals of $\R$ contain $(p,\I)$ (since $(p,\I)$ lies in the Jacobson radical), and then since we know surjectivity of $\Spec \rA_{\perf}/(p,\I) \to \Spec \rA/(p,\I)$ this suffices.

The map of algebraizations of Rees stacks $\mathrm{Rees}_{\I^\bullet} \rA_{\perf} \to \mathrm{Rees}_{\I^\bullet} \rA$ is faithfully flat. It follows since the pullback to $\mathrm{Rees}_{\I^\bullet} \rA_{\perf}$ is discrete (i.e. in the heart for the $t$-structure on $\D_{\mathrm{qc}}(\mathrm{Rees}_{\I^\bullet} \rA_{\perf})$) that the pullback to $\mathrm{Rees}_{\I^\bullet} \rA$ lies in the heart and hence is coherent.

Having verified $\Pi_\X$ sends coherent prismatic $F$-crystals to coherent $F$-gauges, we may also easily check $\Pi_\X$ is compatible with duals. Since any object $\mathcal{E}\in \Vect^{\varphi,\mathrm{an}}(\X_{\Prism},\O_{\Prism})$ has $\mathcal{E} \simeq \mathcal{E}^{\vee \vee}$ by Theorem \ref{thm:prismFcris_reflexive} the same follows for $\Pi_{\X}(\mathcal{E})$.
\end{proof}

\begin{theorem}
Let $\X/\Spf \O_K$ be smooth and quasicompact. Then the functor $\Pi_\X$ induces an equivalence of categories
\[\Vect^{\varphi,\mathrm{an}}(\X_{\Prism},\O_{\Prism})\simeq \mathsf{Refl}(\X^{\Syn}).\]
\end{theorem}

\begin{proof}
What remains is to show essential surjectivity. We need to show any $\mathcal{E}\in \mathsf{Refl}(\X^{\Syn})$ is isomorphic to $\Pi_\X$ applied to an analytic prismatic $F$-crystal. It suffices to show that the natural map
\[\eta: \mathcal{E} \to \Pi_\X(\mathcal{E}|_{\X^{\Prism}})\]
is an equivalence: $\M=\mathcal{E}|_{\X^{\Prism}}$ must be a reflexive prismatic $F$-crystal and hence an analytic prismatic $F$-crystal by Theorem \ref{thm:prismFcris_reflexive}. If $\mathcal{E}$ is coherent in the sense of \cite{Fgaugelift} this is the unit of the adjunction, but the map exists regardless for $\mathcal{E}$ (which can be seen on quasiregular semiperfectoids mapping to $\X$ using the filtered Frobenius in the description provided by \cite[Example 6.1.7]{Fgauge}, essentially because $\Pi_\X$ equips its output with a saturated Nygaardian filtration).

Let us consider the case where $\X$ is smooth affine and integral. We can deduce the general affine case from this via Zariski descent using Stacks Project \href{https://stacks.math.columbia.edu/tag/0357}{0357}, and further Zariski descent would show the general case where $\X/\Spf \O_K$ is smooth. We may then write $\X = \Spf \rA/\I$ for a Breuil-Kisin prism $(\rA,\I)$ where $\rA$ is integral. To check $\eta$ is an equivalence, it suffices to check the induced map on the flat-local surjection $\mathrm{Rees}_{\I^\bullet}\rA$ is an equivalence. Let $\U$ be the complement of $\V(t,p)$ viewed as a substack of the algebraization $(\mathrm{Rees}_{\I^\bullet} \rA)^{\mathrm{alg}} := (\Spec \rA \langle u,t \rangle/(ut-d))/\G_m$ when $\I=(d)$. As we are dealing with coherent sheaves on Rees stacks and $\rA$ is an adic Noetherian ring, $\Coh(\mathrm{Rees}_{\I^\bullet} \rA) \simeq \Coh((\mathrm{Rees}_{\I^\bullet} \rA)^{\mathrm{alg}})$.

The map $\eta$ induces an equivalence when restricted to the $t\neq 0$ substack which covers $\X^{\Prism}$ (as the underlying prismatic $F$-crystal of $\Pi_\X(\M)$ is $\M$). The map $\eta: \mathcal{E} \to \Pi_\X(\mathcal{E}|_{\X^{\Prism}})$ induces the identity on \'{e}tale realization, so the exact sequence
\[0 \to \ker(\eta) \to \mathcal{E} \to \Pi_\X(\mathcal{E}|_{\X^{\Prism}}) \to \coker \eta \to 0\]
is sent to $0 \to \T_\et(\mathcal{E}) \to \T_\et(\mathcal{E}) \to 0$. By Theorem \ref{thm:general_Tet_kernel}, the kernel and cokernel of $\eta$ are $p$-power torsion, so $\eta$ is also an equivalence on the $p\neq 0$ locus in the algebraization and hence on $\U$. The complement of $\U$ as a substack has everywhere codimension two, as the complement is
\[(\Spec \Gamma(\X_s,\O) \llbracket u_0\rrbracket [u]/u_0^n)/\G_m\]
where $u$ has degree $-1$ and $\deg \E(u_0)=n$.

Passing to the Rees algebra smooth cover $\Spec \rA \langle u,t \rangle/(ut-d)$ of the Rees stack, the hypotheses of Stacks Project \href{https://stacks.math.columbia.edu/tag/0EBJ}{0EBJ} are satisfied and tell us the pullback to the open $\D(t) \cup \D(p)$ over $\U$ on this cover is fully faithful on reflexive modules (hence conservative). Since $\eta$ induces an equivalence on this open, by descent we deduce that $\eta$ is an equivalence.
\end{proof}

\begin{theorem}
Let $\X/\Spf \O_K$ be smooth and quasicompact. The functor $\T_\et$ induces an equivalence
\[\mathsf{Refl}(\X^{\Syn}) \simeq \Loc_{\Z_p}^\cris(\X_\eta).\]
\label{thm:generalsmoothTetcrys}
Moreover, if $\mathcal{E}\in \Perf(\X^{\Syn})$ then $\T_\et(\H^i(\mathcal{E}))[1/p]$ is crystalline.
\end{theorem}

\begin{proof}
Apply the main result from \cite{Zpcrystalline} that $\T^{\Prism}_\et$ induces an equivalence from analytic prismatic $F$-crystals to $\Loc_{\Z_p}^\cris(\X_\eta)$. The \'{e}tale realization for $F$-gauges factors as
\begin{display}
\mathsf{Refl}(\X^{\Syn}) \ar{r}{(-)|_{\X^{\Prism}}} &  \Vect^{\varphi,\mathrm{an}}(\X_{\Prism},\O_{\Prism}) \ar{r}{\sim} & \Loc_{\Z_p}^\cris(\X_\eta)
\end{display}
and we know now that $(-)|_{\X^{\Prism}}$ has an inverse $\Pi_\X$ by the previous theorem.

For the second claim, as $\H^i(\mathcal{E})$ is coherent it suffices to observe that on a smooth affine for a coherent $F$-gauge $\mathcal{E}$ we have
\[\T_\et(\mathcal{E})[1/p] \simeq \T_\et(\mathcal{E}^{\vee \vee})[1/p]\]
as the \'{e}tale realization is compatible with duals, and $\Q_p$-local systems are canonically isomorphic to their double dual. As $\mathcal{E}^{\vee \vee}$ is reflexive (the canonical map to the double dual being an isomorphism can be checked on the Noetherian regular cover of $\X^{\Syn}$ given by a covering family of Breuil-Kisin prisms via Proposition \ref{prop:BK_flatcover}), the claim follows.
\end{proof}

We remark that the first equivalence $\mathsf{Refl}(\X^{\Syn}) \simeq \Loc_{\Z_p}^\cris(\X_\eta)$ does not use that $\X$ is quasicompact; this is only used for the isogeny part of the statement.

\section{Perfect complexes up to isogeny on \texorpdfstring{$\X^{\Syn}$}{XSyn}}

\subsection{The essential image of \texorpdfstring{$\T_{\et}$}{Tet}}

We can now apply our results to study the essential image of $\T_\et[1/p]$ on $\Perf(\X^{\Syn})[1/p]$ when $\X/\Spf \O_K$ is smooth and quasicompact.

\begin{corollary}
The $t$-exact functor
\begin{display}
\Perf(\X^{\Syn})[1/p] \ar{r}{\T_\et[1/p]} & \D^{(b)}_{\mathrm{lisse}}(\X_\eta,\Z_p)[1/p]
\end{display}
induces an equivalence $\Coh(\X^{\Syn})[1/p]\simeq \Loc_{\Q_p}^\cris(\X_\eta)$ on the heart. The essential image of $\T_\et[1/p]$ contains the essential image of 
\[\D^b(\Loc_{\Q_p}^\cris(\X_\eta)) \to \D^{(b)}_{\mathrm{lisse}}(\X_\eta,\Z_p)[1/p],\]
and is contained in the full subcategory of $\D^{(b)}_{\mathrm{lisse}}(\X_\eta,\Z_p)[1/p]$ where every cohomology sheaf is crystalline.
\label{cor:Perf_cris}
\end{corollary}

\begin{proof}
The $t$-exactness of $\T_\et$ follows from Lemma \ref{lem:Tet_texact}.

We now show $\T_\et[1/p]$ induces an equivalence on the heart onto the subcategory of crystalline $\Q_p$-local systems, which by Zariski descent of $\Coh(\X^{\Syn})$ and $\Loc_{\Z_p}^\cris(\X_\eta)$ can be tested on a smooth affine (here we also use quasicompactness to ensure the covering is finite). The functor $\T_\et$ has the desired essential image on the heart using Theorem \ref{thm:generalsmoothTetcrys}. For full faithfulness, we have a map
\[\T_\et: \Hom_{\Coh(\X^{\Syn})}(\mathcal{E},\mathcal{E}')[1/p] \to \Hom(\T_\et(\mathcal{E})[1/p], \T_\et(\mathcal{E}')[1/p]).\]
We use that the map $\mathcal{E} \to \mathcal{E}^{\vee \vee}$ is an equivalence in the isogeny category $\Coh(\X^{\Syn})[1/p]$. Indeed, after \'{e}tale realization and inverting $p$ it becomes an equivalence, because an isogeny $\Z_p$-local system is isomorphic to its double dual. Thus the kernel and cokernel of the canonical map $\mathcal{E} \to \mathcal{E}^{\vee \vee}$ are killed by $\T_\et[1/p]$. Applying Proposition \ref{prop:p_tors_Tet} and the proof of Theorem \ref{thm:general_Tet_kernel} which upgrades the statement to the $F$-gauge being $p$-power torsion, the kernel and cokernel must be $0$ in $\Coh(\X^{\Syn})[1/p]$. Thus we can always assume both coherent $F$-gauges are reflexive. But there we already know $\T_\et$ is an equivalence by Theorem \ref{thm:generalsmoothTetcrys}.

The claim that the essential image is contained in the subcategory of $\D^{(b)}_{\mathrm{lisse}}(\X_\eta,\Z_p)[1/p]$ where every cohomology sheaf is crystalline follows immediately from the claim on the heart by $t$-exactness. For producing objects in the essential image, we use the commutative diagram
\begin{display}
\D^b(\Coh(\X^{\Syn}))[1/p] \ar{r}{\simeq} \ar{d} & \D^b(\Loc_{\Q_p}^\cris(\X_\eta)) \ar{d} \\
\Perf(\X^{\Syn})[1/p] \ar{r}{\T_\et[1/p]} & \D^{(b)}_{\mathrm{lisse}}(\X_\eta,\Z_p)[1/p].
\end{display}
The top arrow is an equivalence coming from the result on the heart and $\D^b(\Coh(\X^{\Syn}))[1/p] \simeq \D^b(\Coh(\X^{\Syn})[1/p])$, which is true since we work on bounded derived categories. Commutativity of the diagram is easy to verify, as both composite functors to $\D^{(b)}_{\mathrm{lisse}}(\X_\eta,\Z_p)[1/p]$ agree on the heart and commute with shifts and fibers (so using \cite[Lemma 5.4.3]{Hauck} we can deduce they agree).

Finally, the proposed functor $\D^b(\Coh(\X^{\Syn})) \to \Perf(\X^{\Syn})$ a priori might not land in $\Perf(\X^{\Syn})$ but rather $\D^b_{\Coh}(\X^{\Syn})$, the full subcategory of locally bounded objects in $\D_{\mathrm{qc}}(\X^{\Syn})$ where all cohomology sheaves are coherent. However this coincides with $\Perf(\X^{\Syn})$: since $\X$ is smooth there is a locally Noetherian regular flat-local surjection $\rho$ such that pullback to the covering is $t$-exact on $\D_{\mathrm{qc}}(-)$. The $t$-exactness follows from Remark \ref{rem:t-exact}. Then given $\mathcal{E}\in \D_{\mathrm{qc}}(\X^{\Syn})$, belonging to $\D^b_{\Coh}(\X^{\Syn})$ means $\H^i(\rho^* \mathcal{E})$ lands in $\Coh$ and the complex is locally bounded. But on the cover $\D^b_{\Coh}=\Perf$ and we may test membership in $\Perf$ after pullback along a flat-local surjection by descent. Thus the functor is well-defined. We deduce the essential image of $\T_\et[1/p]$ contains the essential image of $\D^b(\Loc_{\Q_p}^\cris(\X_\eta))$ in $\D^{(b)}_{\mathrm{lisse}}(\X_\eta,\Z_p)[1/p]$.
\end{proof}

\begin{remark}
All of the inclusions in Corollary \ref{cor:Perf_cris} are strict in general, although we will later find for a point $\X=\Spf \O_K$ the essential image coincides with the essential image of $\D^b(\Rep_{\Q_p}^\cris(\rG_K))$ in $\D^{(b)}_{\mathrm{lisse}}(\Spa K, \Z_p)[1/p]$.

In general there are objects in $\D^{(b)}_{\mathrm{lisse}}(\X_\eta,\Z_p)[1/p]$ which have crystalline cohomologies but are not in the essential image of $\T_\et[1/p]$. For example, on a point $\X_\eta=\Spa K$, on the target $\D^{(b)}_{\mathrm{lisse}}(\Spa K, \Z_p)[1/p]$ there exists a nonzero map $\Q_p \to \Q_p(1)[2]$ using that the second rationalized \'{e}tale cohomology of $\Z_p(1)$ is $\Q_p$. Taking the fiber of this map, we obtain a complex with crystalline cohomology sheaves which cannot be obtained from $\T_\et$ (one may see this by computing $\H^2_{\Syn}(\O_K,\O\{1\})[1/p]=0$ so there is no corresponding map of $F$-gauges and using full faithfulness on the heart after inverting $p$). However, as we will see in Proposition \ref{prop:O_K_equiv} in the case of a point the essential image is exactly the essential image of $\D^b(\Rep_{\Q_p}^\cris(\rG_K))$ in $\D^{(b)}_{\mathrm{lisse}}(\Spa K, \Z_p)[1/p]$.

There are also objects in the essential image of $\T_\et[1/p]$, but not contained in the essential image of $\D^b(\Loc_{\Q_p}^\cris(\X_\eta)) \to \D^{(b)}_{\mathrm{lisse}}(\X_\eta,\Z_p)[1/p]$ when $\X=\P^1_{\Z_p}$. Using Example \ref{ex:derived_counter} we can produce a map $f: \O \to \O\{1\}[2]$ in $\Perf(\X^{\Syn})$ so that $\T_\et(\mathrm{fib}(f))[1/p]$ fails to come from $\D^b(\Loc_{\Q_p}^\cris(\X_\eta))$ as $\Ext^2$ for objects in the heart vanishes for this category when $\X=\P^1_{\Z_p}$.
\label{rem:ess_im}
\end{remark}

We may also use this to identify $\Ext^1$ groups in crystalline $\Q_p$-local systems with rationalized syntomic cohomology.

\begin{proposition}
Assume $\X$ is smooth and quasicompact over $\Spf \O_K$ and let $\mathcal{E}\in \Coh(\X^{\Syn})$. Then the \'{e}tale realization induces an isomorphism
\[\H^1_{\Syn}(\X,\mathcal{E})[1/p] \simeq \Ext^1_{\Loc_{\Q_p}^\cris(\X_\eta)}(\Q_p,\T_\et(\mathcal{E})[1/p]).\]
In particular, if $\X=\Spf \O_K$ then $\H^1_{\Syn}(\X,\mathcal{E})[1/p] \simeq \H^1_f(\rG_K,\T_\et(\mathcal{E})[1/p])$ where $\H^1_f$ denotes the Bloch-Kato Selmer group.
\label{prop:ext1_cohom}
\end{proposition}
	
\begin{proof}
Due to $\T_\et$ inducing an equivalence $\Coh(\X^{\Syn})[1/p] \simeq \Loc_{\Q_p}^\cris(\X_\eta)$ as shown in Corollary \ref{cor:Perf_cris}, we immediately have
\[\Ext^i_{\Coh(\X^{\Syn})[1/p]}(\mathcal{E}[1/p],\mathcal{E}'[1/p]) \simeq \Ext^i_{\Loc_{\Q_p}^\cris(\X_\eta)}(\T_\et(\mathcal{E})[1/p],\T_\et(\mathcal{E}')[1/p])\]
for coherent $F$-gauges $\mathcal{E}$ and $\mathcal{E}'$, as an equivalence of abelian categories gives an equivalence of their Yoneda $\Ext$ groups.
	
It then suffices to prove the stronger integral claim that
\[\H^1_{\Syn}(\X,\mathcal{E}) \simeq \Ext^1_{\Coh(\X^{\Syn})}(\O,\mathcal{E}).\]
We will first argue
\[\H^1_{\Syn}(\X,\mathcal{E}) \simeq \Ext^1_{\QCoh(\X^{\Syn})}(\O,\mathcal{E})\]
where the $\Ext^1$ group here is a Yoneda $\Ext$ group (so it makes sense in any abelian category). This is easy to show in general, since we may identify $\H^1_{\Syn}(\X,\mathcal{E}) \simeq \Hom_{\D_{\mathrm{qc}}(\X^{\Syn})}(\O,\mathcal{E}[1])$ and then the latter classifies extensions: we map an extension to the induced connecting morphism, and a map $\O \to \mathcal{E}[1]$ to its fiber.

Given an extension in $\QCoh(\X^{\Syn})$ of the form
\begin{display}
0 \ar{r} & \mathcal{F} \ar{r} & \mathcal{G} \ar{r} & \mathcal{H} \ar{r} & 0
\end{display}
where $\mathcal{F}$ and $\mathcal{H}$ are coherent, we know that $\mathcal{G}$ is coherent. We can test this after pullback to a regular Noetherian flat cover, at which point it follows from the 2 out of 3 property for exact sequences of coherent sheaves. This then means $\H^1_{\Syn}(\X,\mathcal{E}) \simeq \Ext^1_{\Coh(\X^{\Syn})}(\O,\mathcal{E})$.
\end{proof}

Next, we will study what happens to both sides of Corollary \ref{cor:Perf_cris} under smooth proper pushforwards. We will need the following lemma to deal with pullbacks.

\begin{lemma}
The pullbacks $j^*_{\Nyg}$ and $j^*_{\Prism}$ on $\D_{\mathrm{qc}}((-)^{\Syn})$ commute with pushforwards of $F$-gauges along maps of qcqs quasisyntomic $p$-adic formal schemes.
\label{lem:realization_commute}
\end{lemma}

\begin{proof}
First consider the commutative diagram
\begin{display}
\X^{\Prism} \ar{r}{j_{\Prism}} \ar{d}{f} & \X^{\Syn} \ar{d}{f} \\
\Y^{\Prism} \ar{r}{j_{\Prism}} & \Y^{\Syn}
\end{display}
where $j_{\Prism}$ is an open immersion. We want to know the canonical natural transformation $j_{\Prism}^* \R f_* \to \R f_* j_{\Prism}^*$ is an equivalence when we have coefficients in the derived category of quasicoherent sheaves (see Definition A.20 in \cite{PrismaticDelta}).
	
The claim that $j_{\Prism}^* \R f_* \simeq \R f_* j_{\Prism}^*$ follows purely formally, by applying \cite[Lemma A.28]{PrismaticDelta} on the relevant stacks after presenting them as colimits using quasiregular semiperfectoid rings. The qcqs condition ensures there exists a hypercovering which degreewise consists of finite disjoint unions of $\Spf \R_i$ for $\R_i$ quasiregular semiperfectoid; we use this for the colimit presentation. Note that we can use hypercoverings in our setting as our stacks are 1-truncated.

This reduces us to showing base change for the diagram
\begin{display}
\Spf {\Prism}_\R \ar{r}{j_{\Prism}} \ar{d}{f} & \R^{\Syn} \ar{d}{f} \\
\Spf {\Prism}_{\R'} \ar{r}{j_{\Prism}} & (\R')^{\Syn}
\end{display}
for a map of quasiregular semiperfectoid rings $\Spf \R \to \Spf \R'$ (the edge base change squares required in this case are trivial to verify, reducible to the quasiregular semiperfectoid case by the choice of colimit presentation). The analogous square with the Nygaard stack replacing $\R^{\Syn}$ and $(\R')^{\Syn}$ has base change for both of the open immersions $j_{\HT}$ and $j_{\dR}$, so again applying \cite[Lemma A.28]{PrismaticDelta} for the pushout used to produce $\X^{\Syn}$ shows commutativity of the left square.

An identical argument shows the claim for $j^*_{\Nyg}$. We reduce to base change for the diagram
\begin{display}
\R^{\Nyg} \ar{r}{j_{\Nyg}} \ar{d}{f} & \R^{\Syn} \ar{d}{f} \\
(\R')^{\Nyg} \ar{r}{j_{\Nyg}} & (\R')^{\Syn}
\end{display}
which we may check in the same way.
\end{proof}

We then have the following (already known) pushforward stability result. This appears in \cite[Proposition 8.2.6]{madapusi2025perfect}; below we explain some of the argument.

\begin{proposition}[\cite{madapusi2025perfect}, Proposition 8.2.6]
Let $f: \X \to \Y$ be a smooth proper morphism of smooth $p$-adic formal schemes over $\Spf \O_K$. Then if $\mathcal{E}\in \Perf(\X^{\Syn})$, we have $\R f_* \mathcal{E} \in \Perf(\Y^{\Syn})$.
\label{prop:sm_prop_pushforward}
\end{proposition}

\begin{proof}
It suffices to check this for $(-)^{\Nyg}$ using flat descent of $\Perf$, as $\Y^{\Nyg}$ is an \'{e}tale cover of $\Y^{\Syn}$ and similarly for $\X^{\Nyg}$ and $\X^{\Syn}$ (see the discussion around \cite[Definition 6.1.1]{Fgauge}). Then we may use that $j^*_{\Nyg}$ commutes with pushforwards as verified in the previous lemma to reduce to checking $\R f_*$ sends perfect complexes on $\X^{\Nyg}$ to perfect complexes on $\Y^{\Nyg}$ for smooth proper morphisms $f$. There is a stratification of $\X^{\Nyg}$ by $\X^{\Nyg}_{t=0}$ and $\X^{\Nyg}_{t\neq 0}$ (as well as $\Y^{\Nyg}$).

Now we observe that it suffices to check perfectness on the $t=0$ component of $\Y^{\Nyg}$ and derived $t$-completeness. One may also use formal gluing for this stratification to make the reduction, by checking perfectness is also preserved for the pushforward $\X^{\Nyg}_{t\neq 0} \simeq \X^{\Prism} \to \Y^{\Nyg}_{t\neq 0} \simeq \Y^{\Prism}$, which is well-known (see \cite[Corollary 5.16]{Zpcrystalline}). The pushforward $\R f_* \mathcal{E}$ will still be complete for the Nygaard filtration as one can reduce to the same assertion for Hodge-filtered de Rham cohomology in the smooth proper case by the same method as \cite[Proposition 5.8.2]{APC}. The claim for the $t=0$ component is shown in \cite{madapusi2025perfect}, via reduction to the case where $\Y$ is semiperfectoid and then reducing the claim to smooth proper pushforwards of schemes preserving perfect complexes.
\end{proof}

This then formally implies the following result.

\begin{corollary}
Suppose $\X$ and $\Y$ are smooth and quasicompact over $\Spf \O_K$ and let $\mathbb{L}\in \D^{(b)}_{\mathrm{lisse}}(\X_\eta,\Z_p)[1/p]$ be in the essential image of $\T_\et[1/p]$ (as partially characterized in Corollary \ref{cor:Perf_cris}). Then if $f: \X \to \Y$ is smooth proper, $\R^i f_{\eta,*} \mathbb{L}$ is crystalline for all $i$.
\label{cor:der_C_cris}
\end{corollary}

\begin{proof}
It suffices to show that $\T_\et$ commutes with smooth proper pushforwards in light of Corollary \ref{cor:Perf_cris}. In the diagram
\begin{display}
\Perf(\X^{\Syn}) \ar{d}{\R f_*} \ar{r}{(-)|_{\X^{\Prism}}} & \Perf^{\varphi}(\X_{\Prism},\O_{\Prism}) \ar{d}{\R f_*} \ar{r} & \D^{(b)}_{\mathrm{lisse}}(\X_\eta,\Z_p) \ar{d}{\R f_{\eta,*}} \\
\Perf(\Y^{\Syn}) \ar{r}{(-)|_{\Y^{\Prism}}} & \Perf^{\varphi}(\Y_{\Prism},\O_{\Prism}) \ar{r} & \D^{(b)}_{\mathrm{lisse}}(\Y_\eta,\Z_p)
\end{display}
we know the right square commutes by \cite[Theorem 5.8]{Fgaugelift}, so it suffices to verify the left square commutes and $\R f_*$ preserves perfect $F$-gauges. The latter holds by Proposition \ref{prop:sm_prop_pushforward}.

Now we check commutativity of the left square. For prismatic $F$-crystals, by an inspection of the canonical prismatic $F$-crystal structure given to $j^*_{\Prism}(\mathcal{E})$ in Lemma \ref{lem:restriction_Fcris} it suffices to know $j^*_{\Nyg}$ and $j^*_{\Prism}$ commute with pushforwards, which we have already shown in Lemma \ref{lem:realization_commute}.
\end{proof}

The result that $(\R^i f_{\eta,*} \mathbb{L})[1/p]$ is crystalline when $\mathbb{L}\in \Loc_{\Z_p}^\cris(\X_\eta)$ and $f$ is smooth proper is given by Theorem B of \cite{Zpcrystalline} (the isocrystal association portion of the statement is verified later in Proposition \ref{prop:Dcris_compat}). We can regard this as strengthening the result to allow more general $\mathbb{L}$ in the derived category. When $\Y=\Spf \O_K$, by Proposition \ref{prop:O_K_equiv} we actually learn about the complex beyond its cohomology, namely that it lies in the essential image of $\D^b(\Rep_{\Q_p}^\cris(\rG_K)) \to \D^{(b)}_{\mathrm{lisse}}(\Spa K, \Z_p)[1/p]$.

\subsection{Admissible filtered $F$-isocrystals in the proper case}

In \cite[Proposition 5.6.2]{Hauck}, Corollary \ref{cor:Perf_cris} is refined to an equivalence in the case of $\X=\Spf \Z_p$. In particular, it is shown that $\T_\et$ induces an equivalence
\[\Perf(\Z_p^{\Syn})[1/p] \simeq \D^b(\Rep_{\Q_p}^\cris(\rG_{\Q_p})).\]
Such an equivalence cannot hold in general, as the following example shows.

\begin{example}
Let $\X$ be the $p$-adic formal scheme $\P^1_{\Z_p}$. If we had an equivalence
\[\Perf(\X^{\Syn})[1/p] \simeq \D^b(\Loc_{\Q_p}^\cris((\P^1_{\Q_p})^{\mathrm{ad}}))\]
then considering $\R\Hom_{(\P^1_{\Z_p})^{\Syn}}(\O,\O\{1\})[1/p] \simeq \R\Gamma_{\Syn}(\P^1_{\Z_p},\O\{1\})[1/p]$ we would need to show that
\[\H^i_{\Syn}(\P^1_{\Z_p},\O\{1\})[1/p] \simeq \Ext^i_{\Loc_{\Q_p}^\cris((\P^1_{\Q_p})^{\mathrm{ad}})}(\Q_p,\Q_p(1)).\]
When $i=2$, the right hand side must vanish, since $\Ext^2$ vanishes in crystalline $\Q_p$-local systems. Indeed, since $(\P^1_{\Q_p})^{\mathrm{ad}}$ has an \'{e}tale fundamental group isomorphic to $\rG_{\Q_p}$, the subcategory of crystalline local systems is equivalent to the category of crystalline Galois representations of $\rG_{\Q_p}$. In this abelian category, $\Ext^2$ will vanish (by \cite{emerton1999extensions}; see also \cite{Hauck} Remark 5.2.5).

However the left hand side (syntomic cohomology) is nonzero when $i=2$, as we get a nonzero class coming from the syntomic Chern class (see Variant 7.5.4 in \cite{APC}).
\label{ex:derived_counter}
\end{example}

In this subsection we will first show that despite this failure in the higher dimensional case, Hauck's result still extends to the case where $K/\Q_p$ is an arbitrary finite extension:
\[\Perf(\O_K^{\Syn})[1/p] \simeq \D^b(\Rep_{\Q_p}^\cris(\rG_{K})).\]
We will then define a category $\Perf^{\mathrm{adm}}_{\mathrm{fIsoc}^{\varphi}}(\X)$ of admissible filtered $F$-isocrystals in perfect complexes and use the case of $\O_K$ to show that there is a fully faithful functor
\[\Perf(\X^{\Syn})[1/p] \to \Perf^{\mathrm{adm}}_{\mathrm{fIsoc}^{\varphi}}(\X)\]
when $\X/\Spf \O_K$ is smooth proper, and then use \S 4.1 to see the functor is essentially surjective.

The key input in the case of a point is an analogue of the cohomology computation done in \cite[Proposition 6.7.3]{Fgauge} for arbitrary finite extensions $K/\Q_p$.

\begin{proposition}
Let $K/\Q_p$ be a finite extension. Then for $\mathcal{E}\in \Coh(\O_K^{\Syn})$ we have
\[\H^i(\O_K^{\Syn},\mathcal{E})[1/p] \simeq \Ext^i_{\Rep_{\Q_p}^\cris(\rG_K)}(\Q_p,\T_\et(\mathcal{E})[1/p])\]
where the $0$th extension group on the right is $\H^0(\rG_K,(-)[1/p])$, the first is given by 
\[\H^1_f(\rG_K,-) = \Ext^1_{\Rep_{\Q_p}^\cris(\rG_K)}(\Q_p,(-)[1/p]),\]
and both sides vanish for $i\ge 2$.
\label{prop:O_K_cohom}
\end{proposition}

Assuming the conjectured Lagrangian refinement of Tate duality for regular formal schemes, see \cite[Conjecture 6.5.25]{Fgauge}, one may deduce this using the same strategy as \cite[Proposition 6.7.3]{Fgauge}. If $K$ is unramified, since $\Spf \W(k)/\Spf \Z_p$ is finite \'{e}tale and in particular smooth proper we may use geometric Poincar\'{e} duality to deduce the claim (see the remark following \cite[Conjecture 6.5.25]{Fgauge}).

We will use a different approach to prove the claim by adapting the tools of \cite{KtheoryZpn} to work with coefficients. As a byproduct, we will produce a complex computing syntomic cohomology using Breuil-Kisin modules.

\textbf{Notation}. In what follows we will abbreviate absolute prismatic cohomology $\R\Gamma(\O_K^{\Prism},\mathcal{E})$ as $\Prism_{\O_K}(\mathcal{E})$ and the relative prismatic cohomology $\R\Gamma((\O_K/\rA)^{\Prism},\mathcal{E})$ as $\Prism_{\O_K/\rA}(\mathcal{E})$ for a prismatic crystal $\mathcal{E}$. 

Using \cite[Theorem 1.2(6)]{PrismaticDelta} with coefficients (which follows from the identification of stacks in Corollary \ref{cor:Nyg_relqrsp_desc} after restricting to the $t\neq 0$ locus), we may compute
\[\Prism_{\O_K}(\mathcal{E}) \simeq \Prism_{\O_K/\W(k)}(\mathcal{E})\]
as
\[\mathrm{Tot} \left(\begin{tikzcd} \Prism_{\O_K/\W(k)\llbracket u_0\rrbracket}(\mathcal{E}) \ar[shift left]{r} \ar[shift right]{r} & \ar{l} \Prism_{\O_K/\W(k)\llbracket u_0,u_1\rrbracket}(\mathcal{E}) \ar[shift right, shift right]{r} \ar{r} \ar[shift left, shift left]{r} & \ar[shift right]{l} \ar[shift left]{l} \cdots \end{tikzcd}\right)\]
giving a descent complex
\begin{display}
\Prism_{\O_K/\W(k)\llbracket u_0\rrbracket}(\mathcal{E}) \ar{r}{\mathrm{d}^1} & \Prism_{\O_K/\W(k)\llbracket u_0,u_1\rrbracket}(\mathcal{E}) \ar{r}{\mathrm{d}^2} & \ldots
\end{display}
The analogous descent complex also computes Nygaard filtered prismatic cohomology after endowing each $\Prism^{(1)}_{\O_K/\W(k)\llbracket u_0,\ldots, u_i\rrbracket}(\mathcal{E})$ with the na\"{i}ve Nygaard filtration for $\I=\E(u_0)$ (as $\O_K$ is relatively quasisyntomic in the sense of \cite{PrismaticDelta}). We can use Corollary \ref{cor:Nyg_relqrsp_desc} to formally justify this construction with coefficients. We can use this along with a result of Gao-Min-Wang to obtain a shorter description of prismatic cohomology.

\begin{lemma}
\label{lem:trunc_desc_complex}
Let $\mathcal{E}\in \Vect(\O_K^{\Prism})$. Then
\[\R\Gamma(\O_K^{\Prism},\mathcal{E}) \simeq \left[\Prism_{\O_K/\W(k)\llbracket u_0\rrbracket}(\mathcal{E}) \to K_2\right],\]
where $K_2 = \ker(\mathrm{d}^2: \Prism_{\O_K/\W(k)\llbracket u_0,u_1\rrbracket}(\mathcal{E}) \to \Prism_{\O_K/\W(k)\llbracket u_0,u_1,u_2\rrbracket}(\mathcal{E}))$ is the kernel of the second differential in the descent complex.
\end{lemma}

\begin{proof}
This follows from the result of Gao-Min-Wang in \cite[Theorem 4.10]{GaoMinWang} that $\R\Gamma(\O_K^{\Prism},\mathcal{E})$ lies in cohomological degrees 0 and 1.
\end{proof}

In what follows, all filtrations $\mathrm{F}^\ast$ on rings have the property that $\mathrm{F}^0 \rA = \mathrm{F}^{-n} \rA$ for $n\ge 0$ and $\mathrm{F}^0 \rA = \rA$. A filtered $\delta$-pair consists of a filtered $\delta$-ring $\mathrm{F}^\ast \rA$ (one requires $\delta(\mathrm{F}^n \rA) \subseteq \mathrm{F}^{pn} \rA$) and a map of filtered rings $\mathrm{F}^\ast \rA \to \mathrm{F}^\ast \R$. In \cite{PrismaticDelta} \S 10 a filtered variant of prismatic cohomology $\mathrm{F}^\ast \Prism_{\R/\rA}$ is given relative to filtered $\delta$-pairs $\mathrm{F}^\ast \rA \to \mathrm{F}^\ast \R$. From now on we give the Breuil-Kisin prism $\W(k)\llbracket u_0 \rrbracket$ the $u_0$-adic filtration, and $\O_K$ the $\pi$-adic filtration for a uniformizer $\pi$. For a filtered $\delta$-pair $\mathrm{F}^\ast \rA \to \mathrm{F}^\ast \R$, the relative prismatic cohomology inherits a filtration $\mathrm{F}^\ast \Prism_{\R/\rA}$. With coefficients in a prismatic crystal $\mathcal{E}$, we let $\mathrm{F}^\ast \Prism_{\R/\rA}(\mathcal{E})$ be induced from the filtration $\mathrm{F}^\ast \Prism_{\R/\rA}$ (the flat filtration in the terminology of \cite{PrismaticDelta}). Write $\mathrm{F}^\ast \widehat{\Prism}_{\R/\rA}$ for the completion with respect to this filtration.

We will need a $\W(k)\llbracket u_0\rrbracket^{(1)}$-linear map
\[\Theta: \mathrm{F}^{*} \varphi^* \widehat{\Prism}_{\O_K/\W(k)\llbracket u_0,u_1\rrbracket}(\mathcal{E}) \to \mathrm{F}^{*-1} \varphi^* \widehat{\Prism}_{\O_K/\W(k)\llbracket u_0\rrbracket}(\mathcal{E}\{-1\})\]
as well as a Nygaard filtered version, which shifts the Nygaard filtration down by one. In \cite{KtheoryZpn} this is constructed for $\mathcal{E}=\O_{\Prism}$, which induces the general map. It is helpful to understand this construction via the perspective of comodules over a Hopf algebroid, which we now explain.

In \S 4.6 of \cite{KtheoryZpn} they view prismatic crystals as right comodules over a Hopf algebroid $(\Gamma_0,\Gamma_1)=(\mathrm{F}^\ast \widehat{\Prism}_{\O_K/\W(k)\llbracket u_0 \rrbracket}, \mathrm{F}^\ast \widehat{\Prism}_{\O_K/\W(k)\llbracket u_0, u_1 \rrbracket})$ corresponding to the truncated $\mathrm{F}$-completed descent complex. This is only used for Breuil-Kisin twists, but the same general construction applies verbatim in our setting for prismatic crystals in vector bundles. There is of course also a Frobenius twisted variant, instead using the Hopf algebroid 
\[(\mathrm{F}^\ast \varphi^* \widehat{\Prism}_{\O_K/\W(k)\llbracket u_0 \rrbracket}, \mathrm{F}^\ast \varphi^* \widehat{\Prism}_{\O_K/\W(k)\llbracket u_0, u_1 \rrbracket}).\] 
By using the $g_u$ basis of $\Gamma_1$ constructed in \cite[Proposition 3.34]{KtheoryZpn}, they give a map $\Theta: \Gamma_1 \to \Gamma_0$ in \cite[Definition 4.41]{KtheoryZpn} in the Frobenius twisted case (as well as an untwisted variant). This induces a connection $\nabla_\Theta$ on an arbitrary right comodule over this Hopf algebroid via the setup of \S 4.6 as the composite $\begin{tikzcd} \M \ar{r}{\rho} & \M\otimes_{\Gamma_0} \Gamma_1 \ar{r}{1 \otimes \Theta} & \M \end{tikzcd}$ where $\rho$ is the coaction and the second map identifies with the $\Theta$ we needed previously in the Frobenius twisted case (we trivialize the Breuil-Kisin twist for convenience by picking a generator $\E(u_0)$ for the prismatic ideal). We may also use the same argument as in \cite[Lemma 4.36]{KtheoryZpn} to identify the descent complex of a prismatic crystal in vector bundles with the cobar complex of the associated comodule.

There is also a Nygaard filtered version, as by virtue of Corollary \ref{cor:Nyg_relqrsp_desc} we may view sheaves on the stack $\O_K^{\Nyg}$ as Nygaard filtered right comodules for the Frobenius twisted variant $(\Gamma_0^{(1)}, \Gamma_1^{(1)})$; the map $\Theta$ is Nygaard filtered, so it upgrades to a map of Rees stacks once we apply $\mathrm{F}$-completion to each Rees stack. We will see later $\mathrm{F}$-completion is automatic for $\Gamma_0$ which is useful (in particular allowing $\nabla_\M$ to exist before completion), but we must still involve it in the construction as it is necessary for $\Theta$ to be defined. Thus to describe coefficients for Nygaard filtered $\mathrm{F}$-completed prismatic cohomology as right comodules we apply $\mathrm{F}$-completion to each Rees stack used in the descent complex to obtain an $\mathrm{F}$-completed Nygaard filtered variant. The same remark about the cobar complex and descent complex matching again holds, now with Nygaard filtrations.

By the previous discussion we already have a connection $\nabla_{\mathcal{E}}$ on $\varphi^* \mathrm{F}^{*} \widehat{\Prism}_{\O_K/\W(k)\llbracket u_0\rrbracket}(\mathcal{E})$. There is also an equivalent formulation using the descent complex, which will be convenient to have for proofs.

\begin{definition}[\cite{KtheoryZpn}]
Define 
\[\nabla_{\mathcal{E}}: \mathrm{F}^* \widehat{\Prism}_{\O_K/\W(k)\llbracket u_0\rrbracket}(\mathcal{E}) \to \mathrm{F}^{*-1} \widehat{\Prism}_{\O_K/\W(k)\llbracket u_0\rrbracket}(\mathcal{E}\{-1\})\]
to be the composition of $\Theta$ and the first differential of the descent complex, and similarly we may define Frobenius twisted and Nygaard filtered variants (recall we use Corollary \ref{cor:Nyg_relqrsp_desc} to make sense of coefficients). We will trivialize this Breuil-Kisin twist by picking a generator $\E(u_0)$ for the prismatic ideal in $\W(k)\llbracket u_0 \rrbracket$.
\label{def:nabla}
\end{definition}

Using \cite[Proposition 10.43(c)]{PrismaticDelta} we deduce the source and target of $\nabla_{\mathcal{E}}$ are already $F$-complete, allowing us to drop the completion. We can now check that a variant of \cite[Corollary 4.44]{KtheoryZpn} with coefficients holds. The map $\varphi^{\nabla}$ appearing in the statement will be clarified in the proof.

\begin{lemma}
Let $\mathcal{E}$ be a reflexive $F$-gauge on $\O_K^{\Syn}$, and make a choice of a Breuil-Kisin prism $(\W(k)\llbracket u_0\rrbracket,(\E(u_0)))$ for $\O_K$. Let $\M$ be the Breuil-Kisin module associated to $\mathcal{E}$ with the na\"{i}ve Nygaard filtration $\mathrm{Fil}^i_{\Nyg} \varphi^* \M = \{x\in \varphi^* \M: \varphi_\M(x) \in \E(u_0)^i \M\}$, and let $\mathrm{F}^* \M$ be the filtration induced by regarding $\W(k)\llbracket u_0\rrbracket$ as a filtered $\delta$-ring where $\deg(u_0)=1$.

Then $\R\Gamma_{\Syn}(\O_K,\mathcal{E})$ can be computed as the total fiber of the square
\begin{display}
\mathrm{F}^* \mathrm{Fil}_{\Nyg}^0 \varphi^* \M \ar{d}{\mathrm{can}-\varphi_\M} \ar{r}{\mathrm{Fil}^0_{\Nyg} \nabla_\M} & \mathrm{F}^{*-1} \mathrm{Fil}_{\Nyg}^{-1} \varphi^* \M \ar{d}{\mathrm{can}-\varphi_\M^{\nabla}} \\
\mathrm{F}^* \varphi^* \M \ar{r}{\nabla_\M} & \mathrm{F}^{*-1} \varphi^* \M
\end{display}
for a map $\varphi^{\nabla}$. In the square, $\varphi_{\M}$ is regarded as a Frobenius semilinear map.
\label{lem:syntomic_fiber}
\end{lemma}

\begin{proof}
This follows from the same method as \cite[Corollary 4.44]{KtheoryZpn} but with minor modifications for coefficients. Let $\M=\R\Gamma((\O_K/\W(k)\llbracket u_0\rrbracket)^{\Prism},\mathcal{E})$ be the Breuil-Kisin module attached to $\mathcal{E}$. Using Lemma \ref{lem:trunc_desc_complex} we see
\[\R\Gamma(\O_K^{\Prism},\mathcal{E}) \simeq \left[\Prism_{\O_K/\W(k)\llbracket u_0\rrbracket}(\mathcal{E}) \to K_2\right]\]
where $K_2$ is the kernel of the second differential in the descent complex. Similarly using the descent complex colimit presentation of $\O_K^{\Nyg}$ by Rees stacks, we have
\[\R\Gamma(\O_K^{\Nyg},\mathcal{E}) \simeq \left[\Fil^0_{\Nyg}\varphi^* \Prism_{\O_K/\W(k)\llbracket  u_0\rrbracket}(\mathcal{E}) \to \Fil^0_{\Nyg} \varphi^* K_2\right].\]
The cohomological amplitude input is given by the identification with a cobar complex, supplied by the Nygaard filtered version of \cite[Lemma 4.43]{KtheoryZpn}, which we explain later in the proof (it does not depend on this amplitude claim). Alternatively, this may also be supplied by using the presentation of $\O_K^{\Nyg}|_{t=0}$ in \cite{pham2024}.\footnote{Pham's operator $\D$ seems closely related to $\overline{\nabla}_\M$ used in Theorem \ref{thm:general_Tet_kernel}.}

Thus, the syntomic cohomology of the reflexive $F$-gauge $\mathcal{E}$ is given by the total fiber of
\begin{display}
\mathrm{Fil}_{\Nyg}^0 \varphi^* \M \ar{d}{\mathrm{can}-\varphi_\M} \ar{r} & \mathrm{Fil}_{\Nyg}^0 \varphi^* K_2 \ar{d}{\mathrm{can}-\varphi_{K_2}} \\
\varphi^* \M \ar{r} & \varphi^* K_2
\end{display}
and the horizontal maps come from the first differential. We can take $\mathrm{F}$-completions on every term, since syntomic cohomology with coefficients in a reflexive $F$-gauge is still $\mathrm{F}$-complete. Indeed, the proof of \cite[Proposition 2.11]{KtheoryZpn} readily generalizes as \cite[Proposition 8.6]{PrismaticDelta} works with coefficients (syntomic cohomology with coefficients lands in their category of quadruples $(H,N,c,\varphi)$).

Thus, we see syntomic cohomology is computed as the total fiber of the $\mathrm{F}$-completed square
\begin{display}
\mathrm{F}^* \mathrm{Fil}_{\Nyg}^0 \varphi^* \widehat{\M} \ar{d}{\mathrm{can}-\varphi_\M} \ar{r} & \mathrm{F}^* \mathrm{Fil}_{\Nyg}^0 \varphi^* \widehat{K}_2 \ar{d}{\mathrm{can}-\varphi_{K_2}} \\
\mathrm{F}^* \varphi^* \widehat{\M} \ar{r} & \mathrm{F}^* \varphi^* \widehat{K}_2 
\end{display}
As before, the symbol $\widehat{\cdot}$ denotes $\mathrm{F}$-completion.

The map $\Theta$ is an isomorphism when restricted to $K_2$, which can be shown by using the same proof as \cite[Lemma 4.43]{KtheoryZpn}. For completeness, we give a brief overview of the argument. We see $\widehat{\Prism}_{\O_K/\W(k)\llbracket u_0\rrbracket}(\mathcal{E})$ is a comodule over the Hopf algebroid 
\[\Gamma = (\widehat{\Prism}_{\O_K/\W(k)\llbracket u_0\rrbracket}, \widehat{\Prism}_{\O_K/\W(k)\llbracket u_0,u_1\rrbracket})\]
and the cobar complex of $\widehat{\Prism}_{\O_K/\W(k)\llbracket u_0\rrbracket}(\mathcal{E})$ identifies with its descent complex. The claim we want to show is that we get a quasi-isomorphism of complexes
\begin{display}
\mathrm{F}^* \widehat{\Prism}_{\O_K/\W(k)\llbracket u_0\rrbracket}(\mathcal{E}) \ar[equals]{d} \ar{r} & \mathrm{F}^* \widehat{\Prism}_{\O_K/\W(k)\llbracket u_0,u_1\rrbracket}(\mathcal{E}) \ar{r} \ar{d}{\Theta} & \ldots \\
\mathrm{F}^* \widehat{\Prism}_{\O_K/\W(k)\llbracket u_0\rrbracket}(\mathcal{E}) \ar{r}{\nabla} & \mathrm{F}^{*-1} \widehat{\Prism}_{\O_K/\W(k)\llbracket u_0\rrbracket}(\mathcal{E}) \ar{r} & 0
\end{display}
where the top complex is the $\mathrm{F}$-completed descent complex. The argument of Lemma 4.43 now follows identically by using the cobar resolution of $\Prism_{\O_K/\W(k)\llbracket u_0\rrbracket}(\mathcal{E})$ and that $\mathcal{E}$ is a prismatic crystal in vector bundles. The Frobenius twisted and Nygaard filtered versions follow similarly.

Translating the previous square along $\Theta$, one then obtains a square
\begin{display}
\mathrm{F}^* \mathrm{Fil}_{\Nyg}^0 \varphi^* \widehat{\M} \ar{d}{\mathrm{can}-\varphi_\M} \ar{r}{\nabla_\M} & \mathrm{F}^{*-1} \mathrm{Fil}_{\Nyg}^{-1} \varphi^* \widehat{\M} \ar{d}{\mathrm{can}-\varphi_\M^{\nabla}} \\
\mathrm{F}^* \varphi^* \widehat{\M} \ar{r}{\nabla_\M} & \mathrm{F}^{*-1} \varphi^* \widehat{\M}
\end{display}
whose total fiber computes syntomic cohomology as desired. As a final step, we observe that actually $\widehat{\M}=\M$. The $\mathrm{F}$-filtration on $\Prism_{\O_K/\W(k)\llbracket u_0 \rrbracket}$ is the $u_0$-adic filtration by \cite[Proposition 10.43(c)]{PrismaticDelta}, so this is already complete; since we work with prismatic crystals in vector bundles and use the induced flat filtration, the desired completeness follows.
\end{proof}

As seen in the proof, the map $\varphi^{\nabla}$ is induced by the map $\varphi_\M$ on the kernel $K_2$ of $\mathrm{d}^2$ in the descent complex for the prismatic $F$-crystal, which is identified with the desired map by applying $\Theta$ and trivializing the twist. By construction, it satisfies the relation $\nabla \circ \varphi = \varphi^{\nabla} \circ \nabla$.

We may now use the ramified variant of this result in Lemma \ref{lem:syntomic_fiber} to deduce the desired vanishing statement. Before proceeding, we note that the connection $\nabla_\M$ does not actually have a Leibniz rule, but it does have one on the associated $F$-graded with correction terms. In particular, all we will use is that
\[\nabla_\M(u_0 m) = \mathrm{C}_i m + n\]
where if $m\in \mathrm{F}^{i-1} \varphi^* \M \setminus \mathrm{F}^i \varphi^* \M$ we may pick $n\in \mathrm{F}^i \varphi^* \M$ and $\mathrm{C}_i$ is multiplication by $i$ up to units. In general by writing $u_0 m = u_0^i m_0$ where $m_0$ is in $\mathrm{F}^0 \varphi^* \M$ by definition
\[\nabla_\M(u_0^i m_0) = (\mathrm{id} \otimes \Theta)(\rho(u_0^i m_0))\]
where $\rho: \varphi^* \M \to \varphi^* \M \otimes_{\W(k)\llbracket u_0\rrbracket} \widehat{\Prism}^{(1)}_{\O_K/\W(k)\llbracket u_0,u_1 \rrbracket}$ is the coaction of the comodule $\varphi^* \M$. Here, since we work with the Frobenius twist we clarify that $u_0$ is $1 \otimes u_0$. We can write the coaction as $\rho(m_0) = m_0 \otimes 1 + \rho_{\ge 1}$, where $\rho_{\ge 1}$ is the contribution in higher $\mathrm{F}$-degrees; the first component is forced by the relation $(\id \otimes \varepsilon) \rho(m_0) = m_0$. That is, $\rho(m_0)-m_0 \otimes 1\in \ker(\id \otimes \varepsilon)$ and this kernel will lie in $\mathrm{F}^{\ge 1}$ (as $\varepsilon:\Gamma_1\to \Gamma_0$ sends $u_1,u_0$ to $u_0$ and the $\mathrm{F}$-filtration on $\W(k)\llbracket u_0, u_1 \rrbracket \subset \Gamma_1$ is determined by \cite[Proposition 10.43(c)]{PrismaticDelta}). Using this, up to higher terms in the $\mathrm{F}$-filtration (i.e. modulo $\mathrm{F}^i \varphi^* \M$) we can compute $(\mathrm{id} \otimes \Theta)(\rho(u_0^i m_0)) = (\mathrm{id} \otimes \Theta)(\eta_R(u_0^i) (m_0 \otimes 1))$. Using the same argument as \cite[Lemma 4.61(1)]{KtheoryZpn}, there is a total contribution of $\mathrm{C}_i m$ where $\mathrm{C}_i$ is $i$ up to units to the $\gr_{\mathrm{F}}^{i-1}$ component. That is, working in the twist $\varphi^* \M$ we may expand $(\id \otimes \Theta)(\eta_R(u_0^i) (m_0 \otimes 1))$ as $m_0 \otimes i u_0^{i-1} + m_0\otimes \sum_{j=2}^i \binom{i}{j} u_0^{i-j}\Theta(g_0^j)$ where $g_0 := u_1-u_0$ (in \cite{KtheoryZpn} the map $\Theta$ is $\W(k)\llbracket u_0 \rrbracket^{(1)}$-linear and is defined by what it does on monomials in the $g_u$ basis, starting with $g_0$ as the first nontrivial one). The same argument allows us to deduce the surviving $j \ge 2$ terms are multiples of $m_0 (p \cdot i u_0^{i-1})$, from which the claim follows as they can only change the scalar $\mathrm{C}_i$ by a unit. Note that we have already trivialized the Breuil-Kisin twist in this formulation, as we regard $\nabla_\M$ as a map $\varphi^* \M \to \varphi^* \M$.

\begin{proposition}
If $\mathcal{E}$ is in $\Coh(\O_K^{\Syn})$, then $\tau^{\ge 2} \R\Gamma(\O_K^{\Syn},\mathcal{E})[1/p]=0$. 
\label{prop:H2_vanish}
\end{proposition}

\begin{proof}
Without loss of generality, we may assume that $\mathcal{E}$ is reflexive: up to $p$-isogeny this is always possible by the proof of Corollary \ref{cor:Perf_cris}. Indeed, replacing $\mathcal{E}$ by its reflexive hull $\mathcal{E}^{\vee\vee}$ does not change the claim, since the kernel and cokernel of $\mathcal{E}\to \mathcal{E}^{\vee\vee}$ are $p$-power torsion, and hence their syntomic cohomology is $p$-power torsion as well.

By Lemma \ref{lem:syntomic_fiber} it suffices to check that the map
\[
\begin{tikzcd}
\mathrm{F}^{*-1} \mathrm{Fil}_{\Nyg}^{-1} \varphi^* \M \oplus \mathrm{F}^* \varphi^* \M
\ar{rr}{\begin{pmatrix} \mathrm{can}-\varphi_\M^{\nabla} \\ -\nabla_\M \end{pmatrix}^{\mathrm{T}} }
& &
\mathrm{F}^{*-1} \varphi^* \M
\end{tikzcd}
\]
is surjective after inverting $p$. Call this map $\alpha$.

Assume first that $\mathcal{E}$ has all Hodge--Tate weights $\ge 0$. Then $\varphi_\M^{\nabla}$ is defined integrally on $\varphi^*\M$, and $\M$ is $u_0$-adically complete. In the effective case, $\mathrm{Fil}^{-1}_{\Nyg}\varphi^*\M[1/p]\subseteq \varphi^*\M[1/p]$ is an equality, so we may regard $\varphi_{\M}^{\nabla}$ as an endomorphism of $\varphi^* \M$. One may check that $\varphi_\M^{\nabla}$ strictly raises the induced $u_0$-adic filtration (see \cite[Definition 4.54]{KtheoryZpn} for the case of $\mathcal{E}=\O$, noting the filtration is shifted by $-1$). In particular, $\varphi^{\nabla}$ is topologically nilpotent, and the convergent series $\sum_{j\ge 0}(\varphi_{\M}^{\nabla})^j$ gives a right inverse to $\mathrm{can}-\varphi_{\M}^{\nabla}$ proving integral surjectivity of $\alpha$.

In general, we may write the desired reflexive $F$-gauge as $\mathcal{E}\{r\}$ for some $r\ge 0$, where $\mathcal{E}$ has all Hodge--Tate weights $\ge 0$. The corresponding untwisted module $\M\{r\}$ is now equipped with a semilinear map
\[\varphi_{\M\{r\}}:\M\{r\}\longrightarrow \frac{1}{\E(u_0)^{r}}\M\{r\},\]
so due to the indexing shift this will imply $\varphi_{\M\{r\}}^{\nabla}$ is only defined integrally on $\mathrm{Fil}^{-1}_{\Nyg}\varphi^*\M\{r\}\subseteq \varphi^*\M\{r\}$, and iterates $(\varphi_{\M\{r\}}^{\nabla})^j$ need not be defined on all of $\varphi^*\M\{r\}$ when there is a nontrivial twist. Thus, we cannot run the same argument.

However, a similar argument can be run when working high enough in the $\mathrm{F}$-filtration. The cokernel $\coker(\alpha)$ is derived $p$-complete (as syntomic cohomology is derived $p$-complete and the cokernel is identified with $\H^2_{\Syn}$), and the filtered pieces $\mathrm{F}^i\coker(\alpha)$ are derived $p$-complete as well. Thus, to prove $\mathrm{F}^{i_0}\coker(\alpha)=0$ it suffices to prove $(\mathrm{F}^i\coker(\alpha))/p=0$ for all $i\ge i_0$, i.e. surjectivity of $\alpha$ after restricting to $\mathrm{F}^i$ and reducing modulo $p$.

Since $\E(u_0)\equiv u_0^e \pmod p$ (where $e$ is the ramification index), one checks that there is an $i_0$ so that for $i\ge i_0$ the map $\mathrm{can}-\varphi_{\M\{r\}}^{\nabla}$ induces a well-defined endomorphism of $(\mathrm{F}^i\varphi^*\M\{r\})/p$ using that the map $\varphi_{\M\{r\}}^{\nabla}$ scales the $-1$-shifted $\mathrm{F}$-filtration by a factor of $p$ (again see \cite[Definition 4.54]{KtheoryZpn}, noting the filtration is shifted by $-1$). Moreover, for such $i$ we have
\[(\mathrm{F}^i\varphi^*\M\{r\})/p \subseteq (\mathrm{Fil}^{-1}_{\Nyg}\varphi^*\M\{r\})/p,\]
so that passing to $\mathrm{Fil}^{-1}_{\Nyg}$ does not enlarge the source after reduction modulo $p$. Concretely, writing $i=ej+j'$ with $0\le j'<e$, we have $u_0^i\equiv u_0^{j'}\E(u_0)^j \pmod p$; for $j\ge r$ we may factor $\E(u_0)^j=\E(u_0)^{j-r}\E(u_0)^r$, and the bounded-denominator condition for the twist implies that multiplying by $\E(u_0)^r$ places us in $\mathrm{Fil}^{0}_{\Nyg}\subseteq \mathrm{Fil}^{-1}_{\Nyg}$. It follows that, for $i\ge i_0$, we may regard $\mathrm{can}-\varphi_{\M\{r\}}^{\nabla}$ as an endomorphism of $(\mathrm{F}^i\varphi^*\M\{r\})/p$.

Finally, modulo $p$ the operator $\varphi_{\M\{r\}}^{\nabla}$ strictly raises the $u_0$-adic filtration on $(\mathrm{F}^i\varphi^*\M\{r\})/p$ for $i\ge i_0$, hence is topologically nilpotent there. The same convergent power series argument therefore shows that $\mathrm{can}-\varphi_{\M\{r\}}^{\nabla}$ is surjective onto $(\mathrm{F}^i\varphi^*\M\{r\})/p$ for all $i\ge i_0$, and we conclude that $\mathrm{F}^{i_0}\coker(\alpha)=0$ integrally (note that this does not show integral surjectivity of $\mathrm{can}-\varphi_{\M\{r\}}^{\nabla}$, as the cokernel need not be derived $p$-complete).

It remains to check rationalized surjectivity on the finitely many remaining graded pieces of the $\mathrm{F}$-filtration for indices $i<i_0$. For $i\ge 1$ the constants $\mathrm{C}_i$ in the formula $\nabla(u_0m)=\mathrm{C}_i m+n$ are units after inverting $p$, so $\mathrm{gr}^i_{\mathrm{F}}\nabla_{\M\{r\}}$ is a rational isogeny on each such graded piece (it acts by $\mathrm{C}_i$ on $\mathrm{gr}^i_{\mathrm{F}}$, recalling the target filtration is shifted). Since there are only finitely many indices $i<i_0$, it follows that $\coker(\alpha)[1/p]=0$, as desired.
\end{proof}

\begin{proof}[Proof of Proposition \ref{prop:O_K_cohom}]
We have already shown this for $\H^0$ (by the heart claim in Corollary \ref{cor:Perf_cris}) and for $\H^1$ it follows from Proposition \ref{prop:ext1_cohom}. For $i\ge 2$, the $\Ext^i$ groups in $\Rep_{\Q_p}^\cris(\rG_K)$ are zero. Thus we only need to show vanishing of cohomology in degrees $\ge 2$, which was just shown in Proposition \ref{prop:H2_vanish}.
\end{proof}

\begin{remark}
Let $\X/\Spf \O_K$ be smooth and quasicompact, with structure map $f$. There is a hypercohomology Leray spectral sequence
\[\E_1^{i,j}=\H^{2i+j}_{\Syn}(\O_K,\R^{-i} f_* \mathcal{E}) \implies \H^{i+j}_{\Syn}(\X,\mathcal{E}).\]
This is the spectral sequence associated to the filtration on 
\[\R\Gamma_{\Syn}(\X,\mathcal{E}) \simeq \R\Gamma_{\Syn}(\Spf \O_K, \R f_* \mathcal{E})\]
by the complexes $\R\Hom(\O,\tau_{\le -i} \R f_* \mathcal{E})$ (see Stacks Project \href{https://stacks.math.columbia.edu/tag/015X}{015X} for more details about the indexing -- we choose $\tau_{\le -i}$ to make a decreasing filtration). We can check this converges using $t$-boundedness of $\R f_* \mathcal{E}$, which is satisfied because $\X$ is smooth and quasicompact. In the smooth proper case we will focus on, this is immediate from the stronger result that $\R f_* \mathcal{E}$ is perfect. 

This has a practical consequence for smooth proper $\X/\Spf \O_K$. Let $\mathcal{E}\in \Coh(\X^{\Syn})$. 

The differentials on $\E_1$ are
\[d_1^{p,q}: \H^{2p+q}_{\Syn}(\O_K,\R^{-p} f_* \mathcal{E}) \to \H^{2(p+1)+q}_{\Syn}(\O_K,\R^{-p-1} f_* \mathcal{E}).\]
This can never be nontrivial rationally, since $\H^2$ vanishes (as well as all higher $\H^i$). Indeed, $\R^i f_* \mathcal{E}$ is coherent as $\R f_* \mathcal{E}$ is a perfect complex (by Proposition \ref{prop:sm_prop_pushforward}). Thus after rationalization this spectral sequence degenerates on $\E_1$. Note that this also shows the syntomic period map $\H^i_{\Syn}(\X,\mathcal{E}) \to \H^i_\proet(\X_\eta,\T_\et(\mathcal{E}))$ is injective rationally in light of Proposition \ref{prop:O_K_cohom}.
\end{remark}

We are now ready to reduce showing the equivalence $\Perf(\O_K^{\Syn})[1/p] \simeq \D^b(\Rep_{\Q_p}^\cris(\rG_K)) \simeq \D^b(\mathrm{MF}_{K}^{\varphi,\mathrm{wa}})$ to a cohomology computation. Note that the definition of $\R\Hom_{\mathrm{MF}_{K}^{\varphi,\mathrm{wa}}}$ uses the $\Ind$ category of $\mathrm{MF}_{K}^{\varphi,\mathrm{wa}}$ and \cite{Ind}, as the category of weakly admissible filtered isocrystals does not have enough injectives. We can realize the target category $\D^b(\mathrm{MF}_{K}^{\varphi,\mathrm{wa}})$ as an $\infty$-category by applying the usual construction for $\mathcal{A}=\Ind(\mathrm{MF}_{K}^{\varphi,\mathrm{wa}})$ as the stabilization of the animation of the category of injectives and regarding $\D^b(\mathrm{MF}_{K}^{\varphi,\mathrm{wa}})$ as a full subcategory.

\begin{proposition}
There is a symmetric monoidal equivalence of categories
\[\Perf(\O_K^{\Syn})[1/p] \simeq \D^b(\Rep_{\Q_p}^\cris(\rG_K))\]
induced by $\T_\et$.
\label{prop:O_K_equiv}
\end{proposition}

\begin{proof}
Once the cohomology calculation has been done, the proof is essentially the same as in \cite{Hauck}. We spell out some details of this argument for the convenience of the reader. The natural equivalence induced by $\T_\et$ is that $\D^b(\Coh(\X^{\Syn})[1/p])\simeq \D^{b}(\Loc_{\Q_p}^\cris(\X_\eta))$ (which is then also equivalent to $\D^b(\mathrm{MF}_{K}^{\varphi,\mathrm{wa}})$), so our task is actually to show the natural functor
\[F: \D^b(\Coh(\X^{\Syn})[1/p]) \to \Perf(\X^{\Syn})[1/p]\]
is an equivalence for $\X=\Spf \O_K$. Here, we know the functor is well-defined by the same argument as in the proof of Corollary \ref{cor:Perf_cris}. This functor can be seen to be essentially surjective once we know full faithfulness, as we have an equivalence on the heart and may construct objects using shifts and fibers from the heart on the target $\Perf(\X^{\Syn})[1/p]$ via \cite[Lemma 5.4.3]{Hauck}. We may then construct an object $\mathcal{E}\in \D^b(\Coh(\X^{\Syn})[1/p])$ with the desired output $F(\mathcal{E})$ by lifting the attaching maps via full faithfulness and noting $F$ is compatible with shifts and fibers.

Full faithfulness can then be reduced to checking that $\R\Hom$ in both categories agree following the standard reduction to cohomology method in \cite[Proposition 5.6.2]{Hauck}. The functor $F$ is compatible with duals, tensor products, and the internal RHom, so via the tensor-Hom adjunction we reduce to checking cohomology with coefficients agrees (using the isomorphism $\R\Hom_{\Perf(\O_K^{\Syn})}(\O,\mathcal{E})\simeq \R\Gamma_{\Syn}(\O_K,\mathcal{E})$ to pass to cohomology at the end). Note that existence and compatibility with duals and the internal RHom is slightly tricky for $\D^b(\Coh(\X^{\Syn})[1/p])$ but can be justified as $\D^b(\Coh(\X^{\Syn})[1/p])$ is equivalent to the rigid symmetric monoidal category $\D^{b}(\Rep_{\Q_p}^\cris(\rG_K))$ via the symmetric monoidal functor $\T_\et$. Since $F$ is symmetric monoidal this allows us to deduce $F$ is compatible with duals and RHoms as its source is rigid symmetric monoidal. Finally, compatibility of $F$ with shifts and fibers allows us to further reduce to checking cohomologies with coherent coefficients agree.

Writing out cohomology with coefficients in the heart in each category, since there is a map of complexes we need only check that
\[\Ext^i_{\Rep_{\Q_p}^\cris(\rG_K)}(\Q_p,\T_\et(\mathcal{E})[1/p]) \simeq \H^i_{\Syn}(\Spf \O_K,\mathcal{E})[1/p]\]
which is shown in Proposition \ref{prop:O_K_cohom}.
\end{proof}

This result in the case of a point will be useful in proving an analogous result for smooth proper $\X/\Spf \O_K$. In light of Example \ref{ex:derived_counter}, it will be necessary to change how the theorem is formulated. It turns out to be more natural to interpret Proposition \ref{prop:O_K_equiv} as stating that
\[\Perf(\O_K^{\Syn})[1/p] \simeq  \D^b(\Rep^{\cris}_{\Q_p}(\rG_K))\simeq \D^b(\mathrm{MF}_K^{\varphi,\mathrm{wa}})\]
where $\mathrm{MF}_K^{\varphi,\mathrm{wa}}$ is the category of weakly admissible filtered $F$-isocrystals. This is induced by the abelian category level result of Colmez-Fontaine in \cite{ColmezFontaine} that the functor $\D_\cris$ induces an equivalence $\Rep^{\cris}_{\Q_p}(\rG_K) \simeq \mathrm{MF}_K^{\varphi,\mathrm{wa}}$.

In general, we will need to formulate the appropriate notion of a derived filtered $F$-isocrystal. This is given by the Beilinson fiber square; namely, we have a commutative diagram of stacks
\begin{display}
\X^{\Syn} & \ar{l} (\X_s)^{\Syn} \\
(\X/\O_K)^{\dR,+} \ar{u} & \ar{l} \ar{u} (\X/\O_K)^{\dR}
\end{display}
where $\X_s := \X \times_{\Spf \O_K} \Spec k$. The left vertical map is $(\X/\O_K)^{\dR,+}\to \X^{\dR,+}\to \X^{\Syn}$. There is a nontrivial arrow to define here, namely the map of stacks
\[(\X/\O_K)^{\dR} \to (\X_s)^{\Syn}.\]
We only need to construct such a map in the case of a point $\X=\Spf \O_K$ as this then induces a map for all $\X$. In this case, it is given explicitly by
\[\Spf \O_K \to k^{\Nyg} = (\Spf \W(k)\langle u,t \rangle/(ut-p))/\G_m \to k^{\Syn}\]
induced by sending $t \to 1$ and $u \to p$. Note that in the unramified case we end up with the inclusion of the $t \neq 0$ locus, namely $\W(k)^{\dR}\simeq k^{\Prism} \simeq (k^{\Nyg})_{t \neq 0} \to k^{\Syn}$ as in \cite{Hauck}. The map from $(\X/\O_K)^{\dR,+}$ to $\X^{\Syn}$ is induced by $(\X/\O_K)^{\dR,+} \to \X^{\dR,+} \to \X^{\Syn}$.

What we first show is that after taking cohomology and rationalizing, the square
\begin{display}
	\R\Gamma(\X^{\Syn},\mathcal{E})[1/p] \ar{d} \ar{r} & \R\Gamma(\X_{s}^{\Syn}, \mathcal{E})[1/p] \ar{d} \\
	\R\Gamma((\X/\O_K)^{\dR,+}, \mathcal{E})[1/p] \ar{r} & \R\Gamma((\X/\O_K)^{\dR}, \mathcal{E})[1/p]
\end{display}
becomes Cartesian for $\mathcal{E}$ perfect. This is shown in \cite{Hauck} in the unramified case and in \cite{devalapurkar} for $\Z_p[\zeta_p]$. Once definitions are unwound we will see this claim is essentially equivalent to Proposition \ref{prop:O_K_equiv}, similar to the argument in the unramified case in \cite{Hauck}.

\begin{corollary}[Beilinson fiber square for smooth proper $\X$]
Let $\X/\Spf \O_K$ be smooth proper, and denote the residue field of $K$ by $k$. Fix a choice of Breuil-Kisin prism $(\W(k)\llbracket u_0 \rrbracket, (\E(u_0)))$ in order to define the maps of stacks as above and let $\mathcal{E}\in \Perf(\X^{\Syn})$. There is a fiber square
\begin{display}
\R\Gamma(\X^{\Syn},\mathcal{E})[1/p] \ar{d} \ar{r} & \R\Gamma(\X_{s}^{\Syn}, \mathcal{E})[1/p] \ar{d} \\
\R\Gamma((\X/\O_K)^{\dR,+}, \mathcal{E})[1/p] \ar{r} & \R\Gamma((\X/\O_K)^{\dR}, \mathcal{E})[1/p]
\end{display}
where we write $\X_s := \X \times_{\Spf \O_K} \Spec k$ and abuse notation to again write $\mathcal{E}$ for the pullback to the stacks $\X_s^{\Syn},(\X/\O_K)^{\dR}$, and $(\X/\O_K)^{\dR,+}$.
\label{cor:Beilinson_fiber}
\end{corollary}

\begin{proof}
By Proposition \ref{prop:sm_prop_pushforward}, it suffices to prove the claim for $\X=\Spf \O_K$ (note that the realizations involved commute with pushforwards, by similar arguments as in Lemma \ref{lem:realization_commute}; see also the appendix in \cite{Hauck} for the unramified case). As already noted, Proposition \ref{prop:O_K_equiv} says that the functor
\[\Perf(\O_K^{\Syn})[1/p] \to \D^b(\mathrm{MF}_{K}^{\varphi,\mathrm{wa}})\]
is an equivalence. For $\mathcal{E}\in \Perf(\O_K^{\Syn})$, write $\D\in \D^b(\mathrm{MF}_{K}^{\varphi,\mathrm{wa}})$ for the image under this equivalence.

This is the only nontrivial identification in the square. We can identify $\Perf(k^{\Syn})[1/p]$ with isocrystals $\Perf^{\varphi}(K_0)$ (which follows easily using \cite[Lemma 3.4.11]{Fgauge}) where $K_0=\W(k)[1/p]$, and as we take the de Rham stacks relative to $\O_K$ we have $\Perf((\O_K/\O_K)^{\dR})[1/p]\simeq \Perf(K)$ and for the filtered de Rham stack we get filtered perfect complexes of $K$-vector spaces (which we denote $\mathrm{MF}_K$ to align with usual notation). Now as before let $\D$ be the bounded complex of admissible filtered isocrystals associated to $\mathcal{E}\in \Perf(\O_K^{\Syn})[1/p]$, and let $K_0=\W(k)[1/p]$ where $k$ is the residue field of $K$. Unwinding definitions we can rewrite the fiber square for the case of $\O_K$ as
\begin{display}
\R\Hom_{\mathrm{MF}_{K}^{\varphi,\mathrm{wa}}}(K_0,\D) \ar{d} \ar{r} & \R\Hom_{\Perf^{\varphi}(K_0)}(K_0, \T_\cris(\D)) \ar{d} \\
\R\Hom_{\mathrm{MF}_K}(K,\T_{\dR,+}(\D)) \ar{r} & \R\Hom_{K}(K, \T_{\dR}(\D))
\end{display}
where we obtain the obvious maps. The only one which needs any inspection is the map $\Perf^{\varphi}(K_0) \to \Perf(K)$, which is indeed induced by pullback along the integral map $\Spf \O_K \to k^{\Syn}$. We may reduce to $\D$ concentrated in degree zero as well by applying \cite[Lemma 5.4.3]{Hauck}, at which point the claim follows from the $\Ext$ computations in \cite{emerton1999extensions} (just as in the calculation in \cite[Proposition 5.2.4]{Hauck}, we use \cite[Corollary 2.4.4]{emerton1999extensions}).
\end{proof}

\begin{remark}
An analogous statement is likely true for all smooth $\X/\O_K$, but we do not pursue this here.
\end{remark}

\begin{definition}
Let $\X/\Spf \O_K$ be a smooth and quasicompact $p$-adic formal scheme. Define $\Perf_{\mathrm{fIsoc}^{\varphi}}(\X)$ as the isogeny category $\mathcal{C}[1/p]$ of
\[\mathcal{C}=\Perf(\X_s^{\Syn}) \times_{\Perf((\X/\O_K)^{\dR})} \Perf((\X/\O_K)^{\dR,+}).\]
Individually, all three categories carry a natural $t$-structure induced by a Noetherian regular flat covering (which is well-defined by the same argument as $\X^{\Syn}$).
\label{def:admfil_Fisoc}
\end{definition}

The Beilinson fiber square then tells us (using the internal $\R\Hom$ and the tensor-Hom adjunction) in the smooth proper case that we have a fully faithful functor
\[\mathrm{Beil}: \Perf(\X^{\Syn})[1/p] \to \Perf_{\mathrm{fIsoc}^{\varphi}}(\X).\]
We will next characterize the heart of the $t$-structure on this category. In what follows the filtration on a filtered $F$-isocrystal is more general than what is usually allowed, but the essential image of $\mathrm{Beil}$ will later be shown to be the subcategory of admissible filtered $F$-isocrystals and thus lands in the usual subcategory with a stronger filtration condition (for example the convention of \cite{Zpcrystalline} for filtered $F$-isocrystals).

\begin{definition}
Let $\X/\Spf \O_K$ be smooth. The category $\mathrm{Isoc}^{\varphi}(\X_{s}/\W(k))$ is defined as the isogeny category of $F$-crystals in coherent sheaves on $(\X_{s}/\W(k))_\cris$.

We set
\[\mathrm{fIsoc}^{\varphi}(\X) := \mathrm{Isoc}^{\varphi}(\X_{s}/\W(k)) \times_{\Vect^{\nabla}(\X_\eta)} \Vect^{\nabla, +}(\X_\eta)\]
where $\Vect^{\nabla}(\X_\eta)$ is the category of vector bundles $(\E, \nabla)$ on $\X_\eta$ with flat connection, and $\Vect^{\nabla,+}(\X_\eta)$ is the same but with a complete decreasing (possibly non-genuine) filtration on $\E$ so $\nabla$ is a filtered map of degree $-1$ with coherent associated graded.

The functor $\mathrm{Isoc}^{\varphi}(\X_{s}/\W(k)) \to \Vect^{\nabla}(\X_\eta)$ is given by \cite[Theorem 2.15]{ogus1984f}.

A filtered $F$-isocrystal is \textit{admissible} if it is of the form $\D_\cris(\mathbb{L})$ for a crystalline local system $\mathbb{L}$. We define $\D_{\cris}$ as the composite functor $\Loc_{\Q_p}^\cris(\X_\eta) \simeq \Vect^{\varphi,\mathrm{an}}(\X_{\Prism},\O_{\Prism})[1/p] \to \mathrm{fIsoc}^{\varphi}(\X)$ where the second functor is $\tilde{\T}_\cris$ from \cite[Corollary 4.6]{Zpcrystalline} (denoted $\tilde{\D}_{\mathrm{crys}}$ there). Let $\Perf^{\mathrm{adm}}_{\mathrm{fIsoc}^{\varphi}}(\X)$ be the full subcategory consisting of filtered $F$-isocrystals $\mathcal{E}$ in perfect complexes where each $\H^i(\mathcal{E})$ is admissible.
\label{def:fIsoc}
\end{definition}

This variant of $\D_{\cris}$ is compatible with other definitions which are more classical, for example the one in the appendix of \cite{Zpcrystalline2}: using Remark 4.7 in \cite{Zpcrystalline} one need only check the filtered $F$-isocrystal they associate to a crystalline local system $\T_\et(\mathcal{E})$ is associated in the sense of Faltings (see \cite[Definition 2.31]{Zpcrystalline}). In particular, the subcategory of admissible objects we use agrees with any other reasonable definition.
		
The author is not aware of a linear-algebraic condition which characterizes admissibility inside of $\mathrm{fIsoc}^{\varphi}(\X)$, and it would be interesting to know of one (for example, pointwise weak admissibility at each classical point is likely not sufficient). 

The category $\Perf_{\mathrm{fIsoc}^{\varphi}}(\X)$ carries a $t$-structure inherited from the $t$-structure on each of the categories in the fiber product, constructed analogously from a regular Noetherian flat cover. In the next proposition we verify this is actually well-defined, i.e. that the relevant maps in the fiber product are $t$-exact on isogeny categories; the category 
\[\left(\Coh(\X_{s}^{\Syn}) \times_{\Coh((\X/\O_K)^{\dR})} \Coh((\X/\O_K)^{\dR,+})\right)[1/p]\] 
is the heart of this $t$-structure.

\begin{proposition}
Assume $\X/\Spf \O_K$ is smooth and quasicompact. There is an equivalence of abelian categories
\[\Coh(\X_{s}^{\Syn})[1/p] \to \mathrm{Isoc}^\varphi(\X_{s}/\W(k))\]
and moreover
\[\left(\Coh(\X_{s}^{\Syn}) \times_{\Coh((\X/\O_K)^{\dR})} \Coh((\X/\O_K)^{\dR,+})\right)[1/p] \simeq \mathrm{fIsoc}^{\varphi}(\X).\]
This describes the heart of $\Perf_{\mathrm{fIsoc}^{\varphi}}(\X)$ with its natural $t$-structure.
\label{prop:isoc_justification}
\end{proposition}

\begin{proof}
We start by arguing that $\Coh(\X_{s}^{\Syn})[1/p] \simeq \mathrm{Isoc}^\varphi(\X_{s}/\W(k))$. We have a natural functor
\[\Coh(\X_{s}^{\Syn})[1/p] \to \mathrm{Isoc}^\varphi(\X_{s}/\W(k))\]
given by restriction to $\X_{s}^{\Prism}$ (which produces an $F$-isocrystal for the same reasons as in mixed characteristic). To deduce this is an equivalence, we may use the same argument we used to deduce the analogous claim for $\Coh(\X^{\Syn})[1/p]$ and the isogeny category of coherent prismatic $F$-crystals in mixed characteristic. We may first reduce to when $\X$ is smooth affine by Zariski descent. The idea is then that coherent crystals on $(\X_s/\W(k))_\cris$ after evaluation on a covering (now taking the prisms $(\tilde{\R},p)$ where $\tilde{\R}$ is a $p$-completely smooth $\W(k)$-lift) are vector bundles after inverting $p$ by Corollary \ref{cor:pinv_VB}. Up to isogeny, we may then deduce objects on the source and target of $(-)_{\X_s^{\Prism}}$ come from $p$-torsionfree ones. As remarked in \cite{Fgaugelift}, the argument in \cite[Theorem 3.32]{Fgaugelift} shows these $p$-torsionfree $F$-crystals may be fully faithfully lifted via an analogous functor $\Pi_{\X_s}$. From this we can deduce that the restriction functor $(-)_{\X_s^{\Prism}}$ is an equivalence, by checking the unit of the adjunction $\eta: \mathcal{E}\to \Pi_{\X_s}(\mathcal{E}|_{\X_s^{\Prism}})$ for $\mathcal{E}\in \Coh(\X_{s}^{\Syn})$ is rationally an equivalence to deduce $\Pi_{\X_s}$ is essentially surjective. Indeed, working locally when $\X_s$ is affine we have a flat-local surjection $\Spf (\tilde{\R} \langle u,t \rangle/(ut-p))/\G_m \to (\X_s)^{\Syn}$. We know the kernel and cokernel of $\eta$ are $0$ restricted to $\X_s^{\Prism}$, hence when pulled back to the cover they are supported on $\V(t)\subset \V(p)$; it follows that the desired map is rationally an equivalence.

In Remark 2.5.8 of \cite{Fgauge}, explicit descriptions of $\Vect((\X/\O_K)^{\dR})$ and $\Vect((\X/\O_K)^{\dR,+})$ are given (and also a general method to describe sheaves, via the identification as a Rees stack). The first category consists of vector bundles with flat connection and nilpotent $p$-curvature modulo $p$, and the second category adds a Griffiths-transverse decreasing filtration by subbundles. We then have fully faithful functors
\[\Coh((\X/\O_K)^{\dR})[1/p] \to \Vect^{\nabla}(\X_\eta)\]
and similarly for the Hodge-filtered de Rham stack with $\Vect^{\nabla,+}$. We observe here that up to $p$-isogeny $\Coh((\X/\O_K)^{\dR})[1/p]\simeq \Vect((\X/\O_K)^{\dR})[1/p]$, so we do obtain vector bundles with flat connection after inverting $p$. For the Hodge-filtered de Rham stack, considering $\Coh((\X/\O_K)^{\dR,+})$ has the same effect on the underlying vector bundle with connection but allows the filtration to be a complete filtration with coherent associated graded; we just ask that $\nabla$ has degree $-1$. After these identifications we then obtain a fully faithful functor
\[\left(\Coh(\X_{s}^{\Syn}) \times_{\Coh((\X/\O_K)^{\dR})} \Coh((\X/\O_K)^{\dR,+})\right)[1/p] \to \mathrm{fIsoc}^{\varphi}(\X),\]
observing the left hand side is well-defined since the pullback from $\Coh(\X_{s}^{\Syn})[1/p]$ to $\Perf((\X/\O_K)^{\dR})[1/p]$ induced by the map of stacks in Corollary \ref{cor:Beilinson_fiber} indeed lands in the heart (the other map is clear, as it simply forgets the filtration). We may identify this pullback functor, after the first equivalence $\Coh(\X_{s}^{\Syn})[1/p] \to \mathrm{Isoc}^\varphi(\X_{s}/\W(k))$ and unwinding definitions, with passage to the underlying vector bundle with flat connection associated to the isocrystal.

It remains to check essential surjectivity. The essential image of the natural functor 
\[\Coh((\X/\O_K)^{\dR})[1/p] \to \Vect^{\nabla}(\X_\eta)\]
consists of vector bundles $\mathcal{E}_\eta$ with flat connection where the flat connection arises from a formal model $(\mathcal{E},\nabla)$ with a flat connection that has nilpotent $p$-curvature modulo $p$. By \cite[Lemma 2.2]{lutkebohmert} any coherent sheaf on $\X_\eta$ has a formal model, so $\Coh(\X_\eta)\simeq \Coh(\X)[1/p]$. In fact, if $(\mathcal{E}_\eta,\nabla)\in \Vect^{\nabla}(\X_\eta)$ is obtained from $\mathcal{E}'\in \mathrm{Isoc}^{\varphi}(\X_s/\W(k))$, then it has a model locally by \cite[Proposition 2.21]{ogus1984f} so we also have a formal model for the connection. The sheaf $\mathcal{E}_\eta$ must be locally free, as we may check on an affinoid that the algebraization gives a coherent sheaf with flat connection on a characteristic zero scheme which must be a vector bundle.

The desired claim now follows if we can additionally show there is some formal model for $(\mathcal{E}_\eta,\nabla)$ where $\nabla$ has nilpotent $p$-curvature modulo $p$ when $(\mathcal{E}_\eta,\nabla)$ comes from some $\mathcal{E}'\in \mathrm{Isoc}^{\varphi}(\X_s/\W(k))$. This is clear, since the formal model arises from a crystal and must then have a quasi-nilpotent connection, or equivalently nilpotent $p$-curvature modulo $p$. The essential image in the filtered case is derived similarly.
\end{proof}

Our next goal will be to describe the entire category $\Perf_{\mathrm{fIsoc}^{\varphi}}(\X)$ without reference to a stack.

\begin{definition}
Assume $\X=\Spf \R$ is affine over $\Spf \O_K$ and written as $\X_0 \widehat{\otimes}_{\W(k)} \O_K$, where $\X_0$ is a model with an \'{e}tale map
\[\X_0 \to \Spf \W(k) \langle \T_1^\pm, \ldots, \T_n^\pm \rangle\]
so that we have an induced basis of $\widehat{\Omega}^1_\R$. Let $\varphi$ be a global Frobenius lift on $\X_0$ (and thus inducing one on $(\X_0)_\eta$), available in our local situation via the torus chart.

A flat connection $\nabla=\sum_i \theta_i \mathrm{d}\T_i$ on a vector bundle $\mathcal{E}$ on $\X_\eta$ is called \textit{convergent} if for $e\in \Gamma(\X_\eta,\mathcal{E})$ and $0\le \rho < 1$ we have
\[\lim_{|\alpha| \to \infty} \left\|\frac{1}{\alpha!}\theta^\alpha(e)\right\| \rho^{|\alpha|}= 0.\]
Here $\alpha$ is a multi-index. The same definition works for $(\X_0)_\eta$.

Let $\pi$ be a uniformizer for $\O_K$. We set $\widehat{\D}_{\X_{\eta}} := (\widehat{\D}_\X)[1/\pi]$ where $\widehat{\D}_\X$ is the pullback of the sheaf of crystalline differential operators. Let $\Perf(\widehat{\D}_{\X_\eta})$ be the category of $\widehat{\D}_{\X_\eta}$-modules $\mathcal{E}$ whose underlying $\O_{\X_\eta}$-module is a perfect complex on $\X_\eta$. For the Hodge-filtered variant, let $\Perf^{+}(\widehat{\D}_{\X_\eta})$ denote the subcategory of perfect filtered $\widehat{\D}_{\X_\eta}$-modules where we endow $\widehat{\D}_{\X_{\eta}}$ with the order filtration. Here, perfectness means the underlying $\O_{\X_\eta}$-module is perfect, the associated graded is perfect, and the filtration is complete.

Finally, we let $\Perf^{\varphi}(\widehat{\D}_{(\X_0)_\eta})$ denote the category of $\widehat{\D}_{(\X_0)_\eta}$-modules $\mathcal{E}$ with their underlying $\O_{(\X_0)_\eta}$-module a perfect complex, along with a horizontal Frobenius $\varphi^* \mathcal{E}\simeq \mathcal{E}$ and the condition that each $\H^i(\mathcal{E})$ viewed as a vector bundle with connection has a convergent connection.
\label{def:conv_isoc}
\end{definition}

We can in fact describe the category $\Perf_{\mathrm{fIsoc}^{\varphi}}(\X)$ explicitly if we are allowed to work Zariski locally and pick a Frobenius lift. The same argument as \cite[Lemma 2.9]{Zpcrystalline} justifies that we can locally land in the situation of Definition \ref{def:conv_isoc}.

\begin{proposition}
Assume $\X_0$ is affine and \'{e}tale over $\Spf \W(k) \langle \T_i^\pm \rangle$, and choose a global Frobenius lift $\varphi$ on $\X_0$. Then
\[\Perf(\X_{s}^{\Syn})[1/p] \simeq \Perf^{\varphi}(\widehat{\D}_{(\X_0)_\eta}).\]
Let $\X = \X_0 \widehat{\otimes}_{\W(k)} \O_K$. Then there is also an equivalence
\[\Perf_{\mathrm{fIsoc}^{\varphi}}(\X) \simeq \Perf^{\varphi}(\widehat{\D}_{(\X_0)_\eta}) \times_{\Perf(\widehat{\D}_{\X_\eta})} \Perf^+(\widehat{\D}_{\X_\eta}).\]
\label{prop:explicit_isocrystal}
\end{proposition}

\begin{proof}
We consider the claim $\Perf(\X_{s}^{\Syn})[1/p] \simeq \Perf^{\varphi}(\widehat{\D}_{(\X_0)_\eta})$ first. The crystalline realization defines a functor
\[\Perf(\X_{s}^{\Syn})[1/p] \to \Perf(\X_{s,\cris},\O_{\cris})[1/p]^{\varphi=1}.\]
We will deduce this functor is fully faithful using the fiber sequence
\begin{display}
\R\Gamma(\X_{s}^{\Syn},\mathcal{E}) \ar{r} & \mathrm{Fil}^0_{\Nyg} \R\Gamma(\X_{s}^{\Prism},\mathcal{E}) \ar{r}{\varphi_{\mathcal{E}}-1} & \R\Gamma(\X_{s}^{\Prism},\mathcal{E}).
\end{display}
In our situation $\X_{s}$ is equipped with a lift $\X_0=\Spf \R_0$, which also has a global Frobenius lift due to the framing we are working in. This means that crystalline cohomology can be computed via the completed de Rham complex $\widehat{\Omega}_{\X_0}(\mathcal{E})$ of $\mathcal{E}$, and the Nygaard filtration on $\R\Gamma(\X_{s}^\cris,\mathcal{E}) \simeq \varphi^* \R\Gamma(\X_{s}^{\Prism},\mathcal{E})$ will have $k$th filtered piece given by
\[p^k \mathcal{E} \otimes \Omega^0_{\R_0} \to p^{k-1} \mathcal{E} \otimes \widehat{\Omega}^1_{\R_0} \to \cdots \to p \mathcal{E} \otimes \widehat{\Omega}^{k-1}_{\R_0} \to \mathcal{E} \otimes \widehat{\Omega}^k_{\R_0} \to \mathcal{E} \otimes \widehat{\Omega}^{k+1}_{\R_0} \to \cdots\]
where $\mathcal{E}$ is the corresponding object on $\X_0^{\dR} \simeq \X_{s}^\cris$. This description works when $\mathcal{E}$ is in the heart and $p$-torsionfree.

It follows that inverting $p$ we get the fiber sequence
\begin{display}
\R\Gamma(\X_{s}^{\Syn},\mathcal{E})[1/p] \ar{r} & \R\Gamma(\X_{s}^{\Prism},\mathcal{E})[1/p] \ar{r}{\varphi_{\mathcal{E}}-1} & \R\Gamma(\X_{s}^{\Prism},\mathcal{E})[1/p],
\end{display}
for $\mathcal{E}\in \Coh(\X_s^{\Syn})$. Here we use that crystalline cohomology is recovered as the Frobenius pullback of prismatic cohomology in positive characteristic so that
\[\R\Gamma_\cris(\X_{s},\mathcal{E})[1/p]^{\varphi_{\mathcal{E}}=1} \simeq \R\Gamma_{\Prism}(\X_{s},\mathcal{E})[1/p]^{\varphi_{\mathcal{E}}=1}.\]
Then we use the description of the Nygaard filtration; as mentioned in the proof of Proposition \ref{prop:isoc_justification} we may assume $\mathcal{E}$ is $p$-torsionfree as this is always true up to isogeny. The same argument as Proposition \ref{prop:O_K_equiv} allows us to deduce full faithfulness from this result, as well as essential surjectivity once we know full faithfulness as we have shown the claim on the heart in Proposition \ref{prop:isoc_justification}.

Now we can show
\[\Perf(\X_{s,\cris},\O_{\cris})[1/p]^{\varphi=1} \to \Perf^{\varphi}(\widehat{\D}_{(\X_0)_\eta})\]
is an equivalence; we can compute crystalline cohomology as the completed de Rham cohomology of $\X_0/\W(k)$. This then rationalizes to de Rham cohomology on $(\X_0)_\eta$, so we see that the functor is fully faithful by the same ``reduction to cohomology'' argument used in Proposition \ref{prop:O_K_equiv}. Essential surjectivity on the heart follows from the fact that $\mathrm{Isoc}^{\varphi}(\X_{s}/\W(k))$ is equivalent to the category of vector bundles $\mathcal{E}$ on $(\X_0)_\eta$ with a flat convergent connection $\nabla$ and an isomorphism $\varphi^* \mathcal{E} \simeq \mathcal{E}$ by \cite[Proposition 2.18]{ogus1984f}. Now that we know full faithfulness, since all $t$-structures we work with are bounded and nondegenerate we see essential surjectivity follows from compatibility of the functor with shifts and fibers, using \cite[Lemma 5.4.3]{Hauck} and the equivalence on the heart (similar to Proposition \ref{prop:O_K_equiv}). Finally, the functor $\Perf(\X_s^{\Syn})[1/p] \to \Perf(\X_{s,\cris},\O_{\cris})[1/p]^{\varphi=1}$ is an equivalence on the heart by the equivalence $\Coh(\X_s^{\Syn})[1/p]\simeq \mathrm{Isoc}^{\varphi}(\X_{s}/\W(k))$ of Proposition \ref{prop:isoc_justification}. This again extends to essential surjectivity by the same argument.

Now we turn to the second claim. To analyze $\Perf(\widehat{\D}_{\X_\eta})$ and $\Perf^+(\widehat{\D}_{\X_\eta})$ we will not actually need the local form of $\X$ (only that it is smooth). We have fully faithful functors $\Perf((\X/\O_K)^{\dR})[1/p]\to \Perf(\widehat{\D}_{\X_\eta})$ as well as the filtered variant
\[\Perf((\X/\O_K)^{\dR,+})[1/p]\to \Perf^+(\widehat{\D}_{\X_\eta})\]
with the unfiltered essential image consisting of objects $(\mathcal{E}_\eta,\nabla)$ admitting an integral model on $\X$ where each cohomology sheaf has a connection with nilpotent $p$-curvature modulo $p$, and the filtered essential image is similar, but with a filtration where $\nabla$ has filtered degree $-1$. To produce these functors, one way (among many) is to use a local presentation of the stack $\X$. As $\X$ is smooth, Zariski locally we may assume $\X$ is \'{e}tale over the formal scheme $\A^n_{\O_K} = \Spf \O_K \langle \T_1,\ldots,\T_n \rangle$, $(\X/\O_K)^{\dR,+}\simeq (\X \times \A^1/\G_m)/\V(\O(-1))^{\#, n}$ (the $\V(\O(-1))^{\#, n}$ torsor structure given by pulling back the case of $\A^n_{\O_K}$ along the \'{e}tale map). Using this description of the stack and filtered Koszul duality, quasicoherent sheaves on the stack can be interpreted as filtered $\D$-modules for the algebra of crystalline differential operators with the order filtration. We may use the proof of \cite[Lemma 6.7]{APC2} to translate $\Gamma^\ast (\E)$ coactions to continuous $\Sym^\ast(\E^\vee)$ actions. In particular we can apply this to $\QCoh([(\X \times \A^1/\G_m)/\V(\E)^{\#}])$ for $\V(\E)^{\#}$ acting on $\X \times \A^1/\G_m$ to see this category is equivalent to $\Mod_{\Sym^\ast(\E^\vee)}(\X \times \A^1/\G_m)$, the category of quasicoherent sheaves on $\X \times \A^1/\G_m$ equipped with a continuous action of $\widehat{\Sym^\ast}(\E^\vee)$ (completed with respect to the $\Sym^{\ge n}$ filtration) compatible with the action of $\widehat{\Sym^\ast(\E^\vee)}$ on $\O_{\X \times \A^1/\G_m}$. Rationalizing, we see sheaves on this stack are a full subcategory of filtered $\widehat{\D}_{\X_\eta}$-modules with the order filtration; we can check perfectness on the open point $\A^1/\G_m$ and the formal completion on the closed point, resulting in the definition of $\Perf^+(\widehat{\D}_{\X_\eta})$. Note that the filtration being complete with perfect associated graded encodes the second condition. Knowing the first equivalence, we deduce full faithfulness and may again test essential surjectivity in the second claim $\Perf_{\mathrm{fIsoc}^{\varphi}}(\X) \simeq \Perf^{\varphi}(\widehat{\D}_{(\X_0)_\eta}) \times_{\Perf(\widehat{\D}_{\X_\eta})} \Perf^+(\widehat{\D}_{\X_\eta})$ by reducing to the equivalence on the heart shown in Proposition \ref{prop:isoc_justification}.
\end{proof}

On the heart, the functor $\mathrm{Beil}$ then identifies with the functor $\D_\cris: \Loc^{\cris}_{\Q_p}(\X_\eta) \to \mathrm{fIsoc}^{\varphi}(\X)$. This will allow us to characterize the essential image as objects whose cohomology sheaves are admissible.

\begin{proposition}
Let $\X/\Spf \O_K$ be smooth and quasicompact. The diagram
\begin{display}
\Coh(\X^{\Syn})[1/p] \ar{d}{\T_\et} \ar{r}{\mathrm{Beil}} & \left(\Coh(\X_{s}^{\Syn}) \times_{\Coh((\X/\O_K)^{\dR})} \Coh((\X/\O_K)^{\dR,+})\right)[1/p]  \ar{d}{\sim} \\
\Loc_{\Q_p}^\cris(\X_\eta) \ar{r}{\D_\cris} & \mathrm{fIsoc}^{\varphi}(\X)
\end{display}
commutes.\footnote{Recall we define $\D_{\cris}$ using \cite[Corollary 4.6]{Zpcrystalline}, using the equivalence $\Loc_{\Q_p}^\cris(\X_\eta) \simeq \Vect^{\varphi,\mathrm{an}}(\X_{\Prism},\O_{\Prism})[1/p]$ and then applying their fully faithful functor to filtered $F$-isocrystals (which we call $\tilde{\T}_\cris$ here).} Recall that the right vertical arrow is an equivalence by Proposition \ref{prop:isoc_justification}.
\label{prop:Dcris_compat}
\end{proposition}

\begin{proof}
It suffices to check the claim locally, so we can even assume that $\X=\X_0 \widehat{\otimes}_{\W(k)} \O_K$ where $\X_0$ is \'{e}tale over $\W(k) \langle \T_1^\pm, \ldots, \T_n^{\pm} \rangle$ (this reduction follows from \cite[Lemma 2.9]{Zpcrystalline}).

It suffices to show that the following diagram is commutative:
\begin{display}
\Coh(\X^{\Syn})[1/p] \ar{d}{(-)|_{\X^{\Prism}}} \ar{r}{\mathrm{Beil}} & \left(\Coh(\X_{s}^{\Syn}) \times_{\Coh((\X/\O_K)^{\dR})} \Coh((\X/\O_K)^{\dR,+})\right)[1/p]  \ar{d}{\sim} \\
\Coh^{\varphi}(\X_{\Prism},\O_{\Prism})[1/p] \ar{r}{\tilde{\T}_\cris} & \mathrm{fIsoc}^\varphi(\X)
\end{display}
Here, $\tilde{\T}_{\cris}$ is the fully faithful functor to filtered $F$-isocrystals constructed in \cite[Corollary 4.6]{Zpcrystalline}. The commutativity of the square is easy to see once we recall $\tilde{\T}_{\cris}$ is given by taking the restriction functor
\[\T_\cris: \Coh^{\varphi}(\X_{\Prism},\O_{\Prism})[1/p] \to \Coh^{\varphi}(\X_{s,\Prism},\O_{\Prism})[1/p] \simeq \mathrm{Isoc}^{\varphi}(\X_{s}/\W(k))\]
and using \cite[Proposition 2.38]{Zpcrystalline} to equip this with a filtration to obtain $\tilde{\T}_\cris$ (one can pick a reflexive representative). We have already shown in Proposition \ref{prop:isoc_justification} that $\T_\cris$ fits in the diagram
\begin{display}
\Coh(\X^{\Syn})[1/p] \ar{d}{(-)|_{\X^{\Prism}}} \ar{r}{(-)|_{\X_{s}^{\Syn}}} & \Coh(\X_{s}^{\Syn})[1/p] \ar{d}{(-)|_{\X_{s}^{\Prism}}} \\
\Coh^{\varphi}(\X_{\Prism},\O_{\Prism})[1/p] \ar{r}{\T_{\cris}} & \mathrm{Isoc}^{\varphi}(\X_{s}/\W(k))
\end{display}
so we need only verify that we obtain the same filtration from both constructions. By \cite[Proposition 2.38]{Zpcrystalline}, what we need to show is that given $\mathcal{E}\in \Coh(\X^{\Syn})$ we have a canonical identification
\[\D_{\dR}(\T_\et(\mathcal{E})[1/p]) \simeq \mathcal{E}|_{(\X/\O_K)^{\dR,+}}[1/p]\]
via the fully faithful functor $\Coh((\X/\O_K)^{\dR,+})[1/p] \to \Vect^{\nabla,+}(\X_\eta)$ allowing us to regard both sides as filtered vector bundles on $\X_\eta$ with a flat Griffiths-transverse connection. The argument in \cite{Hauck} Remark 7.2.9 shows this is the case (the same reasoning works relative to $\O_K$ when $\X/\Spf \O_K$ is smooth; note again we may choose a reflexive representative of the isogeny class). 
\end{proof}

\begin{theorem}
Assume $\X$ is smooth proper over $\Spf \O_K$. Then $\mathrm{Beil}$ induces an equivalence
\[\Perf(\X^{\Syn})[1/p]\simeq \Perf^{\mathrm{adm}}_{\mathrm{fIsoc}^{\varphi}}(\X).\]
This equivalence is $t$-exact and symmetric monoidal.
\label{thm:rational_equiv}
\end{theorem}

\begin{proof}
To deduce the claimed $t$-exactness, recall we defined the $t$-structure on the target to come from the individual $t$-structures on $\Perf((\X/\O_K)^{\dR,+})$, $\Perf(\X_s^{\Syn})$ and $\Perf((\X/\O_K)^{\dR})$; the pullbacks to each of these are $t$-exact rationally. Of these the only non-obvious one is the pullback to $\Perf(\X_s^{\Syn})$, which can be seen by showing the derived pullback of objects in $\Coh(\X^{\Syn})[1/p]$ lands in $\Coh(\X_s^{\Syn})[1/p]$; this follows since we have already seen that up to isogeny we may choose a reflexive representative. By descent we may work Zariski locally to assume $\X=\Spf \R$, where we have a cover $\mathrm{Rees}_{(\E(u_0))^\bullet}\R_0\llbracket u_0 \rrbracket$ of $\X^{\Syn}$ induced by a corresponding Breuil-Kisin prism with a compatible cover $\mathrm{Rees}_{p^\bullet}\R_0 \to \X_s^{\Syn}$ given by the $u_0=0$ locus. By this, we mean there is a commutative square
\begin{display}
\mathrm{Rees}_{p^\bullet} \R_0 \ar{r} \ar{d} & \mathrm{Rees}_{(\E(u_0))^\bullet} \R_0\llbracket u_0 \rrbracket \ar{d} \\
\X_s^{\Syn} \ar{r} & \X^{\Syn}.
\end{display}
The existence of the reflexive (hence $u_0$-torsionfree) representative shows the derived pullback to the cover $\mathrm{Rees}_{p^\bullet} \R_0$ of $\X_{s}^{\Syn}$ is coherent, hence the pullback lands in $\Coh(\X_s^{\Syn})[1/p]$.

By Corollary \ref{cor:Beilinson_fiber}, we know $\mathrm{Beil}$ is fully faithful so we need only show that the essential image is $\Perf^{\mathrm{adm}}_{\mathrm{fIsoc}^{\varphi}}(\X)$. Each $\H^i(\mathrm{Beil}(\mathcal{E}))$ is an admissible filtered $F$-isocrystal by Proposition \ref{prop:Dcris_compat} and using $t$-exactness of $\mathrm{Beil}$, we see the essential image of $\mathrm{Beil}$ is a full subcategory of $\Perf^{\mathrm{adm}}_{\mathrm{fIsoc}^{\varphi}}(\X)$.

The functor $\mathrm{Beil}$ is symmetric monoidal essentially by construction, as the functor on each component of $\Perf_{\mathrm{fIsoc}^{\varphi}}(\X)$ is induced by a pullback. To show every object of $\Perf^{\mathrm{adm}}_{\mathrm{fIsoc}^{\varphi}}(\X)$ is in the essential image, since we already know $\mathrm{Beil}$ is fully faithful in the derived sense, $t$-exact, and an equivalence on the heart, \cite[Lemma 5.4.3]{Hauck} can be used to deduce essential surjectivity as the $t$-structures are bounded and nondegenerate.
\end{proof}

\nocite{Stacks}
\bibliographystyle{amsalpha}
\bibliography{citations.bib}

\end{document}